\documentclass[10pt,a4paper,reqno]{amsart} 

\usepackage[english]{babel}
\usepackage[utf8]{inputenc}
\usepackage{hyperref}
\usepackage{latexsym}
\usepackage{verbatim}
\usepackage{rotating}
\usepackage{color,bbm}
\usepackage{amssymb,amsfonts,amsmath}
\usepackage{stmaryrd}

\usepackage{caption}
\usepackage{amsaddr}

\usepackage{enumitem}
\usepackage[T1]{fontenc}

\usepackage{psfrag}

\newcommand{\Agf}{{\sf A}}
\newcommand{\Bgf}{{\sf B}}
\newcommand{\Cgf}{{\sf C}}

\newcommand{\E}{{E}}

\newcommand{\Fgf}{{\sf F}}
\newcommand{\Ggf}{{\sf G}}
\newcommand{\Hgf}{{\sf H}}
\newcommand{\Igf}{{\sf I}}

\newcommand{\Lgf}{{\sf L}}
\newcommand{\Mgf}{{\sf M}}
\newcommand{\Pgf}{{\sf P}}
\newcommand{\Qgf}{{\sf Q}}
\newcommand{\Rgf}{{\sf R}}
\newcommand{\Sgf}{{\sf S}}
\newcommand{\Tgf}{{\sf T}}

\newcommand{\Ugf}{{\sf U}}
\newcommand{\Vgf}{{\sf V}}
\newcommand{\Wgf}{{\sf W}}

\newcommand{\cC}{{\mathcal C}}
\newcommand{\cQ}{{\mathcal Q}}

\newcommand{\Dnn}{{\mathcal D}}

\newcommand{\notea}[1]{\textcolor{red}{#1}}

\newtheorem{Theorem}{Theorem}[section]
\newtheorem{Lemma}[Theorem]{Lemma}
\newtheorem{Proposition}[Theorem]{Proposition}
\newtheorem{Definition}[Theorem]{Definition}
\newtheorem{Corollary}[Theorem]{Corollary}

\newtheorem{Question}{Question}

\newcommand{\beq}{\begin{equation}}
\newcommand{\eeq}{\end{equation}}
\def\emm#1,{{\em #1}}

\catcode`\@=11
\def\section{\@startsection{section}{1}%
 \z@{.7\linespacing\@plus\linespacing}{.5\linespacing}%
 {\normalfont\bfseries\scshape\centering}}

\def\subsection{\@startsection{subsection}{2}%
  \z@{.5\linespacing\@plus\linespacing}{.5\linespacing}%
  {\normalfont\bfseries\scshape}}

\def\subsubsection{\@startsection{subsubsection}{3}%
 \z@{.5\linespacing\@plus\linespacing}{-.5em}
 {\normalfont\bfseries}}
\catcode`\@=12

\renewcommand{\th}{\vartheta}

\newcommand{\om}{\omega} 
\renewcommand{\epsilon}{\varepsilon}

\def\TwoDstepset#1#2#3#4#5#6#7#8{%
	\begin{picture}(20,20)(-10,-10)
	\put(0,0){\ifx1#1\thicklines\let\x\vector\x(-1,-1){10}\fi}
	\put(0,0){\ifx1#2\thicklines\let\x\vector\x(-1,0){10}\fi}
	\put(0,0){\ifx1#3\thicklines\let\x\vector\x(-1,1){10}\fi}
	\put(0,0){\ifx1#4\thicklines\let\x\vector\x(0,-1){10}\fi}
	\put(0,0){\ifx1#5\thicklines\let\x\vector\x(0,1){10}\fi}
	\put(0,0){\ifx1#6\thicklines\let\x\vector\x(1,-1){10}\fi}
	\put(0,0){\ifx1#7\thicklines\let\x\vector\x(1,0){10}\fi}
	\put(0,0){\ifx1#8\thicklines\let\x\vector\x(1,1){10}\fi}
	\end{picture}}

\graphicspath{{Figures/}}

\begin{document}
\title[]
{Enumeration of three quadrant walks with small steps and walks on other M-quadrant cones}
\subjclass[2010]{}
\author[]{Andrew Elvey Price}
\address[]{CNRS, Institut Denis Poisson, Universit\'e de Tours, France\\
}
\email{andrew.elvey@univ-tours.fr}

\maketitle

\begin{abstract}
We address the enumeration of walks with small steps confined to a two-dimensional cone, for example the quarter plane, three-quarter plane or the slit plane. In the quarter plane case, the solutions for unweighted step-sets are already well understood, in the sense that it is known precisely for which cases the generating function is algebraic, D-finite or D-algebraic, and exact integral expressions are known in all cases. We derive similar results in a much more general setting: we enumerate walks on an $M$-quadrant cone for any positive integer $M$, with weighted steps starting at any point. The main breakthrough in this work is the derivation of an analytic functional equation which characterises the generating function of these walks, which is analogous to one now used widely for quarter-plane walks. In the case $M=3$, which corresponds to walks avoiding a quadrant, we provide exact integral-expression solutions for walks with weighted small steps which determine the generating function $\Cgf(x,y;t)$ counting these walks. Moreover, for each step-set and starting point of the walk we determine whether the generating function $\Cgf(x,y;t)$ is algebraic, D-finite or D-algebraic as a function of $x$ and $y$. In fact we provide results of this type for any $M$-quadrant cone, showing that this nature is the same for any odd $M$. For $M$ even we find that the generating functions counting these walks are D-finite in $x$ and $y$, and algebraic if and only if the starting point of the walk is on the same axis as the boundaries of the cone.
\end{abstract}

\section{Introduction}
The systematic study of walks with small steps in the quarter plane was initiated by Bousquet-M\'elou and Mishna in 2010 \cite{bousquet2010walks}, and since then there has been great progress on the model \cite{bostan2010complete,fayolle2010holonomy,raschel2012counting,mishna2009two,
melczer2014singularity,
bernardi2017counting,kurkova2012functions,dreyfus2018nature,dreyfus2020walks}. The model is defined as follows: given a step set $S\subset\{-1,0,1\}^{2}\setminus\{(0,0)\}$, determine the generating function
\[\Qgf(x,y;t)=\sum_{n=0}^{\infty}\sum_{i,j\geq1}q(i,j;n)t^{n}x^{i}y^{j},\]
where $q(i,j;n)$ is the number of walks of length $n$, starting at $(1,1)$, and ending at $(i,j)$ using steps in $S$ and staying in the strictly positive quadrant.\footnote{In most of the literature, walks start at $(0,0)$ and stay in the non-strictly positive quadrant, for which the resulting generating function is $\frac{1}{xy}\Qgf(x,y;t)$.}

This article follows a subsequent line of research considering a natural extension of this work, namely the enumeration of walks in the three-quadrant cone 
\[\mathcal{C}=\{(i,j):i>0\text{ or }j>0\},\]
shown in Figure \ref{fig:three-quadrant_walk_example}. That is determining the generating function
\[\Cgf(x,y;t)=\sum_{n=0}^{\infty}\sum_{(i,j)\in \mathcal{C}}c(i,j;n)t^{n}x^{i}y^{j},\]
where $c(i,j;n)$ is the number of walks of length $n$, starting at $(1,1)$, and ending at $(i,j)$ using steps in $S$ and staying in $\mathcal{C}$. As we will discuss in the next subsection, several articles have considered this generating function for specific step-sets $S$ \cite{bousquet2016square,raschel2018walks,mustapha2019non,
dreyfus2020nature,bousquet_wallner,bousquet21+}, see Subsection \ref{subsec:past_research3} for more detail. The quadrant and the three-quadrant cone each play an important role in the understanding of walks in a cone $\mathcal{C}\subset\mathbb{C}$ as a typical {\em convex} cone $\mathcal{C}$ can be transformed into a quadrant via some linear transformation, while a typical {\em non-convex} cone can be similarly transformed into a three quadrant cone.

The starting point of the study of walks in the quarter plane is the functional equation
\[K(x,y;t)\Qgf(x,y;t)=xy-\Fgf_{1}(t)-\Agf_{1}(x;t)-\Bgf_{1}(y;t),\]
where $K(x,y;t)$ is known as the {\em kernel} of the model, and $A_{1}$, $B_{1}$ and $F_{1}$ are to be determined. Methods to solve such equations via reduction to boundary value problems are now somewhat standard \cite{fayolle1999random,raschel2012counting,bernardi2017counting}, the intitial idea being to consider the curve $E_{t}$ on which $K(x,y;t)=0$, as the left hand side of the equation is $0$ on this curve, as long as $\Qgf(x,y;t)$ converges. As we will discuss in Section \ref{sec:functional_equations}, one can write a similar equation characterising the generating function $\Cgf(x,y;t)$ of walks in the 3-quadrant cone:
\[K(x,y;t)\Cgf(x,y;t)=xy-\Fgf(t)-\Agf\left(\frac{1}{x};t\right)-\Bgf\left(\frac{1}{y};t\right).\]
Given the similarity, one may expect the same methods to apply, however for each method applied to walks in the quadrant, a major stumbling block has appeared. For the analytic method of \cite{fayolle1979two,fayolle1999random,raschel2012counting,bernardi2017counting}, the problem has been that the generating function $\Cgf(x,y;t)$ does not converge on a substantial section of $E_{t}$. To get around this problem, previous work on walks in the three-quadrant cone has used symmetries in certain models to relate them to quadrant models prior to further analysis, a method that only applies to a minority of models. The main breakthrough of this work is a way to bypass this issue more generally - intuitively we proceed by cutting the three-quadrant cone into its three quadrants prior to moving to the analytic world, then we reglue the quadrants in the analytic world to derive an analytic functional equation resembling that deduced for walks in the quadrant.

\begin{figure}[ht]
\centering
   \includegraphics[scale=1]{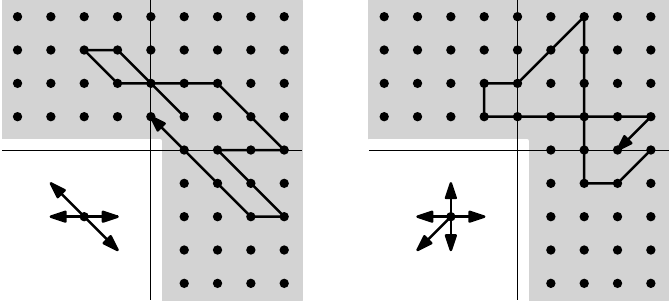}
   \caption{Walks in the three-quadrant cone $\mathcal{C}$. Left: A walk starting at the standard starting point $(1,1)$ using step set \#12 from Table \ref{table:previously_solved}. Right: A walk starting at $(4,0)$ using step set \#10. Even though the generating function for walks with step set \#10 starting at $(1,1)$ is D-transcendental, we will show that the generating function for walks starting at $(4,0)$ is D-algebraic in $x$.}
   \label{fig:three-quadrant_walk_example}
\end{figure}

In the latter half of this article we present a general approach to counting walks in cones, by considering walks in an $M$-quadrant cone, using the same ideas as for $M=3$. Walks confined to such a cone can be equivalently described as walks whose intermediate winding angles around the origin are restricted to some interval (see \cite[Section 3.3]{budd2017winding}). In each case $M=1,2,3,4$ the enumeration of walks in an $M$-quadrant cone has been previously been studied, so our analysis also serves to uniformize the approach to these models.

\subsection{Past research on walks in $M$-quadrant cones}\label{subsec:past_research3}
Much of the past research on walks in cones has used some unweighted step-set $S\subset \{-1,0,1\}^{2}\setminus\{(0,0)\}$. A priori, there are 256 distinct step sets $S$. In the quarter-plane case, after removing duplicates and cases that are equivalent to half-plane models, Bousquet-M\'elou and Mishna identified 79 non-trivial and combinatorially distinct models \cite{bousquet2010walks}. The study of these models is now in some sense complete as it is known for each $S$ where the generating function fits into the hierarchy
\[\text{Algebraic}\subset\text{D-finite}\subset\text{D-Algebraic}.\]
Recall that a generating function is called Algebraic with respect to a certain variable if it is related to that variable by a non-trivial polynomial equation with coefficients only depending on the other variables, and it is called D-finite (resp. D-algebraic) if it satisfies a linear (resp. polynomial) differential equation with respect to that variable whose coefficients are polynomial in that variable.
Of the 79 models proposed by Bousquet-M\'elou and Mishna, 4 models admit an algebraic generating function (with respect to both variables) \cite{bousquet2010walks,bostan2010complete}, 19 further models admit a D-finite generating function \cite{bousquet2010walks,fayolle2010holonomy}, 9 further models admit a D-algebraic generating function \cite{bernardi2017counting,kurkova2012functions} and the remaining 47 models admit a generating function which is not D-algebraic \cite{dreyfus2018nature,dreyfus2020walks}, in which case we say that the generating function is D-transcendental. Moreover, in the 74 cases known as {\em non-singular}, an exact integral expression is known for the generating function \cite{raschel2012counting}, while other exact expressions are known in the $5$ singular cases \cite{mishna2009two,melczer2014singularity}.

\begin{table}[h!]
\begin{tabular}[t]{|c||c|c|c|c|c|c|c|}
\hline 
Model&\#1&\#2&\#3&\#4&\#5&\#6&\#7\\ 
&\TwoDstepset01011010 &\TwoDstepset10100101 &\TwoDstepset11111111 &\TwoDstepset01010001 & \TwoDstepset10001010 &\TwoDstepset11011011 &\TwoDstepset01011011\\
\hline
Nature of $\Cgf(x,y;t)$&DF \cite{bousquet2016square}&DF \cite{bousquet2016square}&DF \cite{bousquet_wallner}& alg. \cite{bousquet21+}& alg. \cite{bousquet21+}& alg. \cite{bousquet21+}& D-alg. \cite{dreyfus2020walks}\\
\hline
\end{tabular}

\begin{tabular}{|c||c|c|c|c|c|c|c|}
\hline 
Model&\#8&\#9&\#10&\#11&\#12&\#13&\#14\\ 
&\TwoDstepset11010001 &\TwoDstepset10001011 &\TwoDstepset11011010 &\TwoDstepset11000011 &\TwoDstepset01100110 &\TwoDstepset01001100 &\TwoDstepset01111110\\
\hline
Nature of $\Cgf(x,y;t)$&D-trans \cite{dreyfus2020walks}&D-trans \cite{dreyfus2020walks}&D-trans \cite{dreyfus2020walks}&\cite{budd2017winding}&\cite{budd2017winding}&\cite{elveyprice_winding_FPSAC}&\cite{elveyprice_winding_FPSAC}\\
\hline
\end{tabular} 
\medskip
\caption{Previously solved models on the 3-quadrant cone in some cases. Models \#1-\#10 solved for walks starting at $(1,1)$. Models \#11-\#14 only solved for certain specified start and end points.}
\label{table:previously_solved}
\end{table}
For walks in the three-quadrant cone, the generating function $\Cgf(x,y;t)$ is algebraic if $S$ is singular, as then the model can be written in terms of half-plane models (see \cite[Section 2.2]{bousquet_wallner}), so only the 74 non-singular cases remain. The study of walks in the three-quadrant cone was initiated by Bousquet-M\'elou in \cite{bousquet2016square}, where she enumerated walks with two step sets (\#1 and \#2 in Table \ref{table:previously_solved}), showing that the associated generating function is D-finite.  Raschel and Trotignon \cite{raschel2018walks} then found a transformation relating walks with any of the step sets \#4-\#10 to walks in the quarter plane - allowing them to deduce integral expression solutions in these cases. Building on this work Dreyfus and Trotignon \cite{dreyfus2020nature} showed that the last three of these step sets (\#8 - \#10) admit non-D-algebraic generating functions, while step set \#7 admits a D-algebraic generating function. Subsequently, using algebraic methods, Bousquet-M\'elou and Wallner \cite{bousquet_wallner} adapted the method of \cite{bousquet2016square} to enumerate walks with step set \#3, again showing that these admit a D-finite generating function. Finally, in a recent article \cite{bousquet21+}, Bousquet-M\'elou used invariants to show that three cases known as the Kreweras cases (\#4-\#6) admit algebraic generating functions.  Together these classify 10 models into the complexity hierarchy (algebraic, D-finite, D-algebraic). Moreover, in \cite{mustapha2019non}, Mustapha used asymptotic properties to show that the series $[x^{0}y^{0}]\Cgf(x,y;t)$ counting excursions starting and ending at $(0,0)$ is not D-finite (in $t$) for any of the $51$ non-singular step sets $S$ admitting an infinite group; that is, precisely the models which are not D-finite in the quarter plane. Remarkably, in all of the cases (\#1-\#10) shown in Table \ref{table:previously_solved}, the generating function has the same nature as in the quarter plane. Dreyfus and Trotignon were the first to conjecture that this holds more generally - that the nature is the same for any of the 74 non-singular step-sets $S$ \cite{dreyfus2020walks}.

Walks with specified endpoints have been enumerated for 2 further models (\#11-\#12) by Budd \cite{budd2017winding} and 2 further models by the current author (\#13-\#14) \cite{elveyprice_winding_FPSAC}, both of these results by relating the model to walks counted by winding angle. In fact, these results apply to the $M$-quadrant cone for any integer $M>0$. We note that for models (\#10-\#13), the (single variable) generating functions considered in \cite{budd2017winding,elveyprice_winding_FPSAC} were found to be algebraic in $t$, which coincides with the general quarter-plane case for model \#10, but for models \#11-\#13, the generating function $\Qgf(x,y;t)$ is D-finite but not algebraic \cite{bousquet2010walks}. This is perhaps explained by the fact that the walks considered in \cite{budd2017winding,elveyprice_winding_FPSAC} for models \#11-\#13 start at $(1,0)$ rather than the standard starting point $(1,1)$.

Finally, we mention two other $M$-quadrant cones. The $2$-quadrant cone is simply the half plane, so walks in this space are well understood. The $4$-quadrant cone is also known as the slit plane, on which walks were studied in  \cite{bousquet2001walks}, \cite{bousquet2002walks} and \cite{rubey2004transcendence}. In particular, Bousquet-M\'elou showed that the full generating function for walks in the slit-plane starting at some point on the $x$-axis is algebraic for any finite step-set $S\subset \mathbb{Z}\times\{-1,0,1\}\setminus\{(0,0)\}$.  

\subsection{Main Results}
Thoughout this article, we fix a step-set $S\subset\{-1,0,1\}^{2}\setminus\{(0,0)\}$, a weight $w_{s}>0$ for each $s\in S$ and a starting point $(p,q)$ with $p>0$, $q\geq0$. We assume that $S$ is a non-singular step-set, that is, for any line $\ell$ through the origin, at least one element of $S$ lies on each side of $\ell$. We then consider the enumeration of walks in some cone $\mathcal{C}$ using this weighted step-set and starting point. A fundamental object of the study of walks with such a fixed weighted step-set is the Kernel $K(x,y;t)$ given by
\[K(x,y;t):=-1+t\sum_{(\alpha,\beta)\in S}w_{(\alpha,\beta)}x^{\alpha}y^{\beta}.\]  

Before we start to consider the problem of counting walks in any specific cone, we define a uniformisation $(X(z),Y(z))$ of the Kernel curve $E_{t}=\{(x,y):K(x,y;t)=0\}$. Explicit definitions of $X(z)$ and $Y(z)$ are given in Appendix \ref{ap:param_thm}, while Lemma \ref{lem:param} serves as an implicit definition. These functions, and Lemmas \ref{lem:param} and \ref{lem:Omega}, which describe their properties, are used extensively throughout the article.

Sections \ref{sec:functional_equations}, \ref{sec:integral_expression}, \ref{sec:axis_start} and \ref{sec:nature} concern the enumeration of walks in the 3-quadrant cone, that is, we consider the generating $\Cgf(x,y;t)$ counting weighted walks starting at $(p,q)$ and staying entirely inside the 3-quadrant cone. Sections \ref{sec:more_cones} and \ref{sec:bc_solutions} concern the enumeration of walks in the more general $M$-quadrant cones. The dependencies between the sections of this article are shown in Figure \ref{fig:Section_dependencies}.
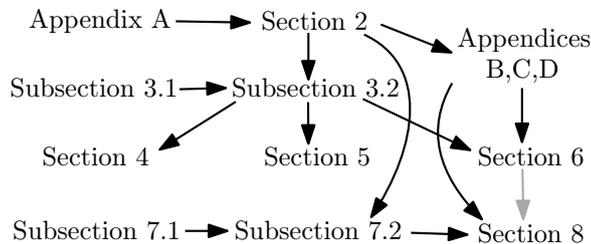
\begin{figure}[ht]
\centering
\begin{picture}(229,100)
   \put(0,0){\includegraphics[scale=1]{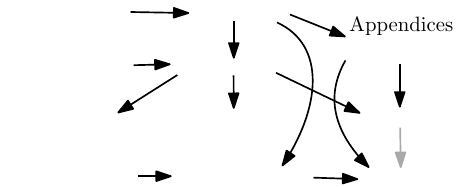}}
   
   \put(2,2){Subsection \ref{subsec:Mon4plane}}
   \put(13,30){Section \ref{sec:integral_expression}}
   \put(1,56){Subsection \ref{sec:formal_functional_equations}}
   \put(8,82){Appendix \ref{ap:param_thm}}
   
   \put(85,2){Subsection \ref{subsec:bc_analytic_functional_equations}}
   \put(96,30){Section \ref{sec:axis_start}}
   \put(84,56){Subsection \ref{sec:functions_analytic}}
   \put(95,81){Section \ref{sec:param}}
   
   \put(176,1){Section \ref{sec:bc_solutions}}
   \put(176,30){Section \ref{sec:nature}}
   \put(179,63){\ref{ap:nature_analytic},\ref{ap:group},\ref{ap:D-trans}}
      \end{picture}
   \caption{Dependencies between Sections. Section \ref{sec:bc_solutions} is largely independent of Section \ref{sec:nature}, only using results from that section which do not rely on Section \ref{sec:functional_equations}.}
   \label{fig:Section_dependencies}
\end{figure}

Our main results concerning walks in the three-quadrant cone are listed below:
\begin{itemize}
\item Theorem \ref{thm:PH_PV_characterisation}: Characterisation of $\Cgf(x,y;t)$ using analytic functional equation, 
\item Theorem \ref{thm:integral_expression}: Explicit (integral expression) form for $\Cgf(x,y;t)$,
\item Theorem \ref{thm:P-xaxis}: Explicit expression for walks starting on $x$-axis ($q=0$),
\item Nature of $\Cgf(x,y;t)$ as function of $x$ and $y$ (section \ref{sec:nature}):
\begin{itemize}
\item Theorem \ref{thm:D-finite_general_t}: $\Cgf(x,y;t)$ D-finite $\iff$ group is finite,
\item Theorem \ref{thm:algebraic_general_t}: $\Cgf(x,y;t)$ algebraic $\iff$ group is finite and orbit sum is $0$,
\item Theorem \ref{thm:D-algebraic_general_t}: if group is infinite, $\Cgf(x,y;t)$ D-algebraic $\iff$ model decouples.
\end{itemize}
\end{itemize}
Theorem \ref{thm:PH_PV_characterisation} is in some sense the main result of this article: in it we give a simple analytic functional equation characterising two functions $A(z)$, $B(z)$ and a value $F$ (all of which depend on $t$), which in turn determines $\Cgf(x,y;t)$ via
\[K(x,y;t)\Cgf(x,y;t)=-x^py^q+A(Y^{-1}(y))+B(X^{-1}(x))+F,\]
which we make more precise in Section \ref{sec:functional_equations}. This analytic characterisation of $\Cgf(x,y;t)$ is analogous to a characterisation found by Raschel \cite{raschel2012counting} of $\Qgf(x,y;t)$, which counts walks in the quadrant. Our method is also a direct generalisation of that of Raschel, which in turn is based on a method of Fayolle, Iasnogorodski and Malyshev \cite{fayolle1979two,fayolle1999random}, which they used in a probabilistic context.

In Sections \ref{sec:integral_expression}, \ref{sec:axis_start} and \ref{sec:nature} we use Theorem \ref{thm:PH_PV_characterisation} to derive several independent results listed above, in part by adapting methods that have been applied to a similar analytic characterisation for walks in the walks in the quadrant. This does not exhaust all possible utility of Theorem \ref{thm:PH_PV_characterisation}, we have simply limited ourselves to some relatively well-defined and accessible consequences.

The first such result, Theorem \ref{thm:integral_expression}, is an explicit integral expression for the functions $A(z)$ and $B(z)$, analogous to one found by Raschel for the enumeration of walks in the quadrant \cite{raschel2012counting}.

In Section \ref{sec:axis_start}, we give an integral-free expression for $A(z)$ in the case that the walk starts on an axis. In Corollary \ref{cor:combi}, we derive a remarkable equinumerosity result relating walks with different prescribed starting and ending points for which we have no combinatorial explanation.

\begin{figure}[ht]
\centering
\begin{picture}(350,200)
   \put(0,0){\includegraphics[scale=1]{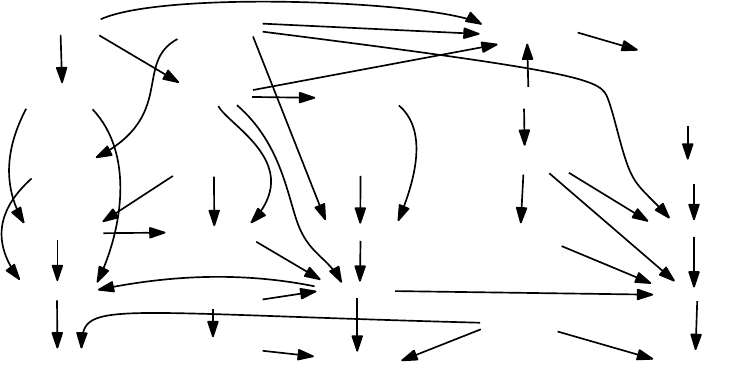}}
   \put(7,2){Thm \ref{thm:algebraic_general_t}}
   \put(7,33){Thm \ref{thm:algebraic_fixed_t}}
   \put(7,62){Thm \ref{thm:finite_group_alg_thm}}
   \put(7,93){Prop \ref{prop:orbit_sum0isdecoupling}}
   \put(0,128){Thm \ref{thm:algebraic_fixed_t} (i)-(vi)}
   \put(7,160){Prop \ref{prop:algebraicofXorY}}
   
   \put(83,9){Prop \ref{prop:finite_group_general_t}}
   \put(83,30){Prop \ref{prop:finite_group_fixed_t}}
   \put(80,62){Thm \ref{thm:finite_group_case_thm}}
   \put(80,93){Lem \ref{lem:finite_group}}
   \put(80,128){Prop \ref{prop:D-finiteofXorY}}
   \put(80,162){Prop \ref{prop:rationalofXorY}}
   
   \put(153,2){Thm \ref{thm:D-finite_general_t}}
   \put(153,33){Thm \ref{thm:D-finite_fixed_t}}
   \put(153,62){Thm \ref{thm:inf_group_non-D-finite}}
   \put(153,93){Lem \ref{lem:not-rational}}
   \put(153,128){Lem \ref{lem:D-fini_z}}
   
   \put(233,16){Lem \ref{lem:function_and_series_complexity}}
   \put(230,62){Thm \ref{thm:D-alg_proof}}
   \put(220,97){Thm \ref{thm:D-algebraic_fixed_t} (i)-(vi)}
   \put(233,128){Prop \ref{prop:D-algebraicofXorY}}
   \put(233,160){Prop \ref{prop:PandXP}}
   
   \put(315,2){Thm \ref{thm:D-algebraic_general_t}}
   \put(315,33){Thm \ref{thm:D-algebraic_fixed_t}}
   \put(315,62){Thm \ref{thm:non-D-alg_proof}}
   \put(315,93){Cor \ref{cor:D-trans}}
   \put(310,118){Prop \ref{prop:D-trans_poles}}
   \put(308,150){Prop \ref{prop:XP_closure_properties}}
   
   \end{picture}
   \caption{Dependencies between results in Section \ref{sec:nature} and Appendices \ref{ap:nature_analytic}, \ref{ap:group} and \ref{ap:D-trans}. The results of Section \ref{sec:nature} rely additionally on Sections \ref{sec:param} and \ref{sec:functional_equations}, while results in the appendices \ref{ap:nature_analytic}, \ref{ap:group} and \ref{ap:D-trans} rely on Section \ref{sec:param}.}
   \label{fig:Theorem_dependencies_sec6}
\end{figure}
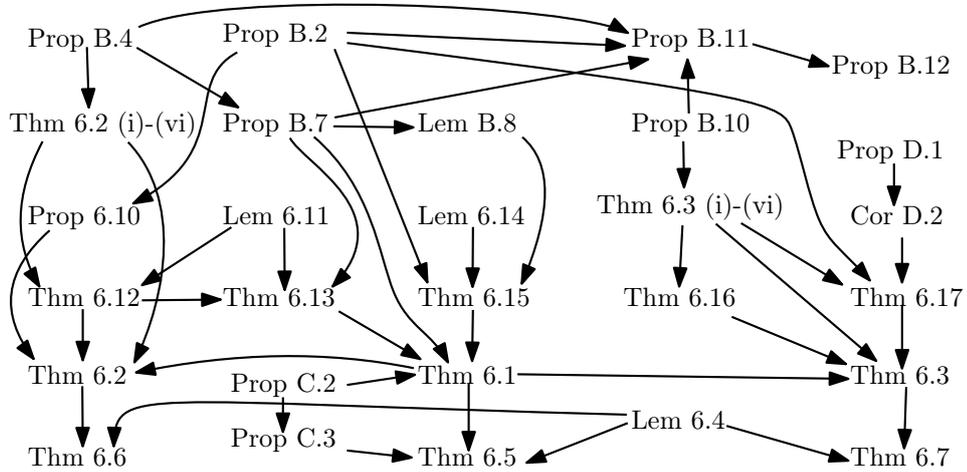 

In Section \ref{sec:nature} we fully characterise the nature of $\Cgf(x,y;t)$ as a function of $x$ into the hierarchy
\[\text{Algebraic}\subset\text{D-finite}\subset\text{D-Algebraic}.\]
We show that the nature of $\Cgf(x,y;t)$ is determined by whether the model decouples (see Definition \ref{defn:decoupling}) and/or has a {\em finite group} (see Appendix \ref{ap:group}). In particular, this shows that $\Cgf(x,y;t)$ has the same nature as the analogous generating function $\Qgf(x,y;t)$ counting walks in the quarter plane (with the same step-set and starting point). Our proof of this result is broken into three theorems: Theorem \ref{thm:algebraic_general_t} characterises the cases where $\Cgf(x,y;t)$ is algebraic, Theorem \ref{thm:D-finite_general_t} characterises the cases where $\Cgf(x,y;t)$ is D-finite and Theorem \ref{thm:D-algebraic_general_t} characterises the remaining cases where $\Cgf(x,y;t)$ is D-algebraic. These breakdown further into several theorems showing some part of the characterisation, as summarised below:
\begin{itemize}
\item Theorem \ref{thm:finite_group_case_thm}: group is finite $\implies$ $\Cgf(x,y;t)$ D-finite, 
\item Theorem \ref{thm:inf_group_non-D-finite}: $\Cgf(x,y;t)$ D-finite $\implies$ group is finite,
\item Theorem \ref{thm:finite_group_alg_thm}: if group is finite, $\Cgf(x,y;t)$ algebraic $\iff$ orbit sum is $0$,
\item Theorem \ref{thm:D-alg_proof}: model decouples $\implies$ $\Cgf(x,y;t)$ D-algebraic,
\item Theorem \ref{thm:non-D-alg_proof}: group is infinite and $\Cgf(x,y;t)$ D-algebraic $\implies$ model decouples.
\end{itemize}
Our results in this section rely on Appendices \ref{ap:nature_analytic}, \ref{ap:group} and \ref{ap:D-trans}, which we will describe below. A chart showing the dependencies between the results involved in the characterisation of the nature of $\Cgf(x,y;t)$ is given in Figure \ref{fig:Theorem_dependencies_sec6}.

Note that by our definition of walks in the 3-quadrant cone, steps directly between $(1,0)$ and $(0,1)$ are allowed, whereas they are forbidden in \cite{bousquet_wallner}, for example. We do not expect this to affect the nature of the generating function, in fact in Section \ref{subsec:annoying_step_forbidding} we show that in most cases the nature is the same.

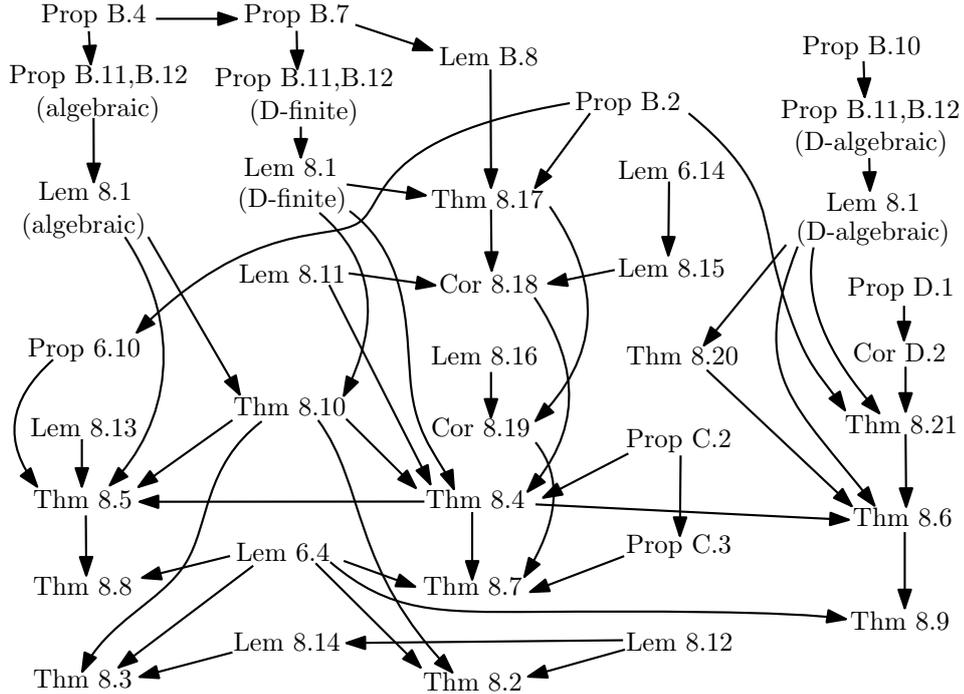
\begin{figure}[ht]
\centering
\begin{picture}(350,280)
   \put(0,0){\includegraphics[scale=1]{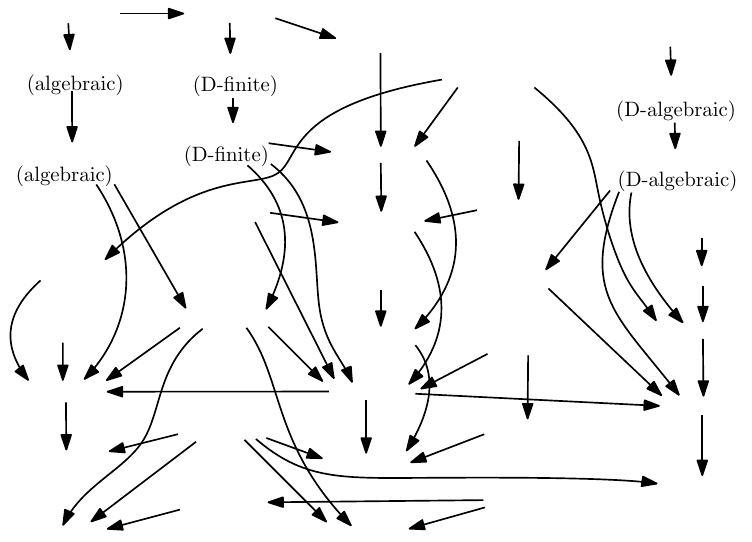}}
   \put(12,2){Thm \ref{thm:bc_even_cone_algebraic}}
   \put(12,37){Thm \ref{thm:bc_algebraic_general_t}}
   \put(12,70){Thm \ref{thm:bc_algebraic_fixed_t}}
   \put(11,97){Lem \ref{lem:bc_odd_algebraic}}
   \put(10,127){Prop \ref{prop:orbit_sum0isdecoupling}}
   \put(14,186){Lem \ref{lem:bc_QjAB_complexities_equivalence}}
   \put(3,230){Prop \ref{prop:PandXP},\ref{prop:XP_closure_properties}}
   \put(15,253){Prop \ref{prop:algebraicofXorY}}
   
   \put(87,16){Lem \ref{lem:bc_even_algebraic}}
   \put(88,50){Lem \ref{lem:function_and_series_complexity}}
   \put(87,105){Thm \ref{thm:bc_finite_group_case_thm}}
   \put(89,155){Lem \ref{lem:bc_odd_D-finite}}
   \put(91,195){Lem \ref{lem:bc_QjAB_complexities_equivalence}}
   \put(80,229){Prop \ref{prop:PandXP},\ref{prop:XP_closure_properties}}
   \put(91,253){Prop \ref{prop:D-finiteofXorY}}
   
   \put(158,1){Thm \ref{thm:bc_even_cone_D-finite}}
   \put(158,37){Thm \ref{thm:bc_D-finite_general_t}}
   \put(159,70){Thm \ref{thm:bc_D-finite_fixed_t}}
   \put(161,97){Cor \ref{cor:bc_inf_group_non-D-finite_M1}}
   \put(161,124){Lem \ref{lem:not-rationalM1}}
   \put(164,151){Cor \ref{cor:bc_inf_group_non-D-finite_Mgeq3}}
   \put(161,183){Thm \ref{thm:bc_inf_group_non-D-finite}}
   \put(164,237){Lem \ref{lem:D-fini_z}}
   
   \put(234,16){Lem \ref{lem:bc_even_D-finite}}
   \put(234,53){Prop \ref{prop:finite_group_general_t}}
   \put(234,93){Prop \ref{prop:finite_group_fixed_t}}
   \put(234,124){Thm \ref{thm:bc_D-alg_proof}}
   \put(231,157){Lem \ref{lem:not-rationalMbig}}
   \put(231,194){Lem \ref{lem:not-rational}}
   \put(215,220){Prop \ref{prop:rationalofXorY}}
   
   \put(318,24){Thm \ref{thm:bc_D-algebraic_general_t}}
   \put(319,63){Thm \ref{thm:bc_D-algebraic_fixed_t}}
   \put(316,98){Thm \ref{thm:bc_non-D-alg_proof}}
   \put(319,125){Cor \ref{cor:D-trans}}
   \put(317,150){Prop \ref{prop:D-trans_poles}}
   \put(309,182){Lem \ref{lem:bc_QjAB_complexities_equivalence}}
   \put(292,217){Prop \ref{prop:PandXP},\ref{prop:XP_closure_properties}}
   \put(300,241){Prop \ref{prop:D-algebraicofXorY}}
   
   \end{picture}
   \caption{Dependencies between results in Section \ref{sec:bc_solutions} and Appendices \ref{ap:nature_analytic}, \ref{ap:group} and \ref{ap:D-trans} as well as three results of Section \ref{sec:nature}. The results of Section \ref{sec:bc_solutions} rely additionally on Sections \ref{sec:param} and \ref{sec:more_cones}, while results in the appendices \ref{ap:nature_analytic}, \ref{ap:group} and \ref{ap:D-trans} rely on Section \ref{sec:param}.}
   \label{fig:Theorem_dependencies_sec8}
\end{figure}

In Sections \ref{sec:more_cones} and \ref{sec:bc_solutions}, we address the enumeration of walks on an $M$-quadrant cone for any positive integer $M$. In Section \ref{sec:more_cones} we define the model precisely and derive analytic functional equations characterising the series $\Qgf_{j}(x,y;t)$ involved, as in Section \ref{sec:functional_equations}. In Section \ref{sec:bc_solutions}, we use these functional equations to determine the nature of the series involved as functions of $x$. Our main results in these sections are listed Below:
\begin{itemize}
\item Theorem \ref{thm:bc_PH_PV_characterisation}: characterisation of generating functions $\Qgf_{j}(x,y;t)$ using analytic functional equation.
\item Nature of $\Qgf_{j}(x,y;t)$ (for $M$ even):
\begin{itemize}
\item Theorem \ref{thm:bc_even_cone_D-finite}: $\Qgf_{j}(x,y;t)$ always D-finite,
\item Theorem \ref{thm:bc_even_cone_algebraic}: $\Qgf_{j}(x,y;t)$ algebraic $\iff$ starting point $(p,q)$ is on certain axis, 
\end{itemize}
\item Nature of $\Qgf_{j}(x,y;t)$ (for $M$ odd):
\begin{itemize}
\item Theorem \ref{thm:bc_D-finite_general_t}: $\Qgf_{j}(x,y;t)$ D-finite $\iff$ group is finite,
\item Theorem \ref{thm:bc_algebraic_general_t}: $\Qgf_{j}(x,y;t)$ algebraic $\iff$ group is finite and orbit sum is $0$,
\item Theorem \ref{thm:bc_D-algebraic_general_t}: if group is infinite, $\Qgf_{j}(x,y;t)$ D-algebraic $\iff$ model decouples.
\end{itemize}
\end{itemize} To our knowledge, Theorem \ref{thm:bc_D-finite_general_t} involves the first proof that the generating function for quarter-plane walks cannot be D-finite in $x$ when the group is infinite, the previously missing element being Lemma \ref{lem:not-rationalM1}. A chart showing the dependicies between the results of Section \ref{sec:bc_solutions} is given in Figure \ref{fig:Theorem_dependencies_sec8}.

Finally in Section \ref{sec:conlusion_and_questions}, we pose a variety of questions left open by this work.

We have a number of appendices in which we prove technical results that we use throughout the article. In Appendix \ref{ap:param_thm}, we use results from \cite{dreyfus2019differential} to prove Lemmas \ref{lem:param} and \ref{lem:Omega} which allow us to relate the generating function $\Cgf(x,y;t)$ to the meromorphic functions ${A}(z)$ and ${B}(z)$. In Appendix \ref{ap:nature_analytic}, we describe how the nature of the generating functions such as $\Cgf(x,y;t)$ relates to the nature of the related analytic functions, such as ${A}(z)$ and ${B}(z)$. In appendix \ref{ap:group}, we define and discuss the group of the walk. Finally in Appendix \ref{ap:D-trans} we discuss results coming from the Galois theory of q-difference equations that we use in the D-transcendental cases.

\section{Parameterisation of the kernel curve}\label{sec:param}
Recall that we fix a (non-singular) step-set $S\subset\{-1,0,1\}^{2}\setminus\{(0,0)\}$, and a weight $w_{s}>0$ for each $s\in S$. A fundamental object of the study of walks with such a fixed weighted step-set is the Kernel $K(x,y;t)$ given by
\[K(x,y;t):=-1+t\Pgf(x,y)=-1+t\sum_{(\alpha,\beta)\in S}w_{(\alpha,\beta)}x^{\alpha}y^{\beta}.\]
Following the method used in the quarter plane pioneered by Fayolle, Iasnogorodski and Raschel \cite{fayolle1979two,fayolle1999random,raschel2012counting} we start by fixing $t\in\left(0,\frac{1}{\Pgf(1,1)}\right)$ and then we consider the curve $\overline{E_{t}}=\{(x,y):K(x,y;t)=0\}$.

Recall our assumption that $S$ is a non-singular step-set. Under this assumption, the curve $\overline{E_{t}}$ is known to have genus $1$, so we will be able to parameterise it using elliptic functions $X(z)$ and $Y(z)$. More precisely we have the following lemma, which we prove in Lemma \ref{lem:xwparam} in Appendix \ref{ap:param_thm},  using results from \cite{dreyfus2019differential}. The transformation that converts Lemma \ref{lem:xwparam} to Lemma \ref{lem:param} is described above Lemma \ref{lem:xwparam}.

\begin{Lemma}\label{lem:param}
There are meromorphic functions $X,Y:\mathbb{C}\to\mathbb{C}\cup\{\infty\}$ and numbers $\gamma,\tau\in i\mathbb{R}$ with $\Im(\pi\tau)>\Im(2\gamma)>0$ satisfying the following conditions
\begin{itemize}
\item $K(X(z),Y(z))=0$
\item $X(z)=X(z+\pi)=X(z+\pi\tau)=X(-\gamma-z)$
\item $Y(z)=Y(z+\pi)=Y(z+\pi\tau)=Y(\gamma-z)$
\item $|X(-\frac{\gamma}{2})|,|Y(\frac{\gamma}{2})|<1$
\item Counting with multiplicity, the functions $X(z)$ and $Y(z)$ each contain two poles and two roots in each fundamental domain $\{z_{c}+r_{1}\pi+r_{2}\pi\tau:r_{1},r_{2}\in[0,1)\}$, where $z_{c}\in\mathbb{C}$ is fixed.
\end{itemize}
Moreover, $X(z)$ and $Y(z)$ are differentially algebraic with respect to $z$ and $t$, while $\tau$ and $\gamma$ are differentially algebraic as functions of $t$.
\end{Lemma}
In fact, it follows from Proposition \ref{prop:xyom} that $(X(z),Y(z))$ parameterises $\overline{E_{t}}$, that is,
\begin{equation}
\overline{E_{t}}=\{(X(z),Y(z)):z\in\mathbb{C}\}.\label{eq:Eparam}
\end{equation}
We intend to substitute $x\to X(z)$ and $y\to Y(z)$ into \eqref{eq:three_quarter_-1}, \eqref{eq:three_quarter_0} and \eqref{eq:three_quarter_1}, however we can only do this as long as the series in these equations converge, which occurs in the situations described by the following lemma

In order to substitute $x\to X(z)$ and $y\to Y(z)$ into the series that we consider, it often suffices to understand how $|X(z)|$ and $|Y(z)|$ compare to $1$. To do this, we prove the following lemma in Appendix \ref{ap:param_thm}, using \cite[Lemma 2.9]{dreyfus2019differential}.
\begin{figure}[ht]
\centering
   \includegraphics[scale=0.5]{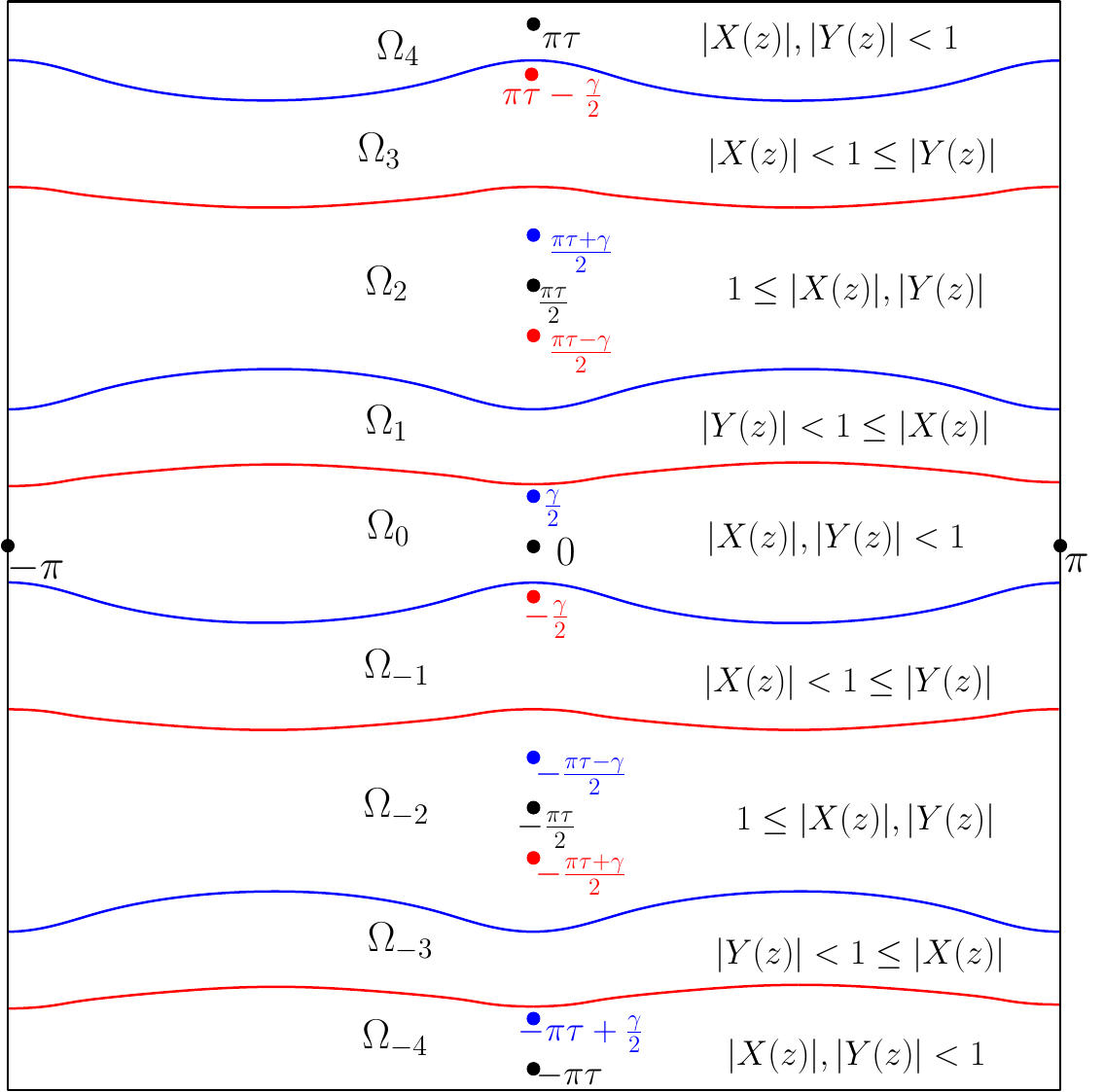}
   \caption{The complex plane partitioned into regions $\Omega_{j}$. For $z$ on the blue lines, $|Y(z)|=1$, while on the red lines $|X(z)|=1$.}
   \label{fig:Omega_chart}
\end{figure}
\begin{Lemma}\label{lem:Omega}
The complex plane can be partitioned into simply connected regions $\{\Omega_{s}\}_{s\in\mathbb{Z}}$ (see Figure \ref{fig:Omega_chart}) satisfying
\begin{align*}
\bigcup_{s\in\mathbb{Z}}\Omega_{4s}\cup\Omega_{4s+1}&=\{z\in\mathbb{C}:|Y(z)|<1\},\\
\bigcup_{s\in\mathbb{Z}}\Omega_{4s-2}\cup\Omega_{4s-1}&=\{z\in\mathbb{C}:|Y(z)|\geq1\},\\
\bigcup_{s\in\mathbb{Z}}\Omega_{4s-1}\cup\Omega_{4s}&=\{z\in\mathbb{C}:|X(z)|<1\},\\
\bigcup_{s\in\mathbb{Z}}\Omega_{4s+1}\cup\Omega_{4s+2}&=\{z\in\mathbb{C}:|X(z)|\geq1\},
\end{align*}
moreover, the equations
\begin{align*}
\pi+\Omega_{s}&=\Omega_{s},\\
s\pi\tau+\gamma-\Omega_{2s}\cup\Omega_{2s+1}&=\Omega_{2s}\cup\Omega_{2s+1}\supset \frac{s\pi\tau+\gamma}{2}+\mathbb{R},\\
s\pi\tau-\gamma-\Omega_{2s}\cup\Omega_{2s-1}&=\Omega_{2s}\cup\Omega_{2s-1}\supset\frac{s\pi\tau-\gamma}{2}+\mathbb{R},
\end{align*}
hold for each $s\in\mathbb{Z}$.
\end{Lemma}
\begin{proof}
This is equivalent to Lemma \ref{lem:appOmega}, using the transformations described in Appendix \ref{ap:param_thm} just before Lemma \ref{lem:xwparam} and just before Lemma \ref{lem:appOmega}
\end{proof}
In some sections it will be useful to parameterise the series using the Jacobi theta function
\begin{equation}\th(z,\tau):=\sum_{n=0}^{\infty}(-1)^n e^{i\pi\tau n(n+1)}\left(e^{(2n+1)iz}-e^{-(2n+1)iz}\right),\label{eq:theta_definition}\end{equation}
which is defined for all $z,\tau\in\mathbb{C}$ satisfying $\Im
(\tau)>0$. In the literature, the function $\th(z,\tau)$ is sometimes written as $\th_{11}(z,\tau)$ or $\theta_{1}(z,e^{i\pi\tau})$. 
Recall that $\tau$ is fixed in this section, so we will generally think of $\th$ as a function of $z$. Note that this function has neither $\pi$ nor $\pi\tau$ as a period, however elliptic functions with these two periods can easily be constructed using $\th$ due to the following relations
\begin{equation}\th(z+\pi,\tau)=-\th(z,\tau)~~~~~~~~\text{and}~~~~~~~~\th(z+\pi\tau,\tau)=-e^{-2iz-i\pi\tau}\th(z,\tau).\label{eq:theta_basic_relations}\end{equation}
Moreover, $\th(z,\tau)$ has no roots and its only poles are at the points $z\in\pi\mathbb{Z}+\pi\tau\mathbb{Z}$. Using these properties allows us to parameterise $X$ and $Y$ using $\th$.
\begin{Proposition}\label{prop:th_param}
There is some $\alpha\in\Omega_{0}\cup\Omega_{-1}$, $\beta\in\Omega_{0}\cup\Omega_{1}$, $\delta\in\Omega_{1}\cup\Omega_{2}$, $\epsilon\in\Omega_{-2}\cup\Omega_{-1}$ and $x_{c},y_{c}\in\mathbb{C}\setminus\{0\}$ satisfying
\begin{align*}
X(z)&=x_{c}\frac{\th(z-\alpha,\tau)\th(z+\gamma+\alpha,\tau)}{\th(z-\delta,\tau)\th(z+\gamma+\delta,\tau)}\\
Y(z)&=y_{c}\frac{\th(z-\beta,\tau)\th(z-\gamma+\beta,\tau)}{\th(z-\epsilon,\tau)\th(z-\gamma+\epsilon,\tau)}.
\end{align*}
\end{Proposition}
\begin{proof} We will prove the result for $X(z)$ as the proof for $Y(z)$ is identical. From Lemma \ref{lem:param}, we know that $X(z)$ contains two roots and two poles in each fundamental domain. Consider the fundamental domain $F=\{z\in\Omega_{-1}\cup\Omega_{0}\cup\Omega_{1}\cup\Omega_{2}:\Re(z)\in[0,\pi)\}$ and let $\alpha\in F$ and $\delta\in F$ be a root and pole of $X(z)$, respectively. From Lemma \ref{lem:Omega}, we must have $\alpha\in \Omega_{-1}\cup\Omega_{0}$ and $\delta\in\Omega_{1}\cup\Omega_{2}$. Now, since $X(z)=X(-\gamma-z)$, the value $-\gamma-\delta$ must also be a pole of $X(z)$, so more generally, all of the points in $\delta+\pi\mathbb{Z}+\pi\tau\mathbb{Z}\cup -\gamma-\delta+\pi\mathbb{Z}+\pi\tau\mathbb{Z}$ are poles of $X$. In the case that $\delta \notin -\frac{\gamma}{2}+\frac{\pi}{2}\mathbb{Z}+\frac{\pi\tau}{2}\mathbb{Z}$, this accounts for all poles of $X(z)$, so $X(z)\th(z-\delta,\tau)\th(z+\gamma+\delta,\tau)$ has no poles. In the case that  $\delta \in -\frac{\gamma}{2}+\frac{\pi}{2}\mathbb{Z}+\frac{\pi\tau}{2}\mathbb{Z}$, we have $X(z-\delta)=X(-\gamma-\delta-z)=X(\delta-z)$, so $X$ must have a double pole at $\delta$. Then again $X(z)\th(z-\delta,\tau)\th(z+\gamma+\delta,\tau)$ has no poles. In either case this implies that the function $\tilde{X}(z)$ defined by
\[\tilde{X}(z):=X(z)\frac{\th(z-\delta,\tau)\th(z+\gamma+\delta,\tau)}{\th(z-\alpha,\tau)\th(z+\gamma+\alpha,\tau)}\]
has no poles except possibly a single pole at each $z\in -\gamma-\alpha+\pi\mathbb{Z}+\pi\tau\mathbb{Z}$. But $\tilde{X}(z)$ is an elliptic function with periods $\pi$ and $\pi\tau$, so it cannot have only a single pole in each fundamental domain \cite[Page 8]{akhiezer1990elements}. Therefore it must have no poles, and is therefore a constant function. Writing $X(z)=x_{c}$ where $x_{c}\in\mathbb{C}$ yields the desired result. Note $x_{c}\neq 0$ as this would imply that $X(z)$ was the $0$ function.
\end{proof}
\textbf{Example:} In the case of simple walks, that is $S=\{(1,0),(0,1),(-1,0),(0,-1)\}$ and $w_{(\alpha,\beta)}=1$ for $(\alpha,\beta)\in S$, the equation relating $X(z)$ and $Y(z)$ is
\[X(z)+\frac{1}{X(z)}+Y(z)+\frac{1}{Y(z)}=\frac{1}{t}.\]
One can check that if we define
\begin{align*}
X(z)&=e^{-i\gamma}\frac{\th(z,\tau)\th(z+\gamma,\tau)}{\th(z-\gamma,\tau)\th(z+2\gamma,\tau)}~~~~~\text{and}~~~~~Y(z)=e^{-i\gamma}\frac{\th(z,\tau)\th(z-\gamma,\tau)}{\th(z+\gamma,\tau)\th(z-2\gamma,\tau)},
\end{align*}
where $\gamma=\frac{\pi\tau}{4}$, then $X(z)+\frac{1}{X(z)}+Y(z)+\frac{1}{Y(z)}$ has no poles, so it must be constant. Then substituting $z=\frac{\gamma}{2}$ yields an equation relating $\tau$ and $t$, which can be written as
\[e^{-i\gamma}\frac{\th\left(\frac{\gamma}{2},\tau\right)^{2}}{\th\left(\frac{3\gamma}{2},\tau\right)^{2}}=\frac{1+2t-\sqrt{1+4t}}{2t}.\]
This allows $q=e^{\frac{i\pi\tau}{4}}=e^{i\gamma}$ to be written as a series in $t$ with initial terms
\[q=t+4t^3+34t^5+360t^7+\cdots\]

\section{Functional equations for walks in the three-quadrant cone}\label{sec:functional_equations}

In this and the following three sections, we consider the generating function $\Cgf(x,y;t)$ counting walks starting at $(p,q)$, taking steps from $S$ with all intermediate points lying in the three-quadrant cone $\mathcal{C}$ and with the weight of the walk being the product of the weights $w_s$ of the steps. Note that the standard starting point is $(p,q)=(1,1)$ and in the unweighted case $w_s=1$ for each $s\in S$. In this section we deduce several systems of functional equations characterising $\Cgf(x,y;t)$. In particular, in Subsection \ref{sec:formal_functional_equations} we discuss functional equations of formal power series, while in Section \ref{sec:functions_analytic} we convert these to analytic functional equations, culminating in Theorem \ref{thm:PH_PV_characterisation}.

\subsection{Functional equations of formal series}\label{sec:formal_functional_equations}

The following lemma results from considering the final step of a walk counted by $\Cgf(x,y;t)$:
\begin{Lemma}\label{lem:3q_eq1}
Define the single step generating function $\Pgf(x,y)$ by
\[\Pgf(x,y)=\sum_{(\alpha,\beta)\in S}w_{(\alpha,\beta)}x^{\alpha}y^{\beta}.\]
Then there are series $\Agf\left(\frac{1}{y};t\right)\in\frac{t}{y}\mathbb{R}\left[\frac{1}{y}\right][[t]]$, $\Bgf\left(\frac{1}{x};t\right)\in\frac{t}{x}\mathbb{R}\left[\frac{1}{x}\right][[t]]$ and $\Fgf(t)\in t\mathbb{R}[[t]]$ which satisfy
\begin{equation}\Cgf(x,y;t)=x^{p}y^{q}+t\Pgf(x,y)\Cgf(x,y;t)-\Fgf(t)-\Bgf\left(\frac{1}{x};t\right)-\Agf\left(\frac{1}{y};t\right).\label{eq:three_quarter_full}\end{equation}
Moreover, this equation together with the fact that $c(i,j;n)=0$ for $i,j\leq 0$, characterises the generating function
\[\Cgf(x,y;t)=\sum_{t\geq 0}\sum_{i,j\in\mathbb{R}} c(i,j;n)x^{i}y^{j}t^{n},\]
as well as the series $\Agf$, $\Bgf$ and $\Fgf$.
\end{Lemma}
\begin{proof}
We start by proving combinatorially that the equation holds for some series $\Agf$, $\Bgf$ and $\Fgf$. Note that $\Cgf(x,y;t)$ is the generating function counting walks restricted to $\cC$, so $t\Pgf(x,y)\Cgf(x,y;t)$ is the generating function counting these walks with an additional step added at the end (possibly not restricted to $\cC$). To get the generating function $\Cgf(x,y;t)$ we need to add the contribution $x^{p}y^{q}$ for the empty walk and subtract the contribution $\Hgf(x,y;t)$ to $t\Pgf(x,y)\Cgf(x,y;t)$ from walks not ending in $\cC$. So
\[\Cgf(x,y;t)=x^{p}y^{q}+t\Pgf(x,y)\Cgf(x,y;t)-\Hgf(x,y;t).\]
Since the walks counted by $\Hgf(x,y;t)$ finish outside $\cC$, and have at least one step, we must have $\Hgf(x,y;t)\in t\mathbb{R}\left[\frac{1}{x},\frac{1}{y}\right][[t]]$. Moreover, since only the final step lies outside $\cC$, the endpoint must be at some $(0,k)$ or $(k,0)$ for $k\leq0$. Hence, we can write $\Hgf$ as
\begin{equation}\Hgf(x,y;t)=\Fgf(t)+\Agf\left(\frac{1}{y};t\right)+\Bgf\left(\frac{1}{x};t\right),\label{eq:newHinFAB}\end{equation}
completing the proof that such an equation holds.

Now we will show that \eqref{eq:three_quarter_full} uniquely determines the series involved. Taking the $x^i y^j t^0$ coefficient on both sides of the equation yields the initial condition $c(i,j;0)=\delta_{(i,j),(p,q)}$, while taking the $x^i y^j t^n$ coefficient for $n\geq1$ and $(i,j)\in \cC$ yields
\[c(i,j;n)=\sum_{(\alpha,\beta)\in S}w_{(\alpha,\beta)}c(i-\alpha,j-\beta;n-1),\]
which determines every value $c(i,j;n)$ inductively. Finally the series $\Hgf$ and therefore the series $\Agf$, $\Bgf$ and $\Fgf$ are determined by
\[\Hgf(x,y;t)=x^{p}y^{q}+t\Pgf(x,y)\Cgf(x,y;t)-\Cgf(x,y;t).\]
\end{proof}
\begin{figure}[ht]
\centering
   \includegraphics[scale=1]{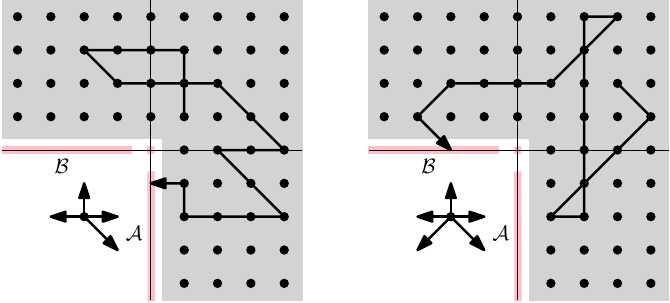}
   \caption{Left: a walk counted by $\Agf\left(\frac{1}{y};t\right)$ in the case where $(p,q)=(1,1)$ and $S$ is the step-set shown. Right: a walk counted by $\Bgf\left(\frac{1}{x};t\right)$ in the case where $(p,q)=(3,2)$ and $S$ is the step-set shown.}
   \label{fig:dead_three-quadrant_walk_examples}
\end{figure}
The equation \eqref{eq:newHinFAB} reveals a combinatorial interpretation of the series $\Agf$, $\Bgf$ and $\Fgf$: While $\Hgf$ counts walks starting at $(p,q)$ and ending just outside $\cC$ whose intermediate points all lie within $\cC$, the series $\Agf$, $\Bgf$ and $\Fgf$ each count a subset of those walks. In particular, the series $\Agf\left(\frac{1}{y};t\right)$ counts those walks ending on the negative $y$-axis $\mathcal{A}:=\{(0,j):j<0\},$ the series
$\Bgf\left(\frac{1}{x};t\right)$  counts those walks ending on the negative $x$-axis
$\mathcal{B}:=\{(i,0):i<0\},$
and $\Fgf(t)$ counts those walks ending at $(0,0)$ (see Figure \ref{fig:dead_three-quadrant_walk_examples}).
\begin{figure}[ht]
\centering
   \includegraphics[scale=1.2]{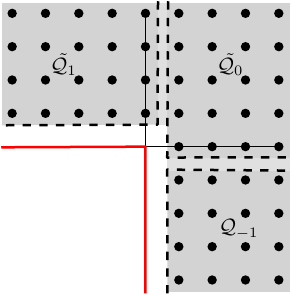}
   \caption{The three-quadrant cone $\mathcal{C}$ partitioned into three quadrants $\tilde{\mathcal{Q}_{1}}$, $\tilde{\mathcal{Q}_{0}}$ and $\mathcal{Q}_{-1}$.}
   \label{fig:three_quarter_quadrants}
\end{figure}

The unusual condition that the coefficients $c(i,j;n)$ of $\Cgf(x,y;t)$ vanish for $i,j\leq 0$ makes this equation difficult to solve directly, so we follow \cite{bousquet2016square,bousquet_wallner} and partition $\mathcal{C}$ into three quadrants $\cQ_{-1},\tilde{\cQ_{0}},\tilde{\cQ_{1}}$, defined as follows (see Figure \ref{fig:three_quarter_quadrants})
\begin{align*}
\cQ_{-1}&:=\{(i,j):i>0;j<0\},\\
\tilde{\cQ_{0}}&:=\{(i,j):i>0;j\geq0\},\\
\tilde{\cQ_{1}}&:=\{(i,j):i\leq0;j>0\}.
\end{align*}
In fact we have shifted the quadrants $\cQ_{-1},\tilde{\cQ_{0}}$ down one space compared to those considered in \cite{bousquet2016square} so that it is impossible to step directly between $\cQ_{-1}$ and $\tilde{\cQ_{1}}$ and so that our condition on the starting point $(p,q)$ is now that $(p,q)\in \tilde{\cQ_{0}}$. Note that we write $\tilde{\cQ_{j}}$ rather than $\cQ_{j}$ when the quadrant includes points on either the $x$-axis or $y$-axis.


Now, we define $\Qgf_{-1}\left(x,\frac{1}{y};t\right)$, $\Qgf_{0}(x,y;t)$ and $\Qgf_{1}\left(\frac{1}{x},y;t\right)$ to be the generating functions counting walks in $\cC$, starting at $(p,q)$ and ending in $\cQ_{-1}$, $\tilde{\cQ_{0}}$ and $\tilde{\cQ_{1}}$, respectively. So
\[\Cgf(x,y;t)=\Qgf_{-1}\left(x,\frac{1}{y};t\right)+\Qgf_{0}(x,y;t)+\Qgf_{1}\left(\frac{1}{x},y;t\right),\]
and $\Qgf_{-1}\in\frac{x}{y}\mathbb{R}\left[x,\frac{1}{y}\right][[t]]$, $\Qgf_{0}\in x\mathbb{R}\left[x,y\right][[t]]$ and $\Qgf_{1}\in y\mathbb{R}\left[\frac{1}{x},y\right][[t]]$. The following lemma rewrites \eqref{eq:three_quarter_full} as three equations characterising $\Qgf_{-1}\left(x,\frac{1}{y};t\right)$, $\Qgf_{0}(x,y;t)$ and $\Qgf_{1}\left(\frac{1}{x},y;t\right)$ using the Kernel $K(x,y;t)=t\Pgf(x,y)-1$.
\begin{Lemma}\label{lem:three3-4eqs}
There are series $\Vgf_{1}(y;t),\Vgf_{2}(y;t)\in\mathbb{R}[y][[t]]$ and $\Hgf_{1}(x;t),\Hgf_{2}(x;t)\in\mathbb{R}[x][[t]]$ satisfying the three equations
\begin{align}
K(x,y;t)\Qgf_{-1}\left(x,\frac{1}{y};t\right)&=\Agf\left(\frac{1}{y};t\right)+\Hgf_{1}(x;t)+\frac{1}{y}\Hgf_{2}(x;t)\label{eq:three_quarter_-1}\\
K(x,y;t)\Qgf_{0}(x,y;t)&=-x^{p}y^{q}+\Fgf\left(t\right)-\Vgf_{1}(y;t)-x\Vgf_{2}(y;t)-\Hgf_{1}(x;t)-\frac{1}{y}\Hgf_{2}(x;t)\label{eq:three_quarter_0}\\
K(x,y;t)\Qgf_{1}\left(\frac{1}{x},y;t\right)&=\Bgf\left(\frac{1}{x};t\right)+\Vgf_{1}(y;t)+x\Vgf_{2}(y;t)\label{eq:three_quarter_1}.
\end{align}
Moreover, these three equations characterise the series $\Qgf_{-1}\left(x,\frac{1}{y};t\right)$, $\Qgf_{0}(x,y;t)$, $\Qgf_{1}\left(\frac{1}{x},y;t\right)$, $\Vgf_{1}(y;t)$, $\Vgf_{2}(y;t)$, $\Hgf_{1}(x;t)$, $\Hgf_{2}(x;t)$, $\Agf\left(\frac{1}{y};t\right)$, $\Bgf\left(\frac{1}{x};t\right)$ and $\Fgf(t)$.
\end{Lemma}
\begin{proof}
We can rewrite \eqref{eq:three_quarter_full} as
\begin{equation}K(x,y;t)\left(\Qgf_{-1}\left(x,\frac{1}{y};t\right)+\Qgf_{0}(x,y;t)+\Qgf_{1}\left(\frac{1}{x},y;t\right)\right)=-x^p y^q+\Fgf(t)+\Agf\left(\frac{1}{y};t\right)+\Bgf\left(\frac{1}{x};t\right)\label{eq:proofeq_three_quarter},\end{equation}
which is precisely the sum of \eqref{eq:three_quarter_-1}, \eqref{eq:three_quarter_0} and \eqref{eq:three_quarter_1}. Rearranging yields
\[K(x,y;t)\Qgf_{1}\left(\frac{1}{x},y;t\right)-\Bgf\left(\frac{1}{x};t\right)=-x^{p}y^{q}+\Fgf(t)+\Agf\left(\frac{1}{y};t\right)-K(x,y;t)\left(\Qgf_{0}(x,y;t)+\Qgf_{-1}\left(x,\frac{1}{y};t\right)\right).\]
The left hand side of this equation lies in $x\mathbb{R}\left[\frac{1}{x},y\right][[t]]$ while the right hand side lies in $\mathbb{R}\left[x,y,\frac{1}{y}\right][[t]]$, so, since the sides are equal, they must lie in the intersection
\[x\mathbb{R}\left[\frac{1}{x},y\right][[t]]\cap\mathbb{R}\left[x,y,\frac{1}{y}\right][[t]]=\mathbb{R}[y][[t]]+x\mathbb{R}[y][[t]],\]
which means there are series $\Vgf_{1}(y;t),\Vgf_{2}(y;t)\in\mathbb{R}[y][[t]]$ satisfying \eqref{eq:three_quarter_1}. Similarly, in the equation
\[K(x,y;t)\left(\Qgf_{0}(x,y;t)+\Qgf_{1}\left(\frac{1}{x},y;t\right)\right)-\Bgf\left(\frac{1}{x};t\right)+x^{p}y^{q}-\Fgf(t)=\Agf\left(\frac{1}{y};t\right)-K(x,y;t)\Qgf_{-1}\left(x,\frac{1}{y};t\right),\]
the left hand side lies in $\frac{1}{y}\mathbb{R}\left[\frac{1}{x},x,y\right][[t]]$ while the right hand side lies in $\mathbb{R}\left[x,\frac{1}{y}\right][[t]]$, so they both lie in the intersection $\mathbb{R}[x][[t]]+\frac{1}{y}\mathbb{R}[x][[t]]$. Hence there are series $\Hgf_{1}(x;t),\Hgf_{2}(x;t)\in\mathbb{R}[x][[t]]$ satisfying \eqref{eq:three_quarter_-1}. Finally \eqref{eq:three_quarter_0} follows from subtracting \eqref{eq:three_quarter_-1} and \eqref{eq:three_quarter_1} from \eqref{eq:three_quarter_full}.
\end{proof}
We also note that there are combinatorial interpretations of the series $\Vgf_{1},\Vgf_{2},\Hgf_{1},\Hgf_{2}$: they each count walks starting at $(p,q)$ and ending either in $\cC$ or just outside $\cC$ with a restriction on the final step. In particular:
\begin{itemize}
\item $\Vgf_{1}(0;t)-\Vgf_{1}(y;t)$ counts walks whose final step is from $\tilde{\cQ_{0}}$ to $\tilde{\cQ_{1}}$,
\item $\Vgf_{1}(0;t)$ counts walks whose final step is from $\tilde{\cQ_{1}}$ to $(0,0)$,
\item $x\Vgf_{2}(y;t)$ counts walks whose final step is from $\tilde{\cQ_{1}}$ to $\tilde{\cQ_{0}}$ ,
\item $\Hgf_{1}(x;t)$ counts walks whose final step is from $\cQ_{-1}$ to $\tilde{\cQ_{0}}$ or $(0,0)$,
\item $-\frac{1}{y}\Hgf_{2}(x;t)$ counts walks whose final step is from $\tilde{\cQ_{0}}$ to $\cQ_{-1}$ or $(0,-1)$.
\end{itemize}
In section \ref{sec:more_cones}, we will use combinatorial interpretations of this form to generalise this section to $M$-quadrant cones for any positive integer $M$.

\subsection{Analytic reformulation of functional equations}\label{sec:functions_analytic}
In this section we will substitute $x\to X(z)$ and $y\to Y(z)$ into the equations of Lemma \ref{lem:three3-4eqs} in order to eliminate the term $K(x,y;t)$. The following lemma explains why these substitutions are possible.
\begin{Lemma}\label{lem:Q_series_convergence}
The series...
\begin{itemize}
\item $\frac{y}{x}\Qgf_{-1}\left(x,\frac{1}{y};t\right)$ converges absolutely for $x,y$ satisfying $|x|\leq 1\leq |y|\leq \infty$,
\item $\frac{1}{x}\Qgf_{0}(x,y;t)$ converges absolutely for $x,y$ satisfying $|x|,|y|\leq 1$,
\item $\frac{1}{y}\Qgf_{1}\left(\frac{1}{x},y;t\right)$ converges absolutely for $x,y$ satisfying $|y|\leq 1\leq |x|\leq \infty$.
\end{itemize}
In fact all of the series that we consider $\Qgf_{-1}$, $\Qgf_{0}$, $\Qgf_{1}$, $A$, $B$ and $\Cgf$ converge absolutely for $x,y$ satisfying
\[|x|,|y|\in \left(\sqrt{tP(1,1)},\sqrt{\frac{1}{tP(1,1)}}\right).\]
\end{Lemma}
\begin{proof}
Since the weighted number of walks of length $n$ in the entire plane is $P(1,1)^n$, the number of walks restricted to the three quarter plane must not be higher than this. Hence, the coefficient $[t^{n}]\Cgf(1,1;t)\leq P(1,1)^{n}$, so for fixed $t<\frac{1}{P(1,1)}$, the series $\Cgf(1,1;t)$ converges. So the series $\Qgf_{-1}(1,1;t)$, $\Qgf_{0}(1,1;t)$ and $\Qgf_{1}(1,1;t)$, whose sum is $\Cgf(1,1;t)$, also converge. Since $\frac{y}{x}\Qgf_{-1}\left(x,\frac{1}{y};t\right)\in \mathbb{R}\left[x,\frac{1}{y}\right][[t]]$, and it has only non-negative coefficients, it must converge for $|x|\leq 1\leq |y|\leq \infty$. Similarly, since $\frac{1}{y}\Qgf_{1}\left(\frac{1}{x},y;t\right)\in \mathbb{R}\left[\frac{1}{x},y\right][[t]]$, it must converge for $|y|\leq 1\leq |x|\leq \infty$. Finally, since $\frac{1}{x}\Qgf_{0}(x,y;t)\in \mathbb{R}\left[x,y\right][[t]]$, it must converge for $|x|,|y|\leq 1$.

The only remaining statement to prove is the final statement of the Lemma, which requires a bit more precision. As we discussed, the coefficient $[t^{n}]\Cgf(1,1,t)\leq P(1,1)^{n}$, moreover, the polynomial $[t^{n}]\Cgf(x,y,t)$ is a Laurent polynomial in $x$ and $y$ with degrees of $x$ and $y$ lying in $[-n,n]$. Hence for fixed $x,y\in\mathbb{C}$, the coefficient 
\begin{align*}|[t^{n}]\Cgf(x,y,t)|&\leq [t^{n}]\Cgf(1,1,t)\max\left\{|xy|^{n},\left|\frac{x}{y}\right|^{n},\left|\frac{y}{x}\right|^{n},\frac{1}{|xy|^{n}}\right\}\\
&\leq \left( \Pgf(1,1)\max\left\{|xy|,\left|\frac{x}{y}\right|,\left|\frac{y}{x}\right|,\frac{1}{|xy|}\right\}\right)^{n}.\end{align*}
So for $t\in\left(0,\frac{1}{\Pgf(1,1)}\right)$ and $|x|,|y|\in \left(\sqrt{tP(1,1)},\sqrt{\frac{1}{t\Pgf(1,1)}}\right)$, the series $\Cgf(x,y,t)$ converges, as required. We can use the same reasoning to prove the same result for all of the other series that we consider. 
\end{proof}

Using this lemma, we can substitute $x=X(z)$ and $y=Y(z)$ into \eqref{eq:three_quarter_-1}, \eqref{eq:three_quarter_0} and \eqref{eq:three_quarter_1} for $z$ in the regions $\Omega_{-1}$, $\Omega_{0}$ and $\Omega_{1}$, respectively, yielding \eqref{eq:3qinz-1}, \eqref{eq:3qinz0} and \eqref{eq:3qinz1} in the following proposition:
\begin{Proposition}\label{prop:PH_PV_prethm}
The functions
\begin{align}
\label{eq:LHdefinition}L_{H}(z)&:=\Hgf_{1}(X(z);t)+\frac{1}{Y(z)}\Hgf_{2}(X(z);t),\qquad&&\text{for $z\in\Omega_{0}\cup\Omega_{-1}$},\\
\label{eq:LVdefinition}L_{V}(z)&:=\Vgf_{1}(Y(z);t)+X(z)\Vgf_{2}(Y(z);t),\qquad&&\text{for $z\in\Omega_{0}\cup\Omega_{1}$},\\
\label{eq:PVdefinition}{A}(z)&:=\Agf\left(\frac{1}{Y(z)};t\right),\qquad&&\text{for $z\in\Omega_{-1}\cup\Omega_{-2}$},\\
\label{eq:PHdefinition}{B}(z)&:=\Bgf\left(\frac{1}{X(z)};t\right),\qquad&&\text{for $z\in\Omega_{1}\cup\Omega_{2}$}.
\end{align}
are well defined and satisfy the equations
\begin{align}
0&={A}(z)+L_{H}(z)\label{eq:3qinz-1}\qquad&&\text{for $z\in\Omega_{-1}$},\\
0&=-X(z)^{p}Y(z)^{q}+\Fgf\left(t\right)-L_{V}(z)-L_{H}(z)\label{eq:3qinz0}\qquad&&\text{for $z\in\Omega_{0}$},\\
0&={B}(z)+L_{V}(z)\label{eq:3qinz1}\qquad&&\text{for $z\in\Omega_{1}$}\\
{B}(z)&={B}(\pi\tau-\gamma-z)={B}(z+\pi)\label{eq:PH_flippy}\qquad&&\text{}\\
{A}(z)&={A}(-\pi\tau+\gamma-z)={A}(z+\pi)\label{eq:PV_flippy}\qquad&&\text{}.
\end{align}
\end{Proposition}
\begin{proof}
In the specified domain of \eqref{eq:LHdefinition}, $|X(z)|\leq1$, so the series $\Hgf_{1}(X(z);t)$ and $\Hgf_{2}(X(z);t)$ converge, which implies that $L_{H}(z)$ is well defined. Similarly the series in \eqref{eq:LVdefinition}, \eqref{eq:PVdefinition} and \eqref{eq:PHdefinition} converge, as the first parameter of each generating function has modulus at most $1$.

Now, \eqref{eq:3qinz-1}, \eqref{eq:3qinz0} and \eqref{eq:3qinz1} follow from substituting $x=X(z)$ and $y=Y(z)$ into \eqref{eq:three_quarter_-1}, \eqref{eq:three_quarter_0} and \eqref{eq:three_quarter_1}, respectively, as we always have $K(X(z),Y(z))=0$, and in each case the series $\Qgf_{j}$ converges.

Finally, for $z\in\Omega_{1}\cup\Omega_{2}$ we have $\pi\tau-\gamma-z\in \Omega_{1}\cup\Omega_{2}$, so ${B}(\pi\tau-\gamma-z)$ is well defined. Then \eqref{eq:PH_flippy} follows from $X(z)=X(\pi\tau-\gamma-z)$. Similarly, for $z\in\Omega_{-1}\cup\Omega_{-2}$, the function ${A}(-\pi\tau+\gamma-z)$ is well defined, and so \eqref{eq:PV_flippy} follows from $Y(z)=Y(-\pi\tau+\gamma-z)$.
\end{proof}

While these equations are a priori defined on different sets, by meromorphic extension we will be able to compare them directly, as we will see from the following theorem:
\begin{Theorem}\label{thm:PH_PV_characterisation}
The functions ${A}(z)$ and ${B}(z)$ extend to meromorphic functions on $\mathbb{C}$ which, along with the constant $F= \Fgf(t)$, are uniquely defined by the equations
\begin{align}X(z)^{p}Y(z)^{q}&={A}(z)+F+{B}(z)\label{eq:PH_PV_nice},\\
{B}(z)&={B}(\pi\tau-\gamma-z)\label{eq:PH_flippy_full},\\
{A}(z)&={A}(-\pi\tau+\gamma-z)\label{eq:PV_flippy_full},\\
{B}(z)&={B}(z+\pi)\label{eq:PH_pi},\\
{A}(z)&={A}(z+\pi)\label{eq:PV_pi},
\end{align}
along with the conditions
\begin{enumerate}[label={\rm(\roman*)},ref={\rm(\roman*)}]
\item ${A}(z)$ has no poles in $\Omega_{0}\cup\Omega_{-1}\cup\Omega_{-2}$,\label{Anopoles}
\item the poles of $Y(z)$ for $z\in\Omega_{-1}\cup\Omega_{-2}$ are roots of $A(z)$,\label{Aroots}
\item ${B}(z)$ has no poles in $\Omega_{0}\cup\Omega_{1}\cup\Omega_{2}$,\label{Bnopoles}
\item the poles of $X(z)$ for $z\in\Omega_{1}\cup\Omega_{2}$ are roots of ${B}(z)$.\label{Broots}
\end{enumerate}
\end{Theorem}
\begin{proof}
From \eqref{eq:3qinz-1}, we have ${A}(z)=-L_{H}(z)$ for $z\in\Omega_{-1}$. But the right hand side is meromorphic on $\Omega_{-1}\cup\Omega_{0}$, so we can use it to extend ${A}$ to a meromorphic function on $\Omega_{-2}\cup\Omega_{-1}\cup\Omega_{0}$. Similarly, \eqref{eq:3qinz0} allows us to write ${A}(z)$ as a function of $L_{V}(z)$ for $z\in\Omega_{0}$, allowing us to extend the domain of ${A}$ to include $\Omega_{1}$. Finally \eqref{eq:3qinz1} allows us to write ${A}(z)$ as a function of ${B}(z)$ on $\Omega_{1}$, allowing us to extend ${A}(z)$ to the entire domain $\Omega_{-2}\cup\Omega_{-1}\cup\Omega_{0}\cup\Omega_{1}\cup\Omega_{2}$. Similarly ${B}(z)$ extends to a meromorphic function on $\Omega_{-2}\cup\Omega_{-1}\cup\Omega_{0}\cup\Omega_{1}\cup\Omega_{2}$ as do $L_{H}(z)$ and $L_{V}(z)$. Since \eqref{eq:3qinz-1}, \eqref{eq:3qinz0} and \eqref{eq:3qinz1} hold in regions with non-empty interior, they must hold on all of $\Omega_{-2}\cup\Omega_{-1}\cup\Omega_{0}\cup\Omega_{1}\cup\Omega_{2}$. Finally, adding these three equations and multiplying by $X(z)Y(z)/t$ yields \eqref{eq:PH_PV_nice}.

To extend the functions to $\mathbb{C}$, we use \eqref{eq:PH_flippy} and \eqref{eq:PV_flippy} as follows: From Lemma \ref{lem:Omega}, we know that
\[-\frac{1}{2}(\pi\tau-\gamma)+\mathbb{R}\subset\Omega_{-2}\cup\Omega_{-1}\qquad\text{and}\qquad\frac{1}{2}(\pi\tau-\gamma)+\mathbb{R}\subset\Omega_{1}\cup\Omega_{2}.\]
Hence both of these lines, and the region delimited by these lines, is contained in $\Omega_{-2}\cup\Omega_{-1}\cup\Omega_{0}\cup\Omega_{1}\cup\Omega_{2}$. Now, since $A(z)$ satisfies \eqref{eq:PV_flippy} for $z\in \Omega_{-2}\cup\Omega_{-1}$, it extends meromorphically to the function $A(z):=A(-\pi\tau+\gamma-z)$ for $z\in -\pi\tau+\gamma-\Omega_{-2}\cup\Omega_{-1}\cup\Omega_{0}\cup\Omega_{1}\cup\Omega_{2}$, which contains the region between the lines
\[-\frac{3}{2}(\pi\tau-\gamma)+\mathbb{R}\qquad\text{and}\qquad-\frac{1}{2}(\pi\tau-\gamma)+\mathbb{R}.\]
Together, these extend the definition of $A(z)$ to the entire region containing $\Omega_{1}\cup\Omega_{2}$ and $-\pi\tau+\gamma-\Omega_{1}\cup\Omega_{2}=\Omega_{1}\cup\Omega_{2}-2(\pi\tau-\gamma)$ as well as the region between these spaces.  
Now, for $z\in\Omega_{1}\cup\Omega_{2}$ we can combine \eqref{eq:PH_flippy}, \eqref{eq:PV_flippy} and \eqref{eq:PH_PV_nice}, giving
\[A(z)=X(z)^{p}Y(z)^{p}-X(\pi\tau-\gamma-z)^{p}Y(\pi\tau-\gamma-z)^{p}+A(z-2(\pi\tau-\gamma)).\]
This recursively allows us to extend $A(z)$ meromorphically to the space between $\Omega_{1}\cup\Omega_{2}+2k(\pi\tau-\gamma)$ and $\Omega_{1}\cup\Omega_{2}+2(k+1)(\pi\tau-\gamma)$ for any integer $k$. Hence $A(z)$ is a meromorphic function on the union $\mathbb{C}$ of these spaces. The relations between $A(z)$, $L_{H}(z)$, $L_{V}(z)$ and $B(z)$ then allow these other three functions to extend meromorphically to $\mathbb{C}$. 

We will now show that conditions \ref{Anopoles}-\ref{Broots} hold. The definition \eqref{eq:PHdefinition} of ${B}(z)$ for $z\in\Omega_{1}\cup\Omega_{2}$ converges, so ${B}(z)$ has no poles in this region. Moreover, since the series has $1/|X(z)|$ as a factor, it has roots at the poles of $X(z)$ in $\Omega_{1}\cup\Omega_{2}$. In $\Omega_{0}$, we have 
\[{B}(z)=-L_{V}(z)=-V_{1}(Y(z);t)-X(z)V_{2}(Y(z);t).\]
Since $|X(z)|,|Y(z)|\leq 1$ in this region, the series converge and there are still no poles. This proves the two conditions \ref{Bnopoles} and \ref{Broots}. Similarly, the conditions \ref{Anopoles} and \ref{Aroots} hold.

Finally we need to show that these conditions uniquely define the functions ${B}(z)$, ${A}(z)$ and the constant $F$. Suppose that $\hat{B}(z)$, $\hat{A}(z)$ and $\hat{F}$ is an arbitrary triple satisfying the same conditions. Then it suffices to show that ${A}(z)=\hat{A}(z)$ and ${B}(z)=\hat{B}(z)$. From \eqref{eq:PH_PV_nice}, we can define
\[\Delta(z):={A}(z)-\hat{A}(z)=\hat{B}(z)-{B}(z)+\hat{F}-F,\]
which satisfies $\Delta(z)=\Delta(\pi\tau-\gamma-z)=\Delta(-\pi\tau+\gamma-z)=\Delta(z+\pi)$. Moreover, the four conditions on ${A}(z)$ and ${B}(z)$ imply, respectively, that
\begin{enumerate}[label={\rm(\roman*)},ref={\rm(\roman*)}]
\item $\Delta(z)$ has no poles in $\Omega_{0}\cup\Omega_{1}\cup\Omega_{2}$,
\item the poles of $X(z)$ for $z\in\Omega_{1}\cup\Omega_{2}$ are roots of $\Delta(z)$+$F-\hat{F}$,
\item $\Delta(z)$ has no poles in $\Omega_{0}\cup\Omega_{-1}\cup\Omega_{-2}$,
\item the poles of $Y(z)$ for $z\in\Omega_{-1}\cup\Omega_{-2}$ are roots of $\Delta(z)$.
\end{enumerate}
Together with $\Delta(z)=\Delta(\pi\tau-\gamma-z)=\Delta(-\pi\tau+\gamma-z)$, these imply that $\Delta(z)$ is an elliptic function with no poles, so it is constant. Moreover, the fourth condition implies that $\Delta(z)$ does have roots, so $\Delta(z)$ is the $0$ function. The second condition then implies that $F=\hat{F}$. Together with the definition of $\Delta$ we have ${A}(z)=\hat{A}(z)$ and ${B}(z)=\hat{B}(z)$. 
\end{proof}

At first glance it may seem that \eqref{eq:PH_PV_nice} follows from substituting $(x,y)=(X(z),Y(z))$ directly into \eqref{eq:three_quarter_full}, however this ignores the questions of whether the series involved converge and the fact that ${A}(z)$ and ${B}(z)$ are defined on non-intersecting domains. So in some sense the intermediate steps are just used to understand which domains should be used to define ${A}(z)$ and ${B}(z)$.

Note that combining \eqref{eq:PH_PV_nice}, \eqref{eq:PH_flippy_full}, \eqref{eq:PV_flippy_full} yields
\begin{equation}{B}(2\pi\tau-2\gamma+z)-{B}(z)=J(z),\label{eq:PH_qdiff}\end{equation}
where $J(z)$ is an elliptic function with periods $\pi$ and $\pi\tau$ given by
\begin{equation}J(z):=\left(X(z-2\gamma)^{p}-X(z)^{p}\right)Y(z)^{q}.\label{def:j}\end{equation}

An equation analogous to \eqref{eq:PH_PV_nice} was found by Raschel for walks in the quarter plane \cite{raschel2012counting}, the only difference being that the transformations $z\to \pi\tau-\gamma-z$ and $z\to -\pi\tau+\gamma-z$ here are $z\to \gamma-z$ and $z\to -\gamma-z$ in the quarter plane. Raschel used this equation to derive an integral-expression solution determining $\Qgf(x,y;t)$ (in the case where the starting point $(p,q)=(1,1)$), and the equation has since been used to determine precisely when $\Qgf(x,y;t)$ is differentially algebraic \cite{bernardi2017counting,dreyfus2018nature,hardouin2020differentially} and to determine when it is algebraic or D-finite with respect to $x$ or $y$ \cite{fayolle2010holonomy,kurkova2012functions}. Following these methods we will prove many of the same results in the three-quadrant cone.
\section{Integral expressions for walks in the three-quadrant cone}\label{sec:integral_expression}


In order to solve the functional equations in Theorem \ref{thm:PH_PV_characterisation}, we will introduce an explicit elliptic function $W(z)$ with periods $\pi$ and $2\pi\tau-2\gamma$, as this will be related to ${B}(z)$ due to \eqref{eq:PH_qdiff}. We will define $W(z)$ using the Jacobi theta function $\th(z,\tau)$ defined in \eqref{eq:theta_definition}, and we will use the parameters $x_c,\delta,\epsilon$ from Proposition \ref{prop:th_param}, which depend on $t$ but not on $z$.
\begin{Definition}\label{def:W}
We define the function $W(z)$ by\begin{equation}\label{eq:om_def}W(z):=\omega_{c}\frac{\th(z-\epsilon,2\tau-\frac{2\gamma}{\pi})\th(z-\pi\tau+\gamma+\epsilon,2\tau-\frac{2\gamma}{\pi})}{\th(z-\delta,2\tau-\frac{2\gamma}{\pi})\th(z-\pi\tau+\gamma+\delta,2\tau-\frac{2\gamma}{\pi})},\end{equation}
where $\omega_{c}$ is given by
\[\omega_{c}=x_{c}\frac{\th'(0,2\tau-\frac{2\gamma}{\pi})\th(2\delta-\pi\tau+\gamma,2\tau-\frac{2\gamma}{\pi})\th(\delta-\alpha,\tau)\th(\delta+\gamma+\alpha,\tau)}{\th(\delta-\epsilon,2\tau-\frac{2\gamma}{\pi})\th(\delta-\pi\tau+\gamma+\epsilon,2\tau-\frac{2\gamma}{\pi})\th'(0,\tau)\th(2\delta+\gamma,\tau)},\]
unless $2\delta+\gamma\in\pi\tau+\pi\mathbb{Z}$, in which case the numerator and denominator are both $0$ so we define
\[\omega_{c}=-e^{i\pi\tau}x_{c}\frac{\th'(0,2\tau-\frac{2\gamma}{\pi})^2\th(\delta-\alpha,\tau)\th(\delta+\gamma+\alpha,\tau)}{\th(\delta-\epsilon,2\tau-\frac{2\gamma}{\pi})\th(\delta-\pi\tau+\gamma+\epsilon,2\tau-\frac{2\gamma}{\pi})\th'(0,\tau)^{2}}.\]
\end{Definition}

In the following proposition we show a number of properties of $W(z)$ which will be useful in relating it to ${A}(z)$, ${B}(z)$ and $F$:

\begin{Proposition}\label{prop:omega}
The function $W(z)$ satisfies the following properties
\begin{enumerate}[label={\rm(\roman*)},ref={\rm(\roman*)}]
\item $W(z)=W(\pi\tau-\gamma-z)=W(-\pi\tau+\gamma-z)=W(z+\pi)$.
\item $1/W(z)$ and $Y(z)$ have exactly the same poles in $\Omega_{-2}\cup\Omega_{-1}$.
\item $W(z)$ and $X(z)$ have exactly the same poles in $\Omega_{1}\cup\Omega_{2}$.
\item $W(z)-X(z)$ has no poles in $\Omega_{1}\cup\Omega_{2}$.
\item the sets $W(\Omega_{-2}\cup\Omega_{-1})$, $W(\Omega_{0})$ and $W(\Omega_{1}\cup\Omega_{2})$ are disjoint and their union is $\mathbb{C}\cup\{\infty\}$. 
\end{enumerate}
\end{Proposition}
\begin{proof}
Property (i) follows immediately from the quasi-periodicity conditions satisfied by $\th$ (see \eqref{eq:theta_basic_relations}). As a consequence of (i), the function $W(z)$ is elliptic with periods $\pi$ and $2\pi\tau-2\gamma$. Moreover $W(z)$ has two roots inside each fundamental domain, one coming from each $\th$ in the numerator of \eqref{eq:om_def}. As a consequence, $W(z)$ takes each value in $\mathbb{C}\cup\infty$ exactly twice on each fundamental domain. We will now prove (v) as this will be useful for proving (ii)-(iv).

\begin{figure}[ht]
\centering
   \includegraphics[scale=0.5]{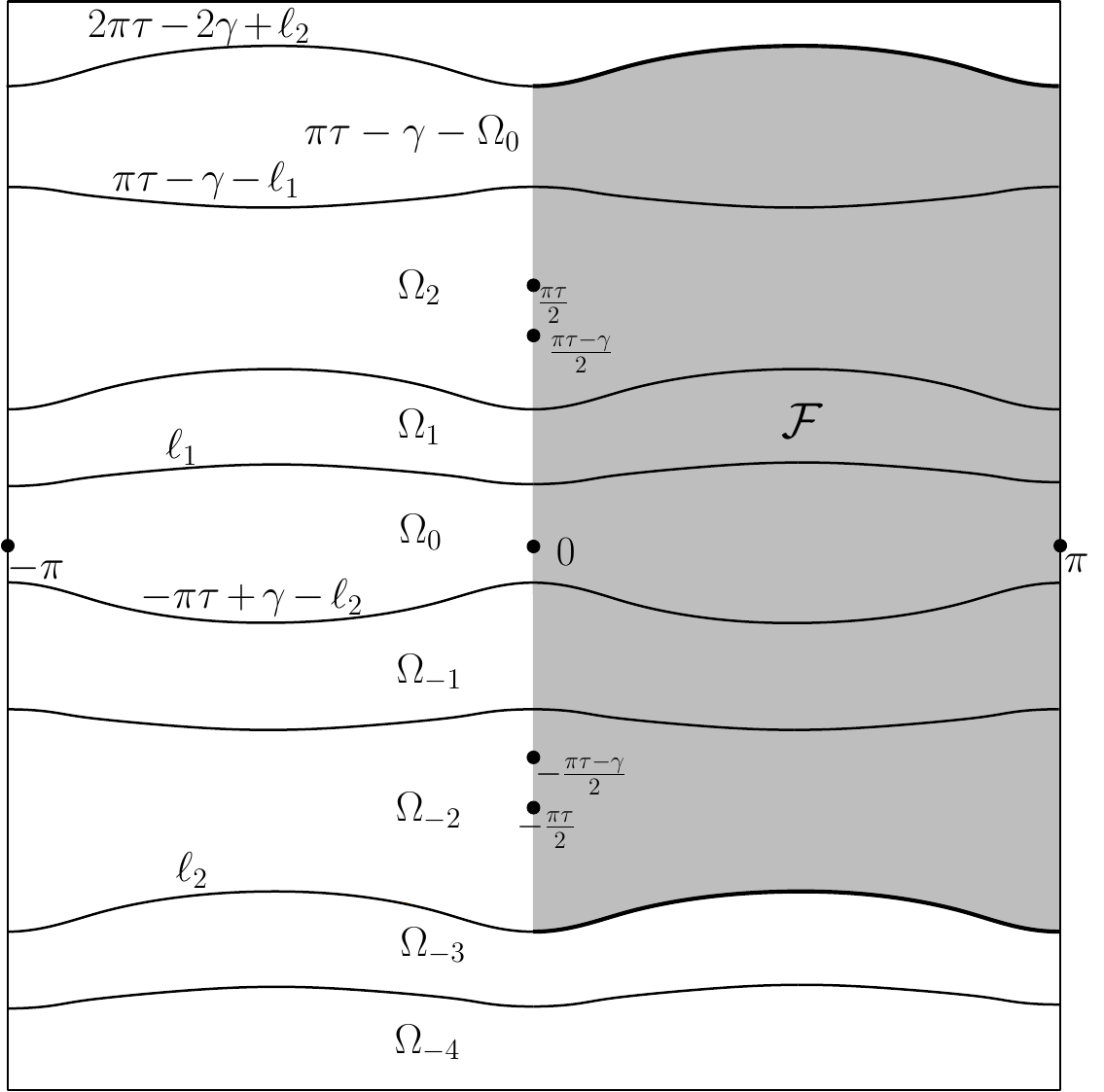}
   \caption{The space $\mathcal{F}$ is shaded. This is a fundamental domain of the function $W(z)$.}
   \label{fig:F_space}
\end{figure}

Let $\mathcal{S}$ denote the strip $\mathcal{S}=\{z\in\mathbb{C}:0\leq\Re(z)<\pi\}$. We claim that the set $\mathcal{F}$ defined by
\begin{align*}\tilde{\mathcal{F}}&=\Omega_{-2}\cup\Omega_{-1}\cup\Omega_{0}\cup\Omega_{1}\cup\Omega_{2}\cup\left(\pi\tau-\gamma-\Omega_{0}\right)\\
\mathcal{F}&=\tilde{\mathcal{F}}\cap\mathcal{S}\end{align*}
is a fundamental domain of $W$, and the sets involved in the union defining $\tilde{\mathcal{F}}$ are disjoint (see Figure \ref{fig:F_space}). Indeed the upper border $\ell_{1}$ of $\Omega_{0}$ is the lower border of $\Omega_{1}\cup\Omega_{2}$, so $\pi\tau-\gamma-\ell_{1}$ is the upper border of $\Omega_{2}$ and the lower border of $\pi\tau-\gamma-\Omega_{0}$. Hence the sets $\Omega_{-2}$, $\Omega_{-1}$, $\Omega_{0}$, $\Omega_{1}$, $\Omega_{2}$ and $\pi\tau-\gamma-\Omega_{0}$ are disjoint and their union $\tilde{\mathcal{F}}$ is a connected strip delimited by the lower border of $\Omega_{-2}$ and the upper border of $\pi\tau-\gamma-\Omega_{0}$. Now let $\ell_{2}$ be the lower border of $\Omega_{-2}$. Then the upper border of $\Omega_{-1}$, which is also the lower border of $\Omega_{0}$ is $-\pi\tau+\gamma-\ell_{2}$, so the upper border of $\pi\tau-\gamma-\Omega_{0}$ is $2\pi\tau-2\gamma+\ell_{2}$. Hence the lower and upper borders of $\tilde{\mathcal{F}}$ are $\ell_{2}$ and $2\pi\tau-2\gamma+\ell_{2}$. Moreover, only the lower border is contained in $\tilde{\mathcal{F}}$. This implies that 
\begin{align*}\mathbb{C}&=\bigcup_{n\in\mathbb{Z}}\tilde{\mathcal{F}}+(2\pi\tau-2\gamma)n\\
&=\bigcup_{n,m\in\mathbb{Z}}\mathcal{F}+(2\pi\tau-2\gamma)n+\pi m,\end{align*}where these unions are in fact {\em disjoint} unions. So $\mathcal{F}$ is a fundamental domain for $W(z)$, as claimed. Hence, counting with multiplicity, $W(z)$ takes each value in $\mathbb{C}\cup\infty$ twice on $\mathcal{F}$. Now, since $W(z)=W(\pi\tau-\gamma-z)$, each value is taken either $0$ or $2$ times for $z\in\left(\Omega_{1}\cup\Omega_{2}\right)\cap\mathcal{S}$ and each value taken by $W(z)$ for $z\in\Omega_{0}\cap\mathcal{S}$ is taken the same number of times in $\pi\tau-\gamma-\Omega_{0}\cap\mathcal{S}$. Similarly, since $W(z)=W(-\pi\tau+\gamma-z)$, each value is taken either $0$ or $2$ times in $\left(\Omega_{-2}\cup\Omega_{-1}\right)\cap\mathcal{S}$. This implies that each value in $\mathbb{C}$ is either taken twice in $\left(\Omega_{-2}\cup\Omega_{-1}\right)\cap\mathcal{S}$, once in $\Omega_{0}\cap\mathcal{S}$ or twice in $\left(\Omega_{1}\cup\Omega_{2}\right)\cap\mathcal{S}$, which proves (v).

To prove (ii), it suffices to prove that $W(z)^{-1}$ and $Y(z)$ have the same poles in $\left(\Omega_{-2}\cup\Omega_{-1}\right)\cap\mathcal{S}$. Indeed, they share the poles $\epsilon$ and $-\pi\tau+\gamma-\epsilon$ (each shifted by an integer multiple of $\pi$) in this region, and since we know that this region is a subset of a fundamental domain of $Y(z)$ and of $W(z)$, neither of these functions can have any other poles in this region.

The proof of (iii) is similar to the proof of (ii): in this case $X(z)$ and $W(z)$ share the poles $\delta$ and $\pi\tau-\gamma-\delta$ in the region $\Omega_{1}\cup\Omega_{2}$. 

Finally to prove (iv) it suffices to observe that the poles at $\epsilon$ and $\pi\tau-\gamma-\epsilon$ cancel in the difference $X(z)-W(z)$. This is due to the choice of $\omega_{c}$.
\end{proof}

We now give integral expressions analogous to those of Raschel \cite{raschel2012counting} which determine ${A}(z)$, ${B}(z)$ and $F$ exactly:
\begin{Theorem}\label{thm:integral_expression}
Let $z_{0}\in\Omega_{0}$ and let $\mathcal{L}$ be a path from $z_{0}$ to $z_{0}+\pi$ contained in the closure $\overline{\Omega_{0}}$ of $\Omega_{0}$. Then ${A}(z)$, ${B}(z)$ and $F$ are given by the integrals
\begin{align}
{B}(u)&=\frac{1}{2\pi i}\int_{\mathcal{L}}X(z)^{p}Y(z)^{q}\frac{W'(z)}{W(z)-W(u)}dz,\qquad&&\text{for $u\in\Omega_{1}\cup\Omega_{2}$,}\label{eq:PH_integral}\\
{A}(u)&=-\frac{1}{2\pi i}\int_{\mathcal{L}}X(z)^{p}Y(z)^{q}\frac{W(u)}{W(z)}\frac{W'(z)}{W(z)-W(u)}dz,\qquad&&\text{for $u\in\Omega_{-1}\cup\Omega_{-2}$,}\label{eq:PV_integral}\\
F&=-\frac{1}{2\pi i}\int_{\mathcal{L}}X(z)^{p}Y(z)^{q}\frac{W'(z)}{W(z)}dz.\qquad\label{eq:C_integral}&&
\end{align}
\end{Theorem}
\begin{proof}
We will start by showing that the integrands in the definitions above are all holomorphic in $\overline{\Omega_{0}}$, that is, they contain no poles in this region. A consequence is that the integrals do not depend on the contour $\mathcal{L}$ taken from $z_{0}$ to $z_{0}+\pi$.

Since $|X(z)|,|Y(z)|\leq 1$ in this region, $X(z)$ and $Y(z)$ have no poles in this region. Moreover, $W(z)$ has no roots or poles in this region, so the only way that a pole could occur in one of the integrands is if $W(z)=W(u)$ for some $u\in\Omega_{-2}\cup\Omega_{-1}\cup\Omega_{1}\cup\Omega_{2}$ and $z\in\Omega_{0}$, but this is impossible by Proposition \ref{prop:omega}, (v). This proves that the integrands are all holomorphic in $\overline{\Omega_{0}}$.

We will now extend the expressions \eqref{eq:PH_integral} and \eqref{eq:PV_integral} of ${B}$ and ${A}$ to definitions that hold in $\Omega_{0}$. Let $\mathcal{L}_{1}$ be the contour which goes in the negative imaginary direction from $z_0$ until reaching a point $c_{1}$ on the boundary of $\Omega_{0}$, then travels along the boundary of $\Omega_{0}$ until reaching $c_{1}+\pi$ then finally travels in the positive imaginary direction until reaching $z_0+\pi$ (see Figure \ref{fig:L_contour}). Then ${B}$ can be defined by taking the integral along $\mathcal{L}_{1}$. In fact the first and last sections of the integral cancel with each other, so this can be defined by taking the integral only using the section $\hat{\mathcal{L}}_{1}$ of $\mathcal{L}_{1}$ lying between $c_{1}$ and $c_{1}+\pi$. With this definition it is clear that ${B}(u)$ extends analytically to $\Omega_{0}$. In particular, this proves that ${B}(z)$ satisfies condition (iii) of Theorem \ref{thm:PH_PV_characterisation}. Similarly \eqref{eq:PV_integral} can be defined using the contour $\mathcal{L}_{2}$ which travels in the positive imaginary direction until reaching a point $c_{2}$ on the boundary of $\Omega_{0}$, then travels along this boundary to $c_{2}+\pi$ before travelling in the negative imaginary direction to $\pi$. In particular, this proves that ${A}(z)$ satisfies condition (i) of Theorem \ref{thm:PH_PV_characterisation}.

\begin{figure}[ht]
\centering
   \includegraphics[scale=0.5]{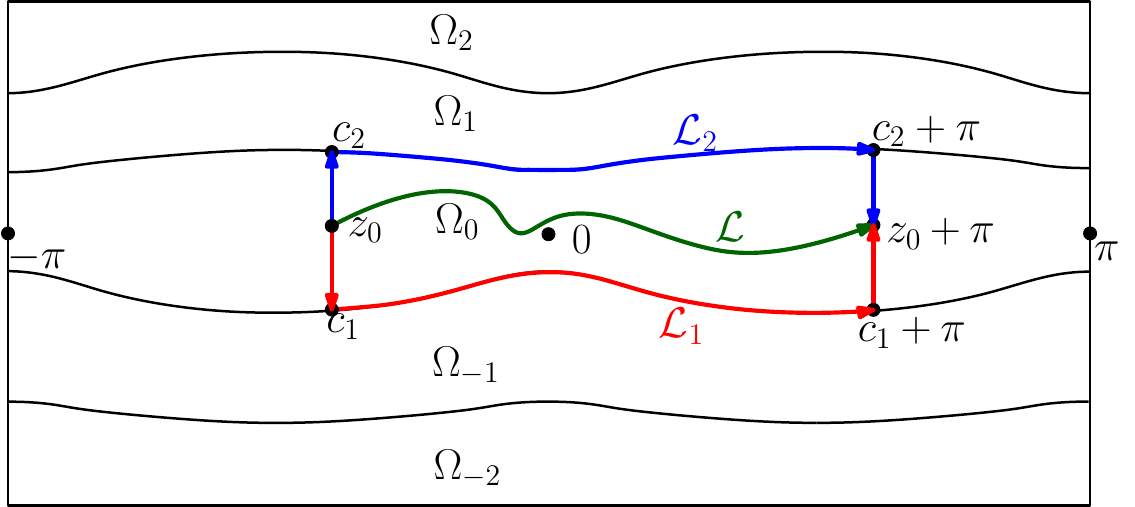}
   \caption{The contours $\mathcal{L}$, $\mathcal{L}_{1}$ and $\mathcal{L}_{2}$ from $z_{0}$ to $z_{0}+\pi$.}
   \label{fig:L_contour}
\end{figure}

We will now show that \eqref{eq:PH_PV_nice}-\eqref{eq:PV_pi} hold. In fact, \eqref{eq:PH_flippy_full}-\eqref{eq:PV_pi} follow immediately from
\[W(u)=W(u+\pi)=W(\pi\tau-\gamma-u)=W(\gamma-\pi\tau-u).\]
To show that \eqref{eq:PH_PV_nice} holds, first note that, from the definitions we have
\[F+{A}(u)=-\frac{1}{2\pi i}\int_{\mathcal{L}_{2}}X(z)^{p}Y(z)^{q}\frac{W'(z)}{W(z)-W(u)}dz.\]
So, defining $\mathcal{L}_{3}$ as the contour formed by $\mathcal{L}_{1}$ followed by $\mathcal{L}_{2}$ reversed, then we have, for $u$ inside this region (and hence inside $\Omega_{0}$),
\[F+{A}(u)+{B}(u)=\frac{1}{2\pi i}\int_{\mathcal{L}_{3}}X(z)^{p}Y(z)^{q}\frac{W'(z)}{W(z)-W(u)}dz.\]
The only pole of the integrand in the interior of this region occurs at $z=u$, and the residue at this point is $X(u)^{p}Y(u)^{q}$, so \eqref{eq:PH_PV_nice} follows from the residue theorem.

We will now show that conditions (ii) and (iv) of Theorem \ref{thm:PH_PV_characterisation} hold. Condition (iv) follows from the fact that the integrand in \eqref{eq:PH_integral} is $0$ when $W(u)=\infty$, along with Proposition \ref{prop:omega}, (iii). Similarly, condition (ii) follows from the fact that the integrand in \eqref{eq:PV_integral} is $0$ when $W(u)=0$  along with Proposition \ref{prop:omega}, (ii).

We have now shown that $F$, ${A}(u)$ and ${B}(u)$ as defined in \eqref{eq:C_integral}, \eqref{eq:PV_integral} and \eqref{eq:PH_integral} satisfy all of the conditions of Theorem \ref{thm:PH_PV_characterisation}. This completes the proof that these are indeed the functions $F$, ${A}(u)$ and ${B}(u)$ defined in the previous section.  
\end{proof}

\section{Three quadrant walks starting on an axis}\label{sec:axis_start}
We will now discuss a question suggested by Kilian Raschel where the walk starts at some point $(p,0)$ for $p>0$ (or equivalently $(0,q)$) rather than the traditional starting point $(1,1)$. Trotignon and Raschel conjectured that with this starting point all finite group models admit algebraic generating functions \cite{raschel2018walks}. This will follow from our results in the Section \ref{sec:nature}, and more precisely Theorem \ref{thm:algebraic_fixed_t} as these cases decouple for any step set (by setting $R_{1}(x)=x^{p}$ and $R_{2}(y)=0$). Moreover, using Theorem \ref{thm:D-algebraic_fixed_t}, the same argument proves that for any step set $S$, the walks starting at such a point have a generating function which is differentially algebraic in $x$. In this section we go further than these results by finding an explicit, general formula for the counting function of walks starting at a point $(p,0)$. This formula will use the function $W(z)$ defined in Definition \ref{def:W}. See also Proposition \ref{prop:omega}, which describes several properties of $W(z)$.

\begin{Theorem}\label{thm:P-xaxis}
For $p\geq 1$, let $A_{p}$, $B_{p}$ and $F_{p}$ denote the functions $A$, $B$ and $F$ arising in the case that the starting point of the walks is $(p,0)$. Then for each $p$, there is a degree $p$ polynomial $H_{p}$ satisfying
\begin{align}
A_{p}(z)&=H_{p}(W(z))-H_{p}(0),\label{eq:PV_jstart}\\
F_{p}(t)&=H_{p}(0),\label{eq:B_jstart}\\
B_{p}(z)&=X(z)^{p}-H_{p}(W(z)).\label{eq:PH_jstart}
\end{align}
Moreover, $H_{p}$ is uniquely determined by the fact that the right hand side of \eqref{eq:PH_jstart} has a root at $z=\delta$. Furthermore, the leading coefficient of $H_{p}$ is $1$.
\end{Theorem}
\begin{proof}
We will start by proving that the polynomial $H_{p}$ is uniquely defined by the fact that the right hand side of \eqref{eq:PH_jstart} has a root at $z=\delta$. In the case that $\delta\neq\pi\tau-\gamma-\delta$, both $X(z)$ and $W(z)$ have a simple pole at $z=\delta$. Taking a series expansion 
$X(z)^{p}-H_{p}(W(z))=g_{p}(x-\delta)^{-p}+g_{p-1}(x-\delta)^{-p+1}+\cdots+g_{0}+O(x-\delta)$,
around $z=\delta$, for an arbitrary polynomial $H$, the polynomial 
\[H_{p}(w)=h_{p}w^{p}+h_{p-1}w^{p-1}+\cdots+h_{1}w+h_{0}\]
is uniquely defined by the fact that $g_{p}=g_{p-1}=\cdots=g_{0}=0$, as $h_{p}$ is determined by the fact that $g_{p}=0$, then $h_{p-1}$ is determined by the fact that $g_{p-1}=0$ as so on until $h_{0}$ is determined by the fact that $g_{0}=0$. Moreover, by these definitions, the expression $X(z)^{p}-H_{p}(W(z))$ does indeed have a root at $z=\delta$. We also note that $h_p=1$ as $X(z)-W(z)$ does not have a pole at $z=\delta$, so setting $h_{p}=1$ causes $g_{p}=0$. In the remaining case, where $\delta=\pi\tau-\gamma-\delta$, both $X(z)$ and $W(z)$ have a double pole at $z=\delta$ and are fixed under the transformation $z\to 2\delta-z=\pi\tau-\gamma-z$. In this case we can write
\[X(z)^{p}-H_{p}(W(z))=g_{p}(x-\delta)^{-2p}+g_{p-1}(x-\delta)^{-2p+2}+\cdots+g_{0}+O((x-\delta)^{2}),\]
then as in the previous case each $h_{j}$, starting with $j=p$, is determined by the fact that $g_{j}=0$.

Now we have determined the unique polynomial $H_{p}$ such that the right hand side of \eqref{eq:PH_jstart} has a root at $z=\delta$. We will now show that the functions $A_{p}(z)$, $F_{p}$ and $B_{p}(z)$ thus defined are the unique functions satisfying the conditions of Theorem \ref{thm:PH_PV_characterisation}. For the equations, \eqref{eq:PH_PV_nice} holds as it is the sum of \eqref{eq:PV_jstart}, \eqref{eq:B_jstart} and \eqref{eq:PH_jstart} (since $q=0$ in this case), \eqref{eq:PH_flippy_full} and \eqref{eq:PH_pi} hold because $W(z)$ and $X(z)$ are also fixed by these transformations while \eqref{eq:PV_flippy_full} and \eqref{eq:PV_pi} hold because $W(z)$ is fixed by these transformations. For (iii) and (iv): poles of $B_{p}(z)$ could only occur at poles of either $X(z)$ or $W(z)$, and the only poles of these functions in the region $\Omega_{0}\cup\Omega_{1}\cup\Omega_{2}$ are $\delta$ and $\pi\tau-\gamma-\delta$. But we know that $\delta$ is a root of $B_{p}$, and since $B_{p}(z)=\pi\tau-\gamma-z$, this means that $\pi\tau-\gamma-\delta$ is also a root of $B_{p}$. So (iii) and (iv) hold as $B_{p}(z)$ has no poles in $\Omega_{0}\cup\Omega_{1}\cup\Omega_{2}$, and it has roots at the poles of $X(z)$ in this region. Finally (i) holds because $W(z)$ has no poles in $\Omega_{0}\cup\Omega_{-1}\cup\Omega_{-2}$, moreover, (ii) holds because the poles of $Y(z)$ in $\Omega_{-1}\cup\Omega_{-2}$ are roots of $W(z)$, and when $W(z)=0$, we clearly have $A_{p}(z)=0$. this completes the proof that the functions $B_{p}$ and $A_{p}$ determined by this theorem are the unique functions characterised by Theorem \ref{thm:PH_PV_characterisation}.
\end{proof}

For $p\geq 1$, let $\Agf_{p}(\frac{1}{y};t)$, $\Bgf_{p}(\frac{1}{x};t)$ and $\Fgf_{p}(t)$ denote the series $\Agf(\frac{1}{y};t)$, $\Bgf(\frac{1}{x};t)$ and $\Fgf(t)$ arising in the case that the starting point is $(p,0)$.
We can convert the formulae for $A_{p}$ and $B_{p}$ above to formulae for $\Agf_{p}$, $\Bgf_{p}$ using series defined by the following lemma:
\begin{Lemma}\label{lem:WA1}There are unique series $\Wgf_{B}\left(\frac{1}{x};t\right)\in x\mathbb{R}[\frac{1}{x}][[t]]$ and $\Wgf_{A}\left(\frac{1}{y};t\right)\in \frac{1}{y}\mathbb{R}[\frac{1}{y}][[t]]$ satisfying $\Wgf_{B}\left(\frac{1}{X(z)};t\right)=W(z)$
for $z\in\Omega_{1}\cup\Omega_{2}$ and $\Wgf_{A}\left(\frac{1}{Y(z)};t\right)=\Wgf(z)$ for $z\in\Omega_{-1}\cup\Omega_{-2}$. Moreover, these series are related to $\Agf_{1}$ and $\Bgf_{1}$ by $\Wgf_{A}\left(\frac{1}{y};t\right)=\Agf_{1}\left(\frac{1}{y};t\right)$ and $\Wgf_{B}\left(\frac{1}{x};t\right)=x-\Fgf_{1}(t)-\Bgf_{1}\left(\frac{1}{x};t\right)$.\end{Lemma}
\begin{proof}
We define the series $\Wgf_{B}$ and $\Wgf_{A}$ by the equations
\begin{align*}\Wgf_{A}\left(\frac{1}{y};t\right)&=\Agf_{1}\left(\frac{1}{y};t\right)\\
\Wgf_{B}\left(\frac{1}{x};t\right)&=x-\Fgf_{1}(t)-\Bgf_{1}\left(\frac{1}{x};t\right),\end{align*}
and we will prove that these are related to $W(z)$ as described. The reason for defining them in this way is that it is now clear from the definition that $\Wgf_{B}\left(\frac{1}{x};t\right)\in x\mathbb{R}[\frac{1}{x}][[t]]$ and $\Wgf_{A}\left(\frac{1}{y};t\right)\in \frac{1}{y}\mathbb{R}[\frac{1}{y}][[t]]$. For fixed $t$ and $z\in\Omega_{1}\cup\Omega_{2}$, define $W_{1}(z)=\Wgf_{B}\left(\frac{1}{X(z)};t\right)$ and for $z\in\Omega_{-1}\cup\Omega_{-2}$, define $W_{2}(z)=\Wgf_{A}\left(\frac{1}{Y(z)};t\right)$. Then from the definitions of $\Wgf_{B}$ and $\Wgf_{A}$, we have
\begin{align*}
W_{1}(z)&=X(z)-F_{1}-B_{1}\left(z\right)\\
W_{2}(z)&=A_{1}(z).
\end{align*}
Combining this with Theorem \ref{thm:P-xaxis}, using $\Hgf_{1}(w)=w+h_{0}$, where $h_{0}$ depends only on $t$, yields
\begin{align*}
W_{1}(z)&=W(z),\qquad &&\text{for $z\in\Omega_{1}\cup\Omega_{2}$,}\\
W_{2}(z)&=W(z),\qquad &&\text{for $z\in\Omega_{-1}\cup\Omega_{-2}$.}
\end{align*}
Finally the fact that the series $\Wgf_{B}$ and $\Wgf_{A}$ are unique is clear, as it would be impossible for two different series applied to $X(z)$ or $Y(z)$ in the same region to have the same result.  
\end{proof}
We can now rewrite Theorem \ref{thm:P-xaxis} as the following theorem:
\begin{Theorem}\label{thm:p0}
For each $p\geq 1$, then there is a degree $p$ polynomial $H_{p}$ (with coefficients depending on $t$) satisfying
\begin{align}
\Agf_{p}\left(\frac{1}{y}\right)&=H_{p}\left(\Wgf_{A}\left(\frac{1}{y};t\right)\right)-H_{p}(0),\label{eq:AV_jstart}\\
\Fgf(t)&=H_{p}(0),\label{eq:F_jstart}\\
\Bgf_{p}\left(\frac{1}{x}\right)&=x^{p}-H_{p}\left(\Wgf_{B}\left(\frac{1}{x};t\right)\right).\label{eq:AH_jstart}
\end{align}
Moreover, this polynomial is uniquely determined by the fact that the right hand side of \eqref{eq:AH_jstart} is a series in $\frac{1}{x}\mathbb{R}[\frac{1}{x}][[t]]$.
\end{Theorem}
\begin{proof}
We define $H_{p}$ to be the polynomial from Theorem \ref{thm:P-xaxis}. Substituting $y\to Y(z)$ for $z\in\Omega_{-1}\cup\Omega_{-2}$ into \eqref{eq:AV_jstart} yields \eqref{eq:PV_jstart}, so we know that \eqref{eq:AV_jstart} holds for $y=Y(z)$ where $z\in\Omega_{-1}\cup\Omega_{-2}$. This region includes poles of $Y(z)$, so \eqref{eq:PV_jstart} holds for $1/y$ in a neighbourhood of $0$, therefore it must hold as an equation of formal series of $\frac{1}{y}$ for any fixed $t\in\left(0,\frac{1}{\Pgf(1,1)}\right)$, and hence it holds as an equation of series of $t$ and $\frac{1}{y}$, as required. Similarly, substituting $x\to X(z)$ for $z\in\Omega_{1}\cup\Omega_{2}$ into \eqref{eq:AH_jstart} yields \eqref{eq:PH_jstart}, so we know that \eqref{eq:AH_jstart} holds.

Finally the fact that the right hand side of \eqref{eq:AH_jstart} lies in $\frac{1}{x}\mathbb{R}[\frac{1}{x}][[t]]$ is equivalent to the fact that the right hand side of \eqref{eq:PH_jstart} has a root at $\delta$, so it uniquely defines the polynomial $H_{p}$. 
\end{proof}
Theorem \ref{thm:p0} is a purely combinatorial statement, yet our proof involved complex analysis and elliptic functions. It would be nice to understand why this theorem is true from a purely combinatorial perspective. Indeed the following corollary is even more striking in its simplicity, given that we have no purely combinatorial proof.
\begin{figure}[ht]
\centering
   \includegraphics[scale=1]{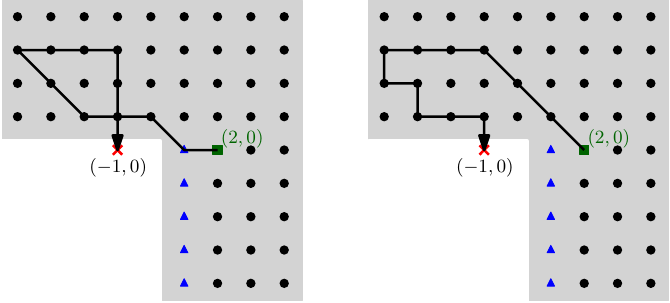}
   \caption{Two paths from $(2,0)$ to $(-1,0)$ using the same steps in a different order and only visiting $(-1,0)$ at their end-point. By Corollary \ref{cor:combi}, exactly half of all such paths pass through a blue triangle.}
   \label{fig:combi_corollary_both}
\end{figure}
   

\begin{Corollary}\label{cor:combi}
Fix a (weighted) step-set $S$. Amongst the (weighted) walks $w$ from $(2,0)$ to $(-1,0)$ of length $n$ using steps in $S$ which only touch the set $\mathcal{D}=\{(x,0):x\leq0\}\cup\{(0,y):y\leq0\}$ at their end-point, exactly half touch the ray $\mathcal{T}=\{(1,y):y\leq0\}$.
\end{Corollary}
\begin{proof}
The walks from $(2,0)$ which only touch $\mathcal{D}$ at their end-point are exactly the walks counted by
\[\Bgf_{2}\left(\frac{1}{x};t\right)+\Agf_{2}\left(\frac{1}{y};t\right)+\Fgf_{2}(t),\]
so the walks from $(2,0)$ to $(-1,0)$ are counted by $[x^{-1}]\Bgf_{2}\left(\frac{1}{x};t\right)$. For the walks from $(2,0)$ to $(-1,0)$ which do not touch $\mathcal{T}$, shifting these walks to the left one space yields exactly the walks from $(1,0)$ to $(-2,0)$; that is, the walks counted by $[x^{-2}]\Bgf_{1}\left(\frac{1}{x};t\right)$. So it suffices to prove that
\[[x^{-1}]\Bgf_{2}\left(\frac{1}{x};t\right)=2[x^{-2}]\Bgf_{1}\left(\frac{1}{x};t\right).\]

From Theorem \ref{thm:p0} and Lemma \ref{lem:WA1}, we have
\begin{equation}
\Bgf_{2}\left(\frac{1}{x}\right)=x^{2}-H_{2}\left(x-\Fgf(t)-\Bgf_{1}\left(\frac{1}{x};t\right)\right),\label{eq:AH2}
\end{equation}
where $H_{2}$ is the unique polynomial of degree $2$, with coefficients depending on $t$, such that the right hand side is a series in $\frac{1}{x}\mathbb{R}[\frac{1}{x}][[t]]$. This is precisely the polynomial
\[H_{2}(w)=(w+\Fgf(t))^{2}+2[x^{-1}]\Bgf_{1}\left(\frac{1}{x};t\right),\]
so \eqref{eq:AH2} becomes
\[\Bgf_{2}\left(\frac{1}{x}\right)
=2x\Bgf_{1}\left(\frac{1}{x};t\right)-\left[x^{0}\right]2x\Bgf_{1}\left(\frac{1}{x};t\right)-\Bgf_{1}\left(\frac{1}{x};t\right)^{2}.\]
Taking the coefficient of $[x^{-1}]$ on both sides of this equation yields the desired result.
\end{proof}
We also give a Second corollary which is perhaps a more explicit version of Theorem \ref{thm:P-xaxis}:
\begin{Corollary}\label{cor:inv_axis_version}
Let $\Igf(w;t)\in w+\mathbb{Z}\left[\frac{1}{w}\right][[t]]$ be the inverse of $\Wgf_{B}(\frac{1}{x};t)=x-\Fgf(t)-\Bgf_{1}(\frac{1}{x};t)$ in the sense that
\begin{equation}\label{eq:Idefinition}\Igf\left(\Wgf_{B}\left(\frac{1}{x};t\right);t\right)=x.\end{equation}
Then
\begin{align*}\Bgf_{p}\left(\frac{1}{x};t\right)&=\left.[w^{<0}]\Igf(w;t)^{p}\right|_{w=W_{B}(\frac{1}{x};t)}\\
\Agf_{p}\left(\frac{1}{y};t\right)&=\left.[w^{>0}]\Igf(w;t)^{p}\right|_{w=\Wgf_{A}(\frac{1}{y};t)}\\
\Fgf_{p}(t)&=[w^{0}]\Igf(w;t)^{p}\end{align*}
and $\Igf(W(z))=X(z)$ for $z\sim\delta$.
\end{Corollary}
\begin{proof}Define
\[H_{p}(w;t):=[w^{\geq 0}] \Igf(w;t)^{p}.\]
We will show that the series $\Agf_{p}$, $\Bgf_{p}$ and $\Fgf_{p}$ as defined above along with $H_{p}$ satisfy the conditions of Theorem \ref{thm:p0}, and therefore are the series defined previously. Since $\Igf(w;t)\in w+\mathbb{Z}\left[\frac{1}{w}\right][[t]]$, we have $\Igf(w;t)^p\in w^p+w^{p-1}\mathbb{Z}\left[\frac{1}{w}\right][[t]]$, so $H_{p}(w;t)=[w^{\geq0}] \Igf(w;t)^{p}$ lies in $\mathbb{R}[[t]][w]$ and has degree $p$ as a polynomial in $w$. Hence 
\begin{align*}\Fgf_{p}(t)&=H_{p}(0;t)\in \mathbb{R}[[t]]&&\text{and}\\
\Agf_{p}\left(\frac{1}{y};t\right)&=H_{p}\left(\Wgf_{A}\left(\frac{1}{y};t\right);t\right)-H_{p}(0;t)\in \frac{1}{y}\mathbb{Z}\left[\frac{1}{y}\right][[t]],&&
\end{align*}
because $\Wgf_{A}\left(\frac{1}{y};t\right)\in \frac{1}{y}\mathbb{Z}\left[\frac{1}{y}\right][[t]]$.
So \eqref{eq:AV_jstart} and \eqref{eq:F_jstart} hold. Moreover, $[w^{<0}] I(w;t)^{p}\in\frac{1}{w}\mathbb{Z}\left[\frac{1}{w}\right][[t]]$, so $\Bgf_{p}\left(\frac{1}{x};t\right)\in \frac{1}{x}\mathbb{Z}\left[\frac{1}{x}\right][[t]]$. Finally, \eqref{eq:AH_jstart} holds because
\[\Bgf_{p}\left(\frac{1}{x};t\right)+H_{p}\left(\Wgf_{B}\left(\frac{1}{x};t\right);t\right)=\Igf\left(\Wgf_{B}\left(\frac{1}{x};t\right);t\right)^{p}=x^{p}.\]
This implies that $H_{p}$ is the unique polynomial defined by Theorem \ref{thm:p0} such that
\[x^{p}-H_{p}\left(\Wgf_{B}\left(\frac{1}{x};t\right)\right)\in \frac{1}{x}\mathbb{Z}\left[\frac{1}{x}\right][[t]],\]
which implies that $\Agf_{p}$, $\Bgf_{p}$ and $\Fgf_{p}$ are the desired series as well.

Finally, for $t$ fixed, if $z\in \Omega_{1}\cup\Omega_{2}$ and $\frac{1}{W(z)}$ is within the radius of convergence of $\Igf(w;t)$ as a function of $\frac{1}{w}$, then we can substitute $x\to X(z)$ into \eqref{eq:Idefinition} yielding $\Igf(W(z))=X(z)$. In particular, this will happen in a neighbourhood of $\delta$ as $\delta\in\Omega_{1}\cup\Omega_{2}$ and $\delta$ is a pole of $W(z)$ (see definition \ref{def:W}). 
\end{proof}

\section{Nature of series in the three-quadrant cone}\label{sec:nature}
In the Section \ref{sec:functional_equations}, we reduced the problem of enumerating walks in the three-quadrant cone to finding the unique meromorphic functions ${A},{B}:\mathbb{C}\to\mathbb{C}\cup\{\infty\}$ and constant $F$ characterised by Theorem \ref{thm:PH_PV_characterisation} (for each $t$). As we discussed, an equation analogous to \eqref{eq:PH_PV_nice} was found by Raschel for walks in the quarter plane \cite{raschel2012counting}, and this equation has since been used to determine precisely when $\Qgf(x,y;t)$ is differentially algebraic \cite{bernardi2017counting,dreyfus2018nature,hardouin2020differentially} and to determine in many cases whether it is algebraic or D-finite with respect to $x$ or $y$ \cite{fayolle2010holonomy,kurkova2012functions}.  

Due to the similarity between our functional equation \eqref{eq:PH_PV_nice} and the equation widely used in the quarter plane, we can apply methods that have been used on the quarter plane functional equation to our functional equation to determine the nature of $\Cgf(x,y;t)$. In particular, in this section we will show that $\Cgf(x,y;t)$ is algebraic, D-finite or D-algebraic with respect to $x$ (or $y$) in the same cases as $\Qgf(x,y;t)$. We note that Fayolle and Raschel also showed that for unweighted models, $\Qgf(x,y;t)$ is D-finite with respect to $t$ in the cases where it is D-finite with respect to $x$ \cite{fayolle2010holonomy}, however these results relied on the precise ratios $\frac{\pi\tau}{\gamma}$ that could occur in these cases, so they do not apply so readily to our equation. Nonetheless, we expect that the same result holds for $\Cgf(x,y;t)$.

There are two properties of the step-set which determine the complexity of the generating function $\Qgf(x,y;t)$. The first is the property that the walk model has a {\em finite group} - in our context this is equivalent to the ratio $\frac{\gamma}{\pi\tau}$ being a rational number independent of $t$ (see Appendix \ref{ap:group}). The second is the property that the model {\em decouples}, that is, under the assumption that $K(x,y)=0$, we can write $x^p y^q=R_{1}(x)+R_{2}(y)$ for some rational functions $R_{1}$ and $R_{2}$. Equivalently these are the cases where one can write
$X(z)^p Y(z)^q=R_{1}(X(z))+R_{2}(Y(z))$ for some rational functions $R_{1}$ and $R_{2}$. The nature of the series $\Cgf(x,y;t)$ as determined by these properties is shown in Table \ref{table:characterisation}. We make this more precise in the following Theorems which will be proved in this section.

\begin{table}[h!]
\begin{tabular}{|c|c|c|}
\hline 
&\textbf{Finite group}&\textbf{Infinite group}\\ 
\hline 
\textbf{Decoupling}&Algebraic&D-algebraic, not D-finite\\ 
\hline 
\textbf{Non-decoupling}&D-finite, not algebraic& not D-algebraic \\ 
\hline 
\end{tabular} 
\medskip
\caption{Complexity of $\Qgf(x,y;t)$ and $\Cgf(x,y;t)$ as functions of $x$ as proven more precisely by Theorems \ref{thm:D-finite_fixed_t}-\ref{thm:D-algebraic_general_t}.}
\label{table:characterisation}
\end{table}
In the following theorems we reference the {\em functions} $\Agf$, $\Bgf$ and $\Cgf$, that is, the functions defined by the series at values of $t,x,y$ where they converge absolutely, as described in Lemma \ref{lem:Q_series_convergence}. Importantly, for $t\in\left(0,\frac{1}{\Pgf(1,1)}\right)$ all of these series converge absolutely to functions of $x$ and $y$ defined for $x,y$ in an open, non-empty subset of $\mathbb{C}$, so it makes sense to discuss the nature of these functions.

\begin{Theorem}\label{thm:D-finite_fixed_t}
For fixed $t\in\left(0,\frac{1}{\Pgf(1,1)}\right)$ the following are equivalent
\begin{enumerate}[label={\rm(\roman*)},ref={\rm(\roman*)}]
\item The function $\Cgf(x,y;t)$  is D-finite in $x$, \label{D-finitex}
\item The function $\Cgf(x,y;t)$  is D-finite in $y$, \label{D-finitey}
\item The function $\Bgf(\frac{1}{x};t)$  is D-finite in $x$, \label{D-finiteAx}
\item The function $\Agf(\frac{1}{y};t)$  is D-finite in $y$, \label{D-finiteAy}
\item ${B}(z)$ satisfies a linear differential equation whose coefficients are elliptic functions with periods $\pi$ and $\pi\tau$, \label{D-finiteHz}
\item ${A}(z)$ satisfies a linear differential equation whose coefficients are elliptic functions with periods $\pi$ and $\pi\tau$, \label{D-finiteVz}
\item the ratio $\frac{\gamma}{\pi\tau}\in\mathbb{Q}$, \label{D-finiterational}
\item the orbit of each point $(x,y)\in\overline{E_{t}}$ under the group of the walk is finite. \label{D-finiteorbits}
\end{enumerate}
\end{Theorem}
We give the proof of some of these equivalences immediately, namely those which are either simple to prove, or follow easily from our general results in Appendix \ref{ap:nature_analytic}:
\begin{proof}[Proof of $\ref{D-finitex}\iff\ref{D-finitey}\iff\ref{D-finiteAx}\iff\ref{D-finiteAy}\iff\ref{D-finiteHz}\iff\ref{D-finiteVz}$]
The equivalences $\ref{D-finitex}\iff\ref{D-finiteAx}$ and $\ref{D-finitey}\iff\ref{D-finiteAy}$ follow from \eqref{eq:three_quarter_full}. The equivalences $\ref{D-finiteAx}\iff\ref{D-finiteHz}$ and $\ref{D-finiteAy}\iff\ref{D-finiteVz}$ follow from Proposition \ref{prop:D-finiteofXorY} due to the definitions \eqref{eq:PVdefinition} and \eqref{eq:PHdefinition} of $A(z)$ and $B(z)$. Finally $\ref{D-finiteHz}\iff\ref{D-finiteVz}$ due to \eqref{eq:PH_PV_nice} as $X(z)^{p}Y(z)^{q}$ has $\pi\tau$ as a period.
\end{proof}
The proof of this theorem will be completed as follows: we define the group of the walk in Appendix \ref{ap:group}, and the equivalence of \ref{D-finiterational} and \ref{D-finiteorbits} is shown in Proposition \ref{prop:finite_group_fixed_t}. We show that these equivalent conditions imply the conditions \ref{algx}-\ref{algVz} in Theorem \ref{thm:finite_group_case_thm}, while we show the converse in Theorem \ref{thm:inf_group_non-D-finite}.

\begin{Theorem}\label{thm:algebraic_fixed_t}
Assume that $t\in\left(0,\frac{1}{\Pgf(1,1)}\right)$. The following are equivalent
\begin{enumerate}[label={\rm(\roman*)},ref={\rm(\roman*)}]
\item The function $\Cgf(x,y;t)$  is algebraic in $x$, \label{algx}
\item The function $\Cgf(x,y;t)$  is algebraic in $y$, \label{algy}
\item The function $\Bgf(\frac{1}{x};t)$  is algebraic in $x$, \label{algAx}
\item The function $\Agf(\frac{1}{y};t)$  is algebraic in $y$, \label{algAy}
\item ${B}(z)$ has $m\pi\tau$ as a period for some positive integer $m$, \label{algHz}
\item ${A}(z)$ has $m\pi\tau$ as a period for some positive integer $m$, \label{algVz}
\item The equivalent conditions of Theorem \ref{thm:D-finite_fixed_t} hold and there are rational functions $R_{1}$ and $R_{2}$ satisfying $x^{p}y^{q}=R_{1}(x)+R_{2}(y)$ for all $(x,y)\in\overline{E_{t}}$, \label{algdec}
\item The equivalent conditions of Theorem \ref{thm:D-finite_fixed_t} hold and there are rational functions $R_{1}$ and $R_{2}$ satisfying $X(z)^{p}Y(z)^{q}=R_{1}(X(z))+R_{2}(Y(z))$ for all $z\in\mathbb{C}$, \label{algdecz}
\item The equivalent conditions of Theorem \ref{thm:D-finite_fixed_t} hold and the orbit sum of the model is $0$ (for $(x,y)\in\overline{E_{t}}$). \label{algorbitsum}
\end{enumerate}
\end{Theorem}
\begin{proof}[Proof of $\ref{algx}\iff\ref{algy}\iff\ref{algAx}\iff\ref{algAy}\iff\ref{algHz}\iff\ref{algVz}$ and $\ref{algdec}\iff\ref{algdecz}$]
The equivalence $\ref{algdec}\iff\ref{algdecz}$ is due to the parameterisation \eqref{eq:Eparam} of $\overline{E_{t}}$. The equivalences $\ref{algx}\iff\ref{algAx}$ and $\ref{algy}\iff\ref{algAy}$ follow from \eqref{eq:three_quarter_full}. The equivalences $\ref{algAx}\iff\ref{algHz}$ and $\ref{algAy}\iff\ref{algVz}$ follow from Proposition \ref{prop:algebraicofXorY} due to the definitions \eqref{eq:PVdefinition} and \eqref{eq:PHdefinition}. Finally $\ref{algHz}\iff\ref{algVz}$ due to \eqref{eq:PH_PV_nice} as $X(z)^{p}Y(z)^{q}$ has $m\pi\tau$ as a period.
\end{proof}
To complete the proof of this Theorem, we define the orbit sum and show that \ref{algdecz} and \ref{algorbitsum} are equivalent conditions in Proposition \ref{prop:orbit_sum0isdecoupling}. We then show that equivalent conditions \ref{algx}-\ref{algVz} are equivalent to \ref{algdec}-\ref{algorbitsum} in Theorem \ref{thm:finite_group_case_thm}.

\begin{Theorem}\label{thm:D-algebraic_fixed_t}
Fix $t\in\left(0,\frac{1}{\Pgf(1,1)}\right)$ and assume the equivalent conditions of Theorem \ref{thm:D-finite_fixed_t} do not hold. The following are equivalent
\begin{enumerate}[label={\rm(\roman*)},ref={\rm(\roman*)}]
\item The function $\Cgf(x,y;t)$ is D-algebraic in $x$, \label{D-algx}
\item The function $\Cgf(x,y;t)$  is D-algebraic in $y$, \label{D-algy}
\item The function $\Bgf(\frac{1}{x};t)$  is D-algebraic in $x$, \label{D-algAx}
\item The function $\Agf(\frac{1}{y};t)$  is D-algebraic in $y$, \label{D-algAy}
\item ${B}(z)$ is D-algebraic in $z$, \label{D-algHz}
\item ${A}(z)$ is D-algebraic in $z$, \label{D-algVz}
\item There are rational functions $R_{1}$ and $R_{2}$ satisfying $x^{p}y^{q}=R_{1}(x)+R_{2}(y)$ for all $(x,y)\in\overline{E_{t}}$, \label{D-algdec}
\item There are rational functions $R_{1}$ and $R_{2}$ satisfying $X(z)^{p}Y(z)^{q}=R_{1}(X(z))+R_{2}(Y(z))$ for all $z\in\mathbb{C}$. \label{D-algdecz}
\end{enumerate}
\end{Theorem}
\begin{proof}[Proof of $\ref{D-algx}\iff\ref{D-algy}\iff\ref{D-algAx}\iff\ref{D-algAy}\iff\ref{D-algHz}\iff\ref{D-algVz}$ and $\ref{D-algdec}\iff\ref{D-algdecz}$]
The equivalence $\ref{D-algdec}\iff\ref{D-algdecz}$ is due to the parameterisation \eqref{eq:Eparam} of $\overline{E_{t}}$. The equivalences $\ref{D-algx}\iff\ref{D-algAx}$ and $\ref{D-algy}\iff\ref{D-algAy}$ follow from \eqref{eq:three_quarter_full}, the equivalences $\ref{D-algAx}\iff\ref{D-algHz}$ and $\ref{D-algAy}\iff\ref{D-algVz}$ follow from Proposition \ref{prop:D-algebraicofXorY} due to the definitions \eqref{eq:PVdefinition} and \eqref{eq:PHdefinition}. Finally $\ref{D-algHz}\iff\ref{D-algVz}$ due to \eqref{eq:PH_PV_nice} as $X(z)^{p}Y(z)^{q}$ is D-algebraic in $z$ (see again the proof of Proposition \ref{prop:D-algebraicofXorY}).
\end{proof}
We complete the proof of this theorem later in this section, starting with Theorem \ref{thm:D-alg_proof} which shows that the equivalent conditions $\ref{D-algdec}\iff\ref{D-algdecz}$ imply the equivalent conditions $\ref{D-algx}\iff\ref{D-algy}\iff\ref{D-algAx}\iff\ref{D-algAy}\iff\ref{D-algHz}\iff\ref{D-algVz}$, then we show the reverse implication in Theorem \ref{thm:non-D-alg_proof}.

The theorems above describe the nature of $\Cgf(x,y;t)$ as a function of $x$ and $y$ for fixed $t$. we can use this to determine its nature as series in $\mathbb{R}[x,y][[t]]$ using the following Lemma:
\begin{Lemma}\label{lem:function_and_series_complexity} Let $\Ggf(x,y;t)\in \mathbb{R}[x,\frac{1}{x},y,\frac{1}{y}][[t]]$ be a series which converges for $|x|,|y|=1$ and $t\in\left(0,\frac{1}{\Pgf(1,1)}\right)$. The series $\Ggf(x,y;t)$ is algebraic (resp. D-finite, D-algebraic) in $x$ if and only if the function $\Ggf(x,y;t)$ of $x$ and $y$ is algebraic (resp. D-finite, D-algebraic) in  $x$ for all $t\in\left(0,\frac{1}{\Pgf(1,1)}\right)$.
\end{Lemma}
\begin{proof}
Let $\Lambda_{1}(x,y,t),\Lambda_{2}(x,y,t),\Lambda_{3}(x,y,t),\ldots$ be an ordering of the set $\{x^{j}\Ggf(x,y;t)^k\}_{j,k\in\mathbb{N}_{0}}$. Then the function $\Ggf(x,y;t)$ is algebraic in $x$ if and only if for each $t$, there is some integer $n>0$ and functions $S_{k}(y;t)$ for $0\leq k\leq n$ and $S(y,t)$ of $y$ satisfying 
\[\sum_{k=0}^{n}S_{k}(y;t)\Lambda_{k}(x,y,t)=S(y,t),\]
with $S_{n}(y;t)\neq0$. A priori, $n$ depends on $t$, however, as there are uncountably many possible values of $t$ and only countably many for $n$, we can chose some $n=N$ for which an equation of the form above holds for uncountably many values of $t$. For a given value of $t$, this happens if and only if for every fixed $x_{0},\ldots,x_{n+1}$, the matrix
\[\begin{bmatrix}1&\Lambda_{0}(x_{0},y,t)&\Lambda_{1}(x_{0},y,t)&\cdots&\Lambda_{n}(x_{0},y,t)\\
1&\Lambda_{0}(x_{1},y,t)&\Lambda_{1}(x_{1},y,t)&\cdots&\Lambda_{n}(x_{1},y,t)\\
1&\Lambda_{0}(x_{2},y,t)&\Lambda_{1}(x_{2},y,t)&\cdots&\Lambda_{n}(x_{2},y,t)\\
\vdots&\vdots&\vdots&\ddots&\vdots\\
1&\Lambda_{0}(x_{n+1},y,t)&\Lambda_{1}(x_{n+1},y,t)&\cdots&\Lambda_{n}(x_{n+1},y,t)\\
\end{bmatrix}\]
has determinant $0$. Since each $\Ggf(x_{k},y;t)$ can be considered to be a series in $\mathbb{C}[y][[t]]$, this determinant, which we denote by $T(y;t)$ is in general a series in $\mathbb{C}[y][[t]]$. Moreover, for any value of $t$ and $y$ for which all of the series $\Lambda_{j}(x_{k},y,t)$ converge absolutely, the series $T(y;t)$ will also converge absolutely to the determinant of these values.

Now, assume that the function $\Ggf(x,y;t)$ is algebraic in $x$ for all $t$. For all sufficiently small $t$, and all $y$ satisfying $|y|=1$, the series $\Lambda_{j}(x_{k},y,t)$ converge absolutely, so there are uncountably many values $t$ such that the series $T(y;t)$ converges to $0$ when $|y|=1$. But this is only possible if $T(y;t)=0$ as a series. Now, since this determinant is $0$, an equation of the form
\[\sum_{k=0}^{n}S_{k}(y;t)\Lambda_{k}(x,y,t)=S(y,t),\]
must hold with each $S_{k}(y;t)$ and $S(y,t)$ a series in $\mathbb{C}[y][[t]]$, so the series $\Ggf(x,y;t)$ is algebraic in $x$.

For the converse, assume that the series $\Ggf(x,y;t)$ is algebraic in $x$. Then for some $n$, and any $x_{0},\ldots,x_{n+1}\in\mathbb{C}$, the determinant $T(y;t)=0$. Hence for each fixed $t$ the determinant is still $0$, so the function $\Ggf(x,y;t)$ is algebraic.

For the property D-finite, the same proof works after changing the definition of $\Lambda_{1}(x,y,t), \Lambda_{2}(x,y,t), \ldots$ to be the functions $x^{j}\left(\frac{\partial}{\partial x}\right)^{k}\Ggf(x,y;t)$ in some order. 

Similarly, for the property D-algebraic, we just have to define $\Lambda_{k}(x,y,t)$ so that the sequence
$\Lambda_{0}(x,y,t),\Lambda_{1}(x,y,t),\ldots$ contains each product $x^{j}\prod_{i=1}^{p}\left(\frac{\partial}{\partial x}\right)^{k_{i}}\Fgf(x,y;t)$ with $p\geq 0$ and $0\leq k_{1}\leq\cdots\leq k_{p}$ exactly once.
 \end{proof}
 
 We now use this lemma to rewrite the theorems above to characterise the complexity of the series $\Cgf(x,y;t)\in\mathbb{R}[x,y][[t]]$:
 
\begin{Theorem}\label{thm:D-finite_general_t}
The following are equivalent
\begin{enumerate}[label={\rm(\roman*)},ref={\rm(\roman*)}]
\item The series $\Cgf(x,y;t)\in\mathbb{R}[x,y][[t]]$  is D-finite in $x$, \label{series_D-finitex}
\item The series $\Cgf(x,y;t)\in\mathbb{R}[x,y][[t]]$  is D-finite in $y$, \label{series_D-finitey}
\item The equivalent conditions of Theorem \ref{thm:D-finite_fixed_t} hold for all $t\in\left(0,\frac{1}{\Pgf(1,1)}\right)$,\label{equivs_D-finite}
\item The group of the walk is finite.\label{finite_group}
\end{enumerate}
\end{Theorem}

\begin{Theorem}\label{thm:algebraic_general_t}
The following are equivalent
\begin{enumerate}[label={\rm(\roman*)},ref={\rm(\roman*)}]
\item The series $\Cgf(x,y;t)\in\mathbb{R}[x,y][[t]]$  is algebraic in $x$, \label{series_algx}
\item The series $\Cgf(x,y;t)\in\mathbb{R}[x,y][[t]]$  is algebraic in $y$, \label{series_algy}
\item The equivalent conditions of Theorem \ref{thm:algebraic_fixed_t} hold for all $t\in\left(0,\frac{1}{\Pgf(1,1)}\right)$.
\end{enumerate}
\end{Theorem}
 
\begin{Theorem}\label{thm:D-algebraic_general_t}
Assume that the group of the walk is infinite. The following are equivalent
\begin{enumerate}[label={\rm(\roman*)},ref={\rm(\roman*)}]
\item The series $\Cgf(x,y;t)\in\mathbb{R}[x,y][[t]]$  is D-algebraic in $x$, \label{series_D-algx}
\item The series $\Cgf(x,y;t)\in\mathbb{R}[x,y][[t]]$  is D-algebraic in $y$, \label{series_D-algy}
\item The equivalent conditions of Theorem \ref{thm:D-algebraic_fixed_t} hold for all $t\in\left(0,\frac{1}{\Pgf(1,1)}\right)$.
\end{enumerate}
\end{Theorem}
\begin{proof}[Proof of Theorems \ref{thm:D-finite_general_t}-\ref{thm:D-algebraic_general_t}]
For each of these theorems, the equivalences $\ref{series_D-finitex}\iff\ref{equivs_D-finite}$ and $\ref{series_D-finitey}\iff\ref{equivs_D-finite}$ are both direct results of Lemma \ref{lem:function_and_series_complexity} with $\Fgf(x,y;t)=\Cgf(x,y;t)$. Finally, the equivalence $\ref{equivs_D-finite}\iff\ref{finite_group}$ in Theorem \ref{thm:D-finite_general_t} is due to Proposition \ref{prop:finite_group_general_t}.
\end{proof}

\subsection{Finite group cases}\label{sec:finite_group}
In this section, we consider the cases in which $\frac{\gamma}{\pi\tau}\in\mathbb{Q}$. We will show that this occurs for fixed $t$ if and only if $\Cgf(x,y;t)$ is D-finite in $x$. As explained in Appendix \ref{ap:group} this occurs for all $t$ if and only if the group of the walk is finite. This is an algebraic property of the single-step generating function $\Pgf(x,y)$, which was shown to be equivalent to D-finiteness for walks in the quarter plane. Hence in this section we are showing that if the generating function $\Qgf(x,y;t)$ is D-finite in $x$ then the generating function $\Cgf(x,y;t)$ is also D-finite in $x$. In the following section we will show that the converse of this statement also holds. Note that by symmetry between $x$ and $y$, the analogous statements with respect to $y$ also hold. We also show that $\Cgf(x,y;t)$ has the more restricted property that it is algebraic in $x$ precisely when $X(z)^{p}Y(z)^{q}$ decouples, which we make precise below. Again this coincides with the nature of $\Qgf(x,y;t)$.

\begin{Definition}\label{defn:decoupling} For fixed $t$, we say that an elliptic function $U(z)$, with periods $\pi$ and $\pi\tau$ {\em decouples} if there is a pair of rational functions $R_{1}$ and $R_{2}$ satisfying
\[U(z)=R_{1}(X(z))+R_{2}(Y(z)).\]
We say that the model decouples if $X(z)^p Y(z)^q$ decouples.
\end{Definition}

It is often alternatively stated that the algebraic cases are those in which the orbit sum is $0$, which we make precise in the following definition.

\begin{Definition}\label{defn:orbit_sum} For fixed $t$ satisfying $\frac{2\gamma}{\pi\tau}=\frac{M}{N}$, for integers $M,N>0$, the {\em orbit sum} $E(z)$ of an elliptic function $U(z)$, with periods $\pi$ and $\pi\tau$, is given by
\[E(z):=\sum_{k=0}^{N-1}U((2k+1)\gamma-z)-U(2k\gamma+z).\]
We say that the orbit sum of the model is the orbit sum of $X(z)^p Y(z)^q$.
\end{Definition}

 In the following Proposition we show that the orbit sum of a function $U(z)$ is $0$ if and only if $U(z)$ decouples. This was proven in an algebraic setting in \cite[Theorem 4.11]{bernardi2017counting}, where decoupling functions were introduced, and essentially the same proof works here:
\begin{Proposition}\label{prop:orbit_sum0isdecoupling}
Let $U(z)$ be an elliptic function with periods $\pi$ and $\pi\tau$, and assume that $\frac{2\gamma}{\pi\tau}=\frac{M}{N}$, with $M,N\in\mathbb{Z}$, $N>0$. The function $U(z)$ decouples if and only if the orbit sum $E(z)$ of $U(z)$ is equal to $0$.
\end{Proposition}
\begin{proof}
In the case that $U(z)$ decouples, let $U(z)=R_{1}(X(z))+R_{2}(Y(z))$, where $R_{1}$ and $R_{2}$ are rational functions. Then we can write the orbit sum as
\begin{align*}E(z)&=\sum_{k=0}^{N-1}R_{1}(X((2k+1)\gamma-z))+R_{2}(Y((2k+1)\gamma-z))-R_{1}(X(2k\gamma+z))-R_{2}(Y(2k\gamma+z))\\
&=\sum_{k=0}^{N-1}R_{1}(X(-2(k+1)\gamma+z))+R_{2}(Y(-2k\gamma+z))-R_{1}(X(2k\gamma+z))-R_{2}(Y(2k\gamma+z)).\end{align*}
Now, the terms $R_{1}(X(2k\gamma+z))$ in the sum are a permutation of the terms $R_{1}(X(-2(k+1)\gamma+z))$, because
\[R_{1}(X(2k\gamma+z))=R_{1}(X(-2(j+1)\gamma+z)),\]
when $j=N-k-1$. Hence these terms cancel out in the sum. The same holds for the remaining terms in the sum as $R_{2}(Y(2k\gamma+z))=R_{2}(Y(-2j\gamma+z))$ for $j=N-k$ or $j=k=0$, so we have $E(z)=0$, as required.

We will now prove the converse, that is that if $E(z)=0$ then $U(z)$ decouples. Define
\begin{align*}
A_{1}(z)&:=\sum_{k=1}^{N}\frac{2k-1}{2N}\left(U(2k\gamma+z)+U((2k-1)\gamma-z)\right),\\
A_{2}(z)&:=\sum_{k=0}^{N-1}-\frac{k}{N}(U(2k\gamma+z)+U((2k+1)\gamma-z)).
\end{align*}
The summand $U(2k\gamma+z)+U((2k-1)\gamma-z)$ is fixed under the transformation $z\to -\gamma-z$ for any $k$, so we have $A_{1}(-\gamma-z)=A_{1}(z)$. Similarly we have $A_{2}(\gamma-z)=A_{2}(z)$. So by Proposition \ref{prop:rationalofXorY}, there are rational functions $R_{1}$ and $R_{2}$ satisfying $A_{1}(z)=R_{1}(X(z))$ and $A_{2}(z)=R_{2}(Y(z))$, so it suffices to show that if $E(z)=0$, then $U(z)=A_{1}(z)+A_{2}(z)$. Indeed this follows from the equation
\[
A_{1}(z)+A_{2}(z)=U(z+2N\gamma)+\frac{1}{2N}E(z),\]
which follows directly from the definitions of $A_{1}$, $A_{2}$ and $E$.\end{proof}

\begin{Lemma}\label{lem:finite_group} If $\frac{2\gamma}{\pi\tau}=\frac{M}{N}\in\mathbb{Q}$, then the {\em orbit sum}
\[E(z):=\sum_{j=0}^{N-1}X((2j+1)\gamma-z)^{p}Y((2j+1)\gamma-z)^{q}-X(2j\gamma+z)^{p}Y(2j\gamma+z)^{q}\]
is related to $B(z)$ by
\begin{equation}\label{eq:PHandE}{B}((2N-M)\pi\tau+z)-{B}(z)=E(z).\end{equation}
\end{Lemma}
\begin{proof}Assume that $2\gamma=\frac{M}{N}\pi\tau$ for some positive $M,N\in\mathbb{Z}$. Now consider \eqref{eq:PH_qdiff}:
\[{B}\left(\frac{2N-M}{N}\pi\tau+z\right)-{B}(z)=J(z).\]
 Taking a telescoping sum of $N$ copies of this equation yields
\[{B}((2N-M)\pi\tau+z)-{B}(z)=\sum_{j=0}^{N-1}J(2j\pi\tau-2j\gamma+z),\]
which we claim is equal to the orbit sum $E(z)$. Indeed, since $\pi\tau$ is a period of $J(z)$, we have $J(2j\pi\tau-2j\gamma+z)=J(-2j\gamma+z)=J(2(N-j)\gamma+z)$, so the sum rearranges to
\[\sum_{j=0}^{N-1}J(2j\gamma+z),\]
which is equal to $E(z)$. Hence
\[{B}((2N-M)\pi\tau+z)-{B}(z)=E(z),\]
as required.
\end{proof}

\begin{Theorem}\label{thm:finite_group_alg_thm} If $\frac{2\gamma}{\pi\tau}=\frac{M}{N}\in\mathbb{Q}$, then the orbit sum $E(z)$ defined in Lemma \ref{lem:finite_group} is equal to $0$ if and only if $\Cgf(x,y)$ is algebraic in $x$.\end{Theorem}
\begin{proof}If $E(z)=0$, then from \eqref{eq:PHandE}, we have
\[{B}((2N-M)\pi\tau+z)={B}(z).\]
Hence we have condition \ref{algHz} of Theorem \ref{thm:algebraic_fixed_t}, so the equivalent conditions are satisfied, including that $\Cgf(x,y)$ is algebraic in $x$.

For the other direction, assume that $\Cgf(x,y)$ is algebraic in $x$. Then again by condition \ref{algHz} of Theorem \ref{thm:algebraic_fixed_t}, there is some positive integer $m$ such that $m\pi\tau$ is a period of ${B}(z)$. It then follows from \eqref{eq:PHandE} that
\[0={B}(m(2N-M)\pi\tau+z)-{B}(z)=mE(z),\]
hence $E(z)=0$, as required.
\end{proof}

\begin{Theorem}\label{thm:finite_group_case_thm} If $\frac{2\gamma}{\pi\tau}=\frac{M}{N}\in\mathbb{Q}$, then $\Cgf(x,y)$ is D-finite in $x$.\end{Theorem}
\begin{proof}
In the case $E(z)=0$, Theorem \ref{thm:finite_group_alg_thm} implies that $\Cgf(x,y)$ is algebraic, and hence D-finite. In the case $E(z)\neq0$, we use \eqref{eq:PHandE}, which yields
\begin{equation}\label{eq:PH_holo}\frac{{B}((2N-M)\pi\tau+z)}{E((2N-M)\pi\tau+z)}-\frac{{B}(z)}{E(z)}=\frac{{B}((2N-M)\pi\tau+z)}{E(z)}-\frac{{B}(z)}{E(z)}=1,\end{equation}
so the function
\begin{equation}\label{eq:F_def}F(z):=\frac{\partial}{\partial z}\frac{{B}(z)}{E(z)}=\frac{1}{E(z)^{2}}({B}'(z)E(z)-{B}(z)E'(z))\end{equation}
satisfies
\[F((2N-M)\pi\tau+z)-F(z)=0.\]
Hence, ${B}(z)=\Bgf\left(\frac{1}{X(z)}\right)$ is weakly $X$-D-finite (see Definition \ref{def:X-D-finite-weak}), so by Proposition \ref{prop:D-finiteofXorY}, the function $\Bgf(\frac{1}{x})$ is D-finite in $x$. Therefore the generating function $\Cgf(x,y)$ is also D-finite in $x$.
\end{proof}

\textbf{Remark:} Note that this Theorem holds for any fixed $t$ for which $\frac{\gamma}{\pi\tau}\in\mathbb{Q}$, which can occur either in the finite group case, where this occurs for all $t$, or in infinite group cases for specific values of $t$. We require that the above statement hold for all $t$ to say that the function $\Cgf(x,y;t)$ is D-finite in $x$. See Appendix \ref{ap:group} for a discussion of the group of the walk.

\subsection{Infinite group cases}

In this section we consider the case $\frac{\gamma}{\pi\tau}\notin\mathbb{Q}$. The first part of this section is dedicated to showing that $\Cgf(x,y;t)$ is not D-finite in $x$ in these cases, and we will subsequently analyse the D-algebraicity of $\Cgf(x,y;t)$. We start with a lemma which essentially proves $\Bgf(\frac{1}{x};t)$ is not rational in $x$, as this turns out to be a case which needs to be treated separately. Surprisingly this seems to be the most difficult result of this section, in the sense that it is the only result for which our proof does not apply systematically to walks in an $M$-quadrant cone for any $M$. In particular, for walks in the quadrant we have only able to prove the result when $t$ is sufficiently small (See Lemma \ref{lem:not-rationalM1}).

\begin{Lemma}\label{lem:not-rational}
Assume that $\frac{\gamma}{\pi\tau}\notin\mathbb{Q}$. Then ${B}(z)$ is not a rational function of $X(z)$. 
\end{Lemma}
\begin{proof}
Assume that ${B}(z)$ is a rational function of $X(z)$. Then for $z\in\Omega_{-2}$, we have $z+\pi\tau\in\Omega_{2}$, so, by \eqref{eq:PHdefinition},
\[{B}(z+\pi\tau)=\Bgf\left(\frac{1}{X(z+\pi\tau)};t\right).\]
Now since $\pi\tau$ is a period of both $X$ and ${B}$, this implies that
\[{B}(z)=\Bgf\left(\frac{1}{X(z)};t\right),\]
an equation that would normally only hold for $z\in\Omega_{0}\cup\Omega_{1}\cup\Omega_{2}.$
By \eqref{eq:PVdefinition}, we have
\[{A}(z)=\Agf\left(\frac{1}{Y(z)};t\right).\]
We will show that this is a contradiction as the sum of these cannot be sufficiently large in absolute value to satisfy \eqref{eq:PH_PV_nice}.

Recall that by our choice $t\in(0,\frac{1}{\Pgf(1,1)})$, the series $\Cgf(x,y;t)$, $\Bgf(\frac{1}{x};t)$ and $\Agf(\frac{1}{y};t)$ all converge when $|x|,|y|=1$. In particular, we can substitute $x=y=1$ into \eqref{eq:three_quarter_full}, which yields
\[(1-t\Pgf(1,1))\Cgf(1,1;t)=1-\Fgf(t)-\Agf(1;t)-\Bgf(1;t).\]
Since the left hand side of this equation in positive, we must have
\[\Fgf(t)+\Agf(1;t)+\Bgf(1;t)<1.\]
It follows that for $x,y$ satisfying $|x|,|y|\geq 1$ we have
\[\left|\Fgf(t)+\Agf\left(\frac{1}{y};t\right)+\Bgf\left(\frac{1}{x};t\right)\right|<1.\]
In particular, for $z\in\Omega_{-2}$, we have $|X(z)|,|Y(z)|\geq 1$, so
\[|F+{B}(z)+{A}(z)|=\left|\Fgf(t)+\Bgf\left(\frac{1}{X(z)};t\right)+\Agf\left(\frac{1}{Y(z)};t\right)\right|<1.\]
However this is a contradiction as, by \eqref{eq:PH_PV_nice}, we have
\[|F+{B}(z)+{A}(z)|=|X(z)^{p}Y(z)^{q}|=|X(z)|^{p}|Y(z)|^{q},\]
and $|X(z)|,|Y(z)|\geq 1$ because $z\in\Omega_{-2}$.
\end{proof}

Now we are ready to prove that $\Bgf(\frac{1}{x};t)$ is not D-finite in $x$ in the case that $\frac{\gamma}{\pi\tau}\notin\mathbb{Q}$. The idea of the proof is that if $\Bgf(\frac{1}{x};t)$ is D-finite in $x$, then the poles of $B(z)$ must be well-behaved, as described in Lemma \ref{lem:D-fini_z}, whereas if $\frac{\gamma}{\pi\tau}\notin\mathbb{Q}$, we can prove that its poles are not well behaved in this way.

\begin{Theorem}\label{thm:inf_group_non-D-finite}
Assume that $\frac{\gamma}{\pi\tau}\notin\mathbb{Q}$. Then $\Bgf(\frac{1}{x};t)$ is not D-finite in $x$.
\end{Theorem}  
\begin{proof}
Suppose the contrary. Then $\Bgf(\frac{1}{x};t)$ is D-finite in $x$. Moreover, recall that ${B}(z)=\Bgf(\frac{1}{X(z)};t)$ for $z\in\Omega_{1}\cup\Omega_{2}$, so by Lemma \ref{lem:D-fini_z}, the poles $z_{c}$ of ${B}(z)$ fall into only finitely many classes $z_{c}+\pi\mathbb{Z}+\pi\tau\mathbb{Z}$. Hence ${B}(\pi\tau+z)-{B}(z)$ has the same property.

Now, from \eqref{eq:PH_qdiff}, we have
\[{B}(2\pi\tau-2\gamma+z)-{B}(z)=J(z)=J(z+\pi\tau)={B}(3\pi\tau-2\gamma+z)-{B}(\pi\tau+z),\]
and rearranging yields
\[{B}(3\pi\tau-2\gamma+z)-{B}(2\pi\tau-2\gamma+z)={B}(\pi\tau+z)-{B}(z).\]
This implies that ${B}(\pi\tau+z)-{B}(z)$ is an elliptic function with periods $\pi$ and $2\pi\tau-2\gamma$. If this function has a pole $z_0$, then for every $k\in\mathbb{Z}$, the value $\tilde{z}_{k}=z_{0}+k(2\pi\tau-2\gamma)$ is a pole. This is a contradiction as these points all define different classes $z_{k}+\pi\tau\mathbb{Z}+\pi\mathbb{Z}$, since $\frac{\gamma}{\pi\tau}\in\mathbb{R}\setminus\mathbb{Q}$. The only remaining case to consider is when ${B}(\pi\tau+z)-{B}(z)$ has no poles, in which case it must be constant:
\[{B}(\pi\tau+z)-{B}(z)=c.\]
In fact combining this with \eqref{eq:PH_flippy_full}, we see that $c=0$, as
\[c={B}(\pi\tau+z)-{B}(z)={B}(-\gamma-z)-{B}(\pi\tau-\gamma-z)=-c.\]
Hence by Proposition \ref{prop:rationalofXorY} the function ${B}(z)$ must be a rational function of $X(z)$ since we have ${B}(z+\pi\tau)={B}(z)$ and ${B}(z)={B}(-\pi\tau+\gamma-z)={B}(\gamma-z)$. But this contradicts Lemma \ref{lem:not-rational}.
\end{proof}



\subsubsection{Decoupling cases}
Recall from Definition \ref{defn:decoupling} that we say that $X(z)^{p}Y(z)^{q}$ is decoupling if there is a pair of rational functions $R_{1}$ and $R_{2}$ satisfying
\[X(z)^{p}Y(z)^{q}=R_{1}(X(z))+R_{2}(Y(z)).\]
As we will show in the following theorem, this implies that $\Cgf(x,y;t)$ is D-algebraic in $x$ and $y$. The analogous result was proven in the quarter plane by Bernardi, Bousquet-M\'elou and Raschel \cite{bernardi2017counting}, and more precisely they proved that $\Cgf(x,y;t)$ is D-algebraic in $t$ under the same condition. This results from the fact that all of the parameters that depend on $t$ and all of the functions involved in the solution depend on $t$ in a D-algebraic way. We believe that the same argument applies here, although a rigorous proof of this is outside the scope of this article. 
\begin{Theorem}\label{thm:D-alg_proof}
Assume that
\[X(z)^pY(z)^q=R_{1}(X(z))+R_{2}(Y(z))\]
holds for some rational functions $R_{1}$ and $R_{2}$. Then $\Cgf(x,y;t)$ is D-algebraic in $x$ and $y$. 
\end{Theorem}
\begin{proof}
Under the assumption, \eqref{eq:PH_PV_nice} can be written as
\begin{equation}\label{eq:Tequation_in_decoupling_case}T(z):=R_{1}(X(z))-{B}(z)={A}(z)+F-R_{2}(Y(z)),\end{equation}
which implies that $T(z)$ satisfies $T(z)=T(\pi\tau-\gamma-z)=T(-\pi\tau+\gamma-z)=T(z+\pi)$. Combining these shows that $T(z)$ is an elliptic function with periods $\pi$ and $2\pi\tau-2\gamma$. This means $T'(z)$ is an elliptic function with the same periods so it is related to $T(z)$ by some non-trivial algebraic equation, implying that $T(z)$ is a D-algebraic function of $z$. Indeed, any elliptic functions are D-algebraic for this reason. Now, since $X(z)$ is also D-algebraic, it follows from \eqref{eq:Tequation_in_decoupling_case} that ${B}(z)$ is also D-algebraic in $z$. This is precisely condition \ref{D-algHz} of Theorem \ref{thm:D-algebraic_fixed_t}, which we showed to be equivalent to conditions \ref{D-algx} and \ref{D-algy}, that $\Cgf(x,y;t)$ is D-algebraic in $x$ and $y$. 
\end{proof}
\textbf{Remark:} Using \eqref{eq:Tequation_in_decoupling_case} it can be proven that $T(z)$ is a rational function of $W(z)$ (see Definition \ref{def:W}) using the same idea as Proposition \ref{prop:rationalofXorY}.


\subsubsection{Non-decoupling cases}\label{subsec:inf_group_non-dec}
In this section we show that if there is no decoupling function, then the generating function is not D-algebraic in $x$. The proof works along the same lines as  \cite{dreyfus2018nature,hardouin2020differentially} for the quarter plane case, which relies on Galois theory of q-difference equations. Rather than essentially rewriting these entire proofs, in Appendix \ref{ap:D-trans} we use results from \cite{dreyfus2018nature} to deduce Corollary \ref{cor:D-trans}, which avoids Galois theory language in its statement, and can be readily applied to show the main result of this section. For the following theorem, recall that for fixed $t\in\left(0,\frac{1}{P(1,1)}\right)$, the series $\Cgf(x,y;t)$ converges for $|x|,|y|\in \left(\sqrt{tP(1,1)},\sqrt{\frac{1}{tP(1,1)}}\right)$, so we can consider it to be a function of $x$ and $y$.
\begin{Theorem}\label{thm:non-D-alg_proof}
Fix $t\in\left(0,\frac{1}{P(1,1)}\right)$. Assume that $\frac{\gamma}{\pi\tau}\notin\mathbb{Q}$ and there are no rational functions $R_{1},R_{2}\in\mathbb{C}(x)$ satisfying 
\[X(z)^{p}Y(z)^{q}=R_{1}(X(z))+R_{2}(Y(z)).\]
Then the function $\Cgf(x,y;t)$ is not D-algebraic in $x$ or $y$.
\end{Theorem}
\begin{proof}
It suffices to prove that ${B}(z)$ is not D-algebraic, as we showed below the statement of Theorem \ref{thm:D-algebraic_fixed_t} that this is equivalent to $\Cgf(x,y;t)$ being D-algebraic in $x$ or $y$. Assume for the sake of contradiction that ${B}(z)$ is D-algebraic in $z$. We will show that this implies that there are rational functions $R_{1}$ and $R_{2}$ satisfying the equation in the theorem. 

By Theorem \ref{thm:PH_PV_characterisation}, the functions $h(z):=X(z)^p Y(z)^q$, $f_{1}(z):={A}(z)+F$ and $f_{2}(z):={B}(z)$ satisfy the conditions of Corollary \ref{cor:D-trans}, with $\gamma_{1}=-\pi\tau+\gamma$ and $\gamma_{2}=\pi\tau-\gamma$. Hence, there are meromorphic functions $a_{1},a_{2}:\mathbb{C}\to\mathbb{C}\cup\{\infty\}$ satisfying 
\begin{align*}
X(z)^p Y(z)^q&=a_{1}(z)+a_{2}(z),\\
a_{1}(z)&=a_{1}(z+\pi)=a_{1}(z+\pi\tau)=a_{1}(-\pi\tau+\gamma-z)=a_{1}(\gamma-z),\\
a_{2}(z)&=a_{2}(z+\pi)=a_{2}(z+\pi\tau)=a_{2}(-\gamma-z)=a_{2}(\pi\tau-\gamma-z).
\end{align*}
Finally, by Proposition \ref{prop:rationalofXorY}, this implies that $a_{1}(z)$ is a rational function of $Y(z)$, while $a_{2}(z)$ is a rational function of $X(z)$. Hence we can write
\[X(z)^p Y(z)^q=R_{1}(X(z))+R_{2}(Y(z)),\]
as required.
\end{proof}

\subsection{Forbidding the steps between $(0,1)$ and $(1,0)$}\label{subsec:annoying_step_forbidding}
In this section we prove that in most cases the nature of the generating function $\Cgf(x,y;t)$ does not change if we forbid either or both of the steps between $(0,1)$ and $(1,0)$. To see this, let $\Lgf(x,y;t)$ be the generating function counting walks starting at the same point $(p,q)$, but for which it is forbidden to step directly from $(0,1)$ to $(1,0)$. Let $\Mgf(x,y;t)$ be the generating function for walks which are forbidden to step in either direction between $(1,0)$ and $(0,1)$. We will show that $\Cgf(x,y;t)$, $\Lgf(x,y;t)$ and $\Mgf(x,y;t)$ have the same nature as functions of $x$, except possibly in the case that $\Cgf(x,y;t)$ is D-algebraic but not D-finite, where we do not rule out the possibility that either or both of $\Lgf(x,y;t)$ and $\Mgf(x,y;t)$ are D-finite.

\begin{Proposition}Let $\tilde{\Cgf}(x,y;t)$ count walks with the same step-set starting at $(1,0)$, let $\Cgf_{1}(t)$ denote the generating function counting walks ending at $(0,1)$ and let $\tilde{\Cgf}_{1}(t)$ count walks starting at $(1,0)$ and ending at $(0,1)$. Let $\tilde{\Lgf}(x,y;t)$, $\Lgf_{1}(t)$ and $\tilde{\Lgf}_{1}(t)$ denote the analogous series counting walks where the step from $(0,1)$ to $(1,0)$ is forbidden. Then these are related by
\begin{equation}
\Lgf(x,y;t)=\Cgf(x,y;t)-\omega_{(1,-1)}t\Lgf_{1}(t)\tilde{\Cgf}(x,y;t).\label{eq:BinCCtilde}
\end{equation}
\end{Proposition}
\begin{proof}
Consider the walks counted by $\Cgf(x,y;t)-\Lgf(x,y;t)$, that is, walks that use the step from $(0,1)$ to $(1,0)$ at least once. The section prior to the first of these steps is any walk counted by $\Lgf_{1}(t)$, as it must end at $(0,1)$ and must not use the step from $(0,1)$ to $(1,0)$. The step itself contributes the weight $\omega_{(1,-1)}t$, then the rest of the walk can be any walk starting at $(1,0)$, which are counted by $\tilde{\Cgf}(x,y;t)$. Hence, we have
\[\Cgf(x,y;t)-\Lgf(x,y;t)=\omega_{(1,-1)}t\Lgf_{1}(t)\tilde{\Cgf}(x,y;t).\]
Rearranging this yields the desired equation.
\end{proof}

\begin{Proposition}\label{prop:CtoB_dfinite} If $\Cgf(x,y;t)$ is D-finite with respect to $x$, then $\Lgf(x,y;t)$ is also D-finite with respect to $x$. Moreover, in this case, $\Cgf(x,y;t)$ is algebraic if and only if $\Lgf(x,y;t)$ is algebraic as functions of $x$.
\end{Proposition}
\begin{proof}
Assuming that $\Cgf(x,y;t)$ is D-finite in $x$, we have $\frac{\gamma}{\pi\tau}\in\mathbb{Q}$. Since the walks counted by $\tilde{\Cgf}(x,y;t)$ start on an axis, $\tilde{\Cgf}(x,y;t)$ is algebraic in $x$. Hence by the \eqref{eq:BinCCtilde}, $\Lgf(x,y;t)$ is D-finite in $x$. Moreover, $\Lgf(x,y;t)$ is algebraic in $x$ if and only if $\tilde{\Cgf}(x,y;t)$ is algebraic in $x$.
\end{proof}

\begin{Proposition}\label{prop:CtoB_dalg} $\Cgf(x,y;t)$ is D-algebraic with respect to $x$ if and only if $\Lgf(x,y;t)$ is D-algebraic with respect to $x$.
\end{Proposition}
\begin{proof}
The case where $\Cgf(x,y;t)$ is D-finite with respect to $x$ is covered by Proposition \ref{prop:CtoB_dfinite}.

So we are left with the case where $\Cgf(x,y;t)$ is not D-finite in $x$, that is $\frac{\gamma}{\pi\tau}\notin\mathbb{Q}$. Since the walks counted by $\tilde{\Cgf}(x,y;t)$ start on an axis, $\tilde{\Cgf}(x,y;t)$ is D-algebraic in $x$. Hence by \eqref{eq:BinCCtilde}, $\Lgf(x,y;t)$ is D-algebraic in $x$ if and only if $\tilde{\Cgf}(x,y;t)$ is D-algebraic in $x$.
\end{proof}

\begin{Proposition}\label{prop:CtoA_dfinite} If $\Cgf(x,y;t)$ is D-finite with respect to $x$, then $\Mgf(x,y;t)$ is also D-finite with respect to $x$. Moreover, in this case, $\Cgf(x,y;t)$ is algebraic in $x$ if and only if $\Mgf(x,y;t)$ is algebraic in $x$.
\end{Proposition}
\begin{proof}
Using Proposition \ref{prop:CtoB_dfinite}, we know that $\Lgf(x,y;t)$ is D-finite, and that it is algebraic if an only if $\Cgf(x,y;t)$ is algebraic. Then using the same argument starting with $\Lgf(x,y;t)$ and forbidding the step from $(1,0)$ to $(0,1)$, we can prove that $\Mgf(x,y;t)$ is D-finite and that it is algebraic if and only if $\Lgf(x,y;t)$ is algebraic. 
\end{proof}

\begin{Proposition}\label{prop:CtoA_dalg} $\Cgf(x,y;t)$ is D-algebraic with respect to $x$ if and only if $\Mgf(x,y;t)$ is D-algebraic with respect to $x$.
\end{Proposition}
\begin{proof}
Using Proposition \ref{prop:CtoB_dalg}, we know that $\Lgf(x,y;t)$ is D-algebraic if and only if $\Mgf(x,y;t)$ is D-algebraic. Then using the same argument starting with $\Lgf(x,y;t)$ and forbidding the step from $(1,0)$ to $(0,1)$, we can prove that $\Mgf(x,y;t)$ is D-algebraic if and only if $\Lgf(x,y;t)$ is D-algebraic. 
\end{proof}
 
\section{Walks in an $M$-quadrant cone}\label{sec:more_cones}
In this section we consider the enumeration of walks in more general cones formed by gluing together quarters of the plane along their boundaries. The quadrants $\cQ_{0}$ and $\tilde{\cQ}_{0}$ are defined by:
\begin{align*}
\cQ_{0}&=\{(i,j):i,j>0\},\\
\tilde{\cQ}_{0}&=\{(i,j):i>0;j\geq0\}.
\end{align*}
To define the rest of the quadrants that we will glue together, we consider the function $r:\mathbb{Z}^{2}\to\mathbb{Z}^{2}$, which rotates the plane by $\pi/2$ anticlockwise, that is $r((a,b))=(-b,a)$. This allows us to define quadrants $\cQ_{s}$ and $\tilde{\cQ}_{s}$ for all other $s\in\mathbb{Z}$ recursively using the equations:
\[
\cQ_{s+1}=r(\cQ_{s})\qquad\text{and}\qquad\tilde{\cQ}_{s+1}=r(\tilde{\cQ}_{s})\qquad\text{for $s\in\mathbb{Z}$.}
\]
So $\cQ_{s+4}=\cQ_{s}$ and $\tilde{\cQ}_{s+4}=\tilde{\cQ}_{s}$.

The $3$-quadrant cone considered in the previous section is simply the union of three quadrants $\cQ_{-1}$, $\tilde{\cQ}_{0}$ and $\tilde{\cQ}_{1}$. For $M=1$, $M=2$ and $M=4$, an $M$-quadrant cone can be similarly defined as a union of $M$ of these quadrants, a $4$-quadrant cone being alternatively called the slit plane (see Figure \ref{fig:slit_plane}).
\[\mathbb{Z}^{2}\setminus\{(n,0):n\in\mathbb{Z},~n\leq 0\}.\]

\begin{figure}[ht]
\centering
   \includegraphics[scale=1.2]{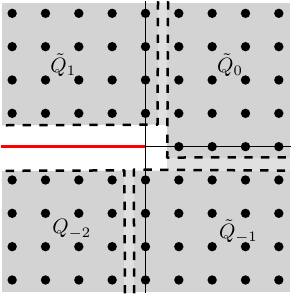}
   \caption{The slit plane $\mathbb{Z}^{2}\setminus\{(n,0):n\in\mathbb{Z},~n\leq 0\}=\tilde{Q}_{1}\cup\tilde{Q}_{0}\cup\tilde{Q}_{-1}\cup Q_{-2}$. This is equivalent to the $4$-quadrant cone $\Pi_{2,1}$.}
   \label{fig:slit_plane}
\end{figure}

For $M>4$, we have to be more careful in our definition, as the quadrants are not, a-priori, disjoint. To define an $M$-quadrant cone, we start by defining the {\em spiral quadrants}
\[\Gamma_{s}=\{(s,z):z\in \cQ_{s}\}\subset\tilde{\Gamma}_{s}=\{(s,z):z\in \tilde{\cQ}_{s}\}.\]
$\tilde{\Gamma}_{s}$ is then isomorphic to $\tilde{\cQ}_{s}$, but, importantly, it is disjoint from $\tilde{\Gamma}_{r}$ for $r\neq s$. Now, rather than considering walks in the whole plane (restricted to some subspace), we will consider walks in the {\em spiral space} $\Pi:=\cup_{s\in\mathbb{Z}}\tilde{\Gamma}_{s}$. Recall that walks in the plane are permitted to have steps coming from some fixed step-set $S$. In the spiral space $\Pi$, a step is allowed from $\pi_{1}=(s_{1},(a_{1},b_{1}))$ to $\pi_{2}=(s_{2},(a_{2},b_{2}))$ if and only if $(a_{2}-a_{1},b_{2}-b_{1})\in S$ and $|s_{1}-s_{2}|\leq 1$. In Lemma \ref{lem:spiral_walk_bij}, we will show that these naturally correspond to walks in $\mathbb{Z}^{2}\setminus\{(0,0)\}$ by projecting each point $(s,(a,b))$ onto $(a,b)$. The additional value $s$ keeps track of the {\em winding angle} of the walk around $(0,0)$.


\begin{Lemma}\label{lem:one_step_on_spiral} Assume $p'$ is adjacent to a point $p\in\tilde{\cQ}_{s}$. Then  $p'$ lies in exactly one of the sets $\tilde{\cQ}_{s-1}$, $\tilde{\cQ}_{s}$, $\tilde{\cQ}_{s+1}$ and $\{(0,0)\}$.\end{Lemma}
\begin{proof}
Applying the transformation $r^{-s}$ yields a point $r^{-s}(p')$ adjacent to $r^{-s}(p)\in\tilde{\cQ}_{0}$, so it suffices to show that $r^{-s}(p')$ lies in exactly one of the sets $\tilde{\cQ}_{-1}=\{(i,j):i\geq 0>j\}$, $\tilde{\cQ}_{0}=\{(i,j):i>0;j\geq 0\}$, $\tilde{\cQ}_{1}=\{(i,j):i\leq 0<j\}$ and $\{(0,0)\}$. These sets are disjoint, so $r^{-s}(p')$ cannot lie in more than one of these sets. Indeed, the points adjacent to $\tilde{\cQ}_{0}$ which lie outside $\tilde{\cQ}_{0}$ are on the ray $\{(i,-1):i\geq0\}$, which lies in $\tilde{\cQ}_{-1}$ or the ray $\{(0,j):j\geq0\}$ of which all but the point $(0,0)$ lies in $\tilde{\cQ}_{1}$.
\end{proof}

\begin{Lemma}\label{lem:spiral_walk_bij} If $p\in\tilde{\cQ}_{s}\subset\mathbb{Z}^{2}$, then there is a bijection between walks in $\mathbb{Z}^{2}\setminus\{(0,0)\}$ starting at $p$ and walks in the spiral space $\Pi$ starting at $(s,p)$ with the same sequence of steps. \end{Lemma}
\begin{proof}
for any walk $(s_{0},p_{0}),(s_{1},p_{1}),\ldots,(s_{n},p_{n})$ in $\Pi$, with $(s_{0},p_{0})=(s,p)$ there is a corresponding walk $p_{0},p_{1},\ldots,p_{n}$ in $\mathbb{Z}^{2}\setminus\{(0,0)\}$ using the same step set $S$. To show that this correspondence forms a bijection, it suffices to show that there is exactly one choice of the letters $s_{j}$ for any walk $p_{0},p_{1},\ldots,p_{n}$ with $s_{0}=s$ fixed. We will show this inductively.
Assume that there is a unique choice of $s_{k}$ and consider the possible values of $s_{k+1}$. Since $s_{k}$ has been chosen, we must have $(s_{k},p_{k})\in\Pi$, that is $p_{k}\in \tilde{\cQ}_{s_{k}}$.
Since $p_{k+1}$ is adjacent to $p_{k}$ and $p_{k+1}\neq (0,0)$, from Lemma \ref{lem:one_step_on_spiral}, we must have $p_{k+1}\in \tilde{\cQ}_{s_{k}-1}$ or $p_{k+1}\in \tilde{\cQ}_{s_{k}}$ or $p_{k+1}\in \tilde{\cQ}_{s_{k}+1}$, and $p_{k+1}$ can only be in one of these sets as they are disjoint.
Since $s_{k+1}$ must satisfy $|s_{k+1}-s_{k}|\leq 1$, this implies that there is a unique choice of $s_{k+1}$ satisfying $p_{k+1}\in \tilde{\cQ}_{s_{k+1}}$. This completes the induction. 
\end{proof}

\begin{Definition}An {\em $M$-quadrant cone} is a subspace $\Pi_{L,K}$ of $\Pi$ defined by
\[\Pi_{L,K}=\Gamma_{ -L}\cup\tilde{\Gamma}_{-L+1}\cup\tilde{\Gamma}_{-L+2}\cup\cdots\cup\tilde{\Gamma}_{K},\]
where $0\leq L,K$ and $L+K+1=M$.\end{Definition}

Note that simultaneously decreasing $L$ by $1$ and increasing $K$ by $1$ simply rotates the cone anticlockwise, so all $M$-quadrant cones are congruent. Hence, we could consider all possible walks in an $M$-quadrant cone while fixing $L=0$. We will not do this, however, as it is more convenient to allow any value of $L$, but insist that the starting point of the walk is in $\tilde{\Gamma}_{0}$. 

In Subsection \ref{subsec:Mon4plane}, we define the generating function $\Qgf_{j}(x,y;t)$ counting walks in the $M$-quadrant cone $\Pi_{L,M}$, and we derive functional equations characterising these series, then in Subsection \ref{subsec:bc_analytic_functional_equations} we use these to determine an analytic functional equation, generalising Theorem \ref{thm:PH_PV_characterisation}. Finally in Section \ref{sec:bc_solutions}, we use this analytic functional equation to determine the complexity of the generating functions $\Qgf_{j}(x,y;t)$. In a forthcoming paper we will consider walks in the complete spiral space $\Pi$ and walks restricted to the half spiral space $\displaystyle\Pi_{L}=\Gamma_{-L}\cup\bigcup_{s> -L}\tilde{\Gamma}_{s}$. 

\subsection{Functional equations for walks in an $M$-quadrant cone}
\label{subsec:Mon4plane}
In this section we consider walks restricted to the $M$-quadrant cone $\Pi_{L,M}$ starting at some point $(0,(p,q))\in\tilde{\Gamma}_{0}$ using (weighted) steps in $S$.

We will now define the generating functions that we use to count these walks. 
We define generating functions $\Qgf_{s}(x,y;t)$ for $-L\leq s\leq K$ as follows: Let $q_{s}(i,j;n)$ be the weighted number of walks of length $n$ starting at $(0,(p,q))$ and ending at the point $(s,(i,j))$. We define the generating function $\Qgf_{s}(x,y;t)$ by
\[\Qgf_{s}(x,y;t):=\sum_{n\geq 0}\sum_{i,j\in\mathbb{Z}} q_{s}(i,j;n)x^{i}y^{j}t^{n}.\]
Note that for $j\in\mathbb{Z}$ we have
\begin{itemize}
\item $\Qgf_{4j}(x,y;t)\in x\mathbb{R}[x,y][[t]],$
\item $\Qgf_{4j+1}(x,y;t)\in y\mathbb{R}[\frac{1}{x},y][[t]],$
\item $\Qgf_{4j+2}(x,y;t)\in\frac{1}{x}\mathbb{R}[\frac{1}{x},\frac{1}{y}][[t]],$
\item $\Qgf_{4j+3}(x,y;t)\in\frac{1}{y}\mathbb{R}[x,\frac{1}{y}][[t]].$
\end{itemize}
We can write this more concisely by defining $x_{s}$ and $y_{s}$ to be rotated versions of $x$ and $y$ for each $s$. precisely, for $k\in\mathbb{Z}$, we define
\begin{align*}
x_{4k}&=y_{4k-1}=x,\\
x_{4k+1}&=y_{4k\phantom{+1}}=y,\\
x_{4k+2}&=y_{4k+1}=\frac{1}{x},\\
x_{4k+3}&=y_{4k+2}=\frac{1}{y}.
\end{align*}
Using these variables, we have
\[\Qgf_{j}(x,y;t)\in x_{j}\mathbb{R}[x_{j},y_{j}][[t]].\]
Moreover, since walks counted by $\Qgf_{-L}(x,y;t)$ finish in $\Gamma_{-L}$, we must have
\[\Qgf_{-L}(x,y;t)\in x_{-L}y_{-L}\mathbb{R}[x_{-L},y_{-L}][[t]].\]

We will now derive a system of functional equations characterising the series $\Qgf_{s}(x,y;t)$. 
In order to write these functional equations we define, for $-L\leq i,j\leq K$, the series $\Sgf_{i,j}(x,y;t)$ to be the generating function for walks in the spiral plane $\Pi$, whose last step is from $\tilde{\Gamma}_{i}$ to $\tilde{\Gamma}_{j}$. So $\Sgf_{i,j}=0$ unless $|i-j|\leq 1$. Furthermore, for $-L\leq j\leq K$, we define $\Tgf_{j}(x,y;t)$ to be the generating function for weighted walks using the step-set $S$ which do not end in $\Pi_{L,K}$, but for which the removal of the final step yields a walk in $\Pi_{L,K}$ ending in $\tilde{\Gamma}_{j}$. In the following two lemmas we describe spaces in which the series $\Sgf_{i,j}(x,y;t)$ and $\Tgf_{j}(x,y;t)$ lie.
\begin{Lemma}\label{lem:S_ser_spaces}
The series $\Sgf_{i,j}(x,y;t)$ lie in:
\begin{itemize}
\item $\Sgf_{j,j+1}(x,y;t)\in y_{j}\mathbb{R}[y_{j}][[t]]$,
\item $\Sgf_{j+1,j}(x,y;t)\in x_{j}\mathbb{R}[y_{j}][[t]]$.
\end{itemize}
\end{Lemma}
\begin{proof}
Let $w$ be a walk counted by $\Sgf_{j,j+1}(x,y;t)$ and let $w'$ be the same walk with the final step removed. Then $w$ is also counted by $\Qgf_{j+1}(x,y;t)\in x_{j+1}\mathbb{R}[x_{j+1},y_{j+1}][[t]]=y_{j}\mathbb{R}\left[y_{j},\frac{1}{x_{j}}\right][[t]]$, while $w'$ is counted by $\Qgf_{j}(x,y;t)\in x_{j}\mathbb{R}[x_{j},y_{j}][[t]]$. Since the contributions from $w'$ and $w$ can only differ by a factor of $x^{\alpha}y^{\beta}t$, where $-1\leq\alpha,\beta\leq 1$, the contribution from $w$ must lie in $y_{j}\mathbb{R}[y_{j}][[t]]$, so the sum $\Sgf_{j,j+1}(x,y;t)$ of the contributions from all such walks also lies in $y_{j}\mathbb{R}[y_{j}][[t]]$. 

Similarly, if we let $w'$ be a walk counted by $\Sgf_{j,j+1}(x,y;t)$ and let $w$ be the same walk with the final step removed, we can again deduce that $w$ and $w'$ are also counted by $\Qgf_{j+1}(x,y;t)$ and $\Qgf_{j}(x,y;t)$, respectively. This implies that the contribution from $w'$ lies in $x_{j}\mathbb{R}[y_{j}][[t]]$, so the sum $\Sgf_{j+1,j}(x,y;t)$ of the contributions from all such walks also lies in $x_{j}\mathbb{R}[y_{j}][[t]]$.
\end{proof}
\begin{Lemma}\label{lem:T_series_cases}
The series $\Tgf_{j}(x,y;t)$ lie in
\begin{itemize}
\item $\Tgf_{0}(x,y;t)\in \mathbb{R}[x][[t]]+\mathbb{R}[y][[t]]$, if $M=1$,
\item $\Tgf_{-L}(x,y;t)\in \mathbb{R}[x_{-L}][[t]]$, if $M\geq2$,
\item $\Tgf_{K}(x,y;t)=\Tgf_{1-L}(x,y;t)\in y_{1-L}^{-1}\mathbb{R}[y_{1-L}][[t]]$, if $M=2$,
\item $\Tgf_{1-L}(x,y;t)\in y_{1-L}^{-1}\mathbb{R}[[t]]+\mathbb{R}[[t]]$, if $M\geq 3$,
\item $\Tgf_{K}(x,y;t)\in \mathbb{R}[y_{K}][[t]]$, if $M\geq 3$,
\item $\Tgf_{j}(x,y;t)\in \mathbb{R}[[t]]$, for $1-L<j<K$.
\end{itemize}
\end{Lemma}
\begin{proof}
Let $w$ be a walk counted by $\Tgf_{j}(x,y;t)$ and let $w'$ be the same walk with the final step removed. Let $ux_{j}^{a}y_{j}^{b}t^n$ be the contribution from $w$, where $u$ is the weight of the walk (i.e., the product of the weights of the steps). then $\hat{w}:=r^{-j}(w)$ ends at $(a,b)$ while the walk $\hat{w}':=r^{-j}(w')$ ends at some point $(a-\alpha,b-\beta)$, where $\alpha,\beta\in \{-1,0,1\}$. Moreover, the endpoint $(a-\alpha,b-\beta)$ of $\hat{w}'$ lies in $r^{-j}(\tilde{\cQ}_{j})=\tilde{\cQ}_{0}$, and more precisely in $r^{L}(\cQ_{-L})=\cQ_{0}$ in the case that $j=-L$. We will now separately consider the six cases in the statement of this lemma, and show that in each case the contribution $u x_{j}^{a}y_{j}^{b}t^n$ to $\Tgf_{j}(x,y;t)$ from $w$ lies in the claimed ring. This implies the desired result as $\Tgf_{j}(x,y;t)$ is a sum of these contributions over all possible $w$, and there are only finitely many terms in this sum for each length $n$. \newline
\textbf{Case 1:} $M=1$.\newline
In this case we necessarily have $j=L=0$. Hence $(a-\alpha,b-\beta)\in\cQ_{0}=\{(k,\ell):k,\ell>1\}$, but $(a,b)\notin\cQ_{0}$. This is only possible if either $a=0$ and $b\geq 0$ or $b=0$ and $a\geq0$. Hence the contribution $u x_{j}^{a}y_{j}^{b}t^n=x^{a}y^{b}t^n$ to $\Tgf_{0}(x,y;t)$ from $w$ lies in $\mathbb{R}[x][[t]]+\mathbb{R}[y][[t]]$.\newline
\textbf{Case 2:} $M\geq 2$ and $j=-L$.\newline
In this case $(a-\alpha,b-\beta)\in\cQ_{0}=\{(k,\ell):k,\ell>1\}$, but $(a,b)\notin\cQ_{0}$ and $(a,b)\notin r^{-j}(\tilde{\cQ}_{1-L})=\tilde{\cQ_{1}}$. As in the last case $a=0$ and $b>0$ or $b=0$ and $a\geq0$, except that the case $a=0$ and $b>0$ does not occur because then $(a,b)$ would lie in $\tilde{\cQ_{1}}$. Hence the only possibility is that $b=0$ and $a\geq0$, so the contribution $u x_{-L}^{a}y_{-L}^{b}t^n$ to $\Tgf_{-L}(x,y;t)$ from $w$ lies in $\mathbb{R}[x_{-L}][[t]]$.\newline
\textbf{Case 3:} $M=2$ and $j=K=1-L$.\newline
In this case $(a-\alpha,b-\beta)\in\tilde{\cQ}_{0}=\{(k,\ell):k>1;\ell\geq 0\}$, but $(a,b)\notin\tilde{\cQ}_{0}$ and $(a,b)\notin r^{-j}(\cQ_{-L})=\cQ_{-1}$. The union $\tilde{\cQ}_{0}\cup\cQ_{-1}=\{(i,j):i>1\}$, so for $(a,b)$ to be adjacent to this set but not in this set, we must have $a=0$. Moreover, since $(a-\alpha,b-\beta)\in\tilde{\cQ}_{0}$, we must have $b\geq -1$. Hence, the contribution $u x_{1-L}^{a}y_{1-L}^{b}t^n$ to $\Tgf_{1-L}(x,y;t)$ from $w$ lies in $y_{1-L}^{-1}\mathbb{R}[y_{1-L}][[t]]$.\newline
\textbf{Case 4:} $M\geq 3$ and $j=1-L$.\newline
In this case $(a-\alpha,b-\beta)\in\tilde{\cQ}_{0}=\{(k,\ell):k>1;\ell\geq 0\}$, but $(a,b)\notin\tilde{\cQ}_{0}$ and $(a,b)\notin r^{-j}(\cQ_{-L})=\cQ_{-1}$ and $(a,b)\notin r^{-j}(\tilde{\cQ}_{2-L})=\tilde{\cQ}_{1}$. The only points adjacent to $\tilde{\cQ}_{0}$ which do not lie in $\tilde{\cQ}_{1}$ or $\cQ_{-1}$ are $(0,0)$ and $(0,-1)$, so $(a,b)$ is one of these points. Hence, the contribution $u x_{1-L}^{a}y_{1-L}^{b}t^n$ to $\Tgf_{1-L}(x,y;t)$ from $w$ lies in $y_{1-L}^{-1}\mathbb{R}[[t]]+\mathbb{R}[[t]]$.\newline
\textbf{Case 5:} $M\geq 3$ and $j=K$.\newline
In this case $(a-\alpha,b-\beta)\in\tilde{\cQ}_{0}=\{(k,\ell):k>1;\ell\geq 0\}$, but $(a,b)\notin\tilde{\cQ}_{0}$ and $(a,b)\notin r^{-j}(\tilde{\cQ}_{K-1})=\tilde{\cQ}_{-1}$. The only points adjacent to $\tilde{\cQ}_{0}$ which do not lie in $\tilde{\cQ}_{-1}$ are the points $(0,b)$ for $b\geq 0$. Hence, the contribution $u x_{K}^{a}y_{K}^{b}t^n$ to $\Tgf_{K}(x,y;t)$ from $w$ lies in $\mathbb{R}[y_{K}][[t]]$.\newline
\textbf{Case 6:} $1-L<j<K$.\newline
In this case $(a-\alpha,b-\beta)\in\tilde{\cQ}_{0}=\{(k,\ell):k>1;\ell\geq 0\}$, but $(a,b)\notin\tilde{\cQ}_{0}$ and $(a,b)\notin r^{-j}(\tilde{\cQ}_{j-1})=\tilde{\cQ}_{-1}$ and $(a,b)\notin r^{-j}(\tilde{\cQ}_{j+1})=\tilde{\cQ}_{1}$. The only point adjacent to $\tilde{\cQ}_{0}$ which does not lie in any of these sets is the point $(0,0)$. Hence, the contribution $u x_{j}^{a}y_{j}^{b}t^n$ to $\Tgf_{j}(x,y;t)$ from $w$ lies in $\mathbb{R}[[t]]$.\newline
\end{proof}

It will be convenient to write the series $\Tgf_{j}(x,y;t)$, $\Sgf_{j,j+1}(x,y;t)$ and $\Sgf_{j+1,j}(x,y;t)$ in terms of series involving only 1 or none of the variables $x$, $y$. 
To do this we define 
\begin{align}
\Fgf_{j}(t)&=[x^{0}y^{0}]T_{j}(x,y;t),&&\text{for $j\in[-L,K]$},\notag\\
\Ugf_{0}(x_{-L};t)&=[x_{-L}^{\geq 1}y_{-L}^{0}]T_{-L}(x,y;t),\notag\\
\Ugf_{1}(t)&=[x_{-L}^{1}y_{-L}^{0}]T_{1-L}(x,y;t),\notag\\
\Ugf_{2}(y_{K};t)&=[x_{K}^{0}y_{K}^{\geq 1}]T_{K}(x,y;t),\notag\\
\Rgf_{j}(x_{j};t)&=x_{j}^{-1}\Sgf_{j-1,j}(x,y;t),&&\text{for $j\in[-L+1,K]$},\label{eq:bc_Rgf_def}\\
\Lgf_{j}(y_{j};t)&=x_{j}^{-1}\Sgf_{j+1,j}(x,y;t)+\delta_{j,-L}\Ugf_{1}(t),&&\text{for $j\in[-L,K-1]$.}\label{eq:bc_Lgf_def}
\end{align}
Then from Lemma \ref{lem:T_series_cases}, we have in all cases
\begin{equation}\label{eq:Tsplitting}\Tgf_{j}(x,y;t)=\Fgf_{j}(t)+\delta_{j,-L}\Ugf_{0}(x_{-L};t)+\delta_{j,1-L}x_{-L}\Ugf_{1}(t)+\delta_{j,K}\Ugf_{2}(y_{K};t)\end{equation}
and $F_{j}(t)\in\mathbb{R}[[t]]$,  $\Ugf_{0}(x_{-L};t)\in x_{-L}\mathbb{R}[x_{-L}][[t]]$, $\Ugf_{1}(t)\in \mathbb{R}[[t]]$ and $\Ugf_{2}(y_{K};t)\in y_{K}\mathbb{R}[y_{K}][[t]]$. Moreover, $\Rgf_{j}(x_{j};t)\in\mathbb{R}[x_{j}][[t]]$ and $\Lgf_{j}(y_{j};t)\in\mathbb{R}[y_{j}][[t]]$.

For the following lemma, recall that the Kernel $K(x,y;t)$ of the model, discussed in Section \ref{sec:param}, is given by $K(x,y;t)=t\Pgf(x,y)-1$. 
\begin{Lemma}
The series $\Qgf_{j}(x,y;t)$ and $\Tgf_{j}(x,y;t)$ satisfy the equations
\begin{align}\begin{split}
K(x,y;t)\Qgf_{j}(x,y;t)&=-\delta_{0,j}x^{p}y^{q}-\Sgf_{j+1,j}(x,y;t)+\Sgf_{j,j+1}(x,y;t)\\
&~~-\Sgf_{j-1,j}(x,y;t)+\Sgf_{j,j-1}(x,y;t)+\Tgf_{j}(x,y;t),\label{eq:eq_for_q_spiral}
\end{split}\end{align}
\end{Lemma}
\begin{proof}
All non-trivial walks ending in $\tilde{\Gamma}_{j}$ are counted by exactly one of $\Sgf_{j-1,j}(x,y;t)$, $\Sgf_{j,j}(x,y;t)$ and $\Sgf_{j+1,j}(x,y;t)$, depending on the position of the point before the final step. the contribution of the trivial walk is $x^{p}y^{q}$, and this only contributes if $j=0$. Hence we have
\[\Qgf_{j}(x,y;t)=\delta_{0,j}x^{p}y^{q}+\Sgf_{j-1,j}(x,y;t)+\Sgf_{j,j}(x,y;t)+\Sgf_{j+1,j}(x,y;t).\]
Now we consider non-trivial walks whose last step starts within $\tilde{\Gamma}_{j}$. These can be counted using the generating function $\Qgf_{j}(x,y;t)$ and multiplying by $t\Pgf(x,y)$ to account for the addition of a single arbitrary step. Alternatively we can observe that each of these walks contributes to exactly one of $\Sgf_{j,j-1}(x,y;t)$, $\Sgf_{j,j}(x,y;t)$, $\Sgf_{j,j+1}(x,y;t)$ and $T_{j}(x,y;t)$. This yields the equation
\[t\Pgf(x,y)\Qgf_{j}(x,y;t)=\Sgf_{j,j-1}(x,y;t)+\Sgf_{j,j}(x,y;t)+\Sgf_{j,j+1}(x,y;t)+\Tgf_{j}(x,y;t).\]
Taking the difference between these to cancel the term $\Sgf_{j,j}$ yields \eqref{eq:eq_for_q_spiral}.
\end{proof}

\begin{Lemma}\label{lem:VH_tilde_ser_stuff} Define the series $\Agf(x_{-L};t)$ and $\Bgf(y_{K};t)$ by
\begin{align}
\label{eq:Agfdefinition}\Agf(x_{-L};t)&:=\Ugf_{0}(x_{-L};t)+x_{-L}\Ugf_{1}(t),\\
\label{eq:BgfPVdefinition}\Bgf(y_{K};t)&:=\Ugf_{2}(y_{K};t).
\end{align}
 If $M\geq 2$ then the series $\Agf(x_{-L})$ and $\Bgf(y_{K})$ satisfy the equations
\begin{align}
\begin{split}K(x,y)\Qgf_{j}(x,y)&=-\delta_{j,0}x^{p}y^{q}-x_{j}\Lgf_{j}(y_{j})+y_{j}\Rgf_{j+1}(y_{j})\\
&\qquad-x_{j}\Rgf_{j}(x_{j})+y_{j}^{-1}\Lgf_{j-1}(x_{j})+\Fgf_{j}
\end{split}\qquad\text{for $1-L\leq j\leq K-1$},\label{eq:Vj_Hj_series}\\
K(x,y)\Qgf_{-L}(x,y)&=-\delta_{-L,0}x^{p}y^{q}-x_{-L}\Lgf_{-L}(y_{-L})+y_{-L}\Rgf_{-L+1}(y_{-L})+\Agf(x_{-L})+\Fgf_{-L},\label{eq:V-L_A_series}\\
K(x,y)\Qgf_{K}(x,y)&=-\delta_{K,0}x^{p}y^{q}-x_{K}\Rgf_{K}(x_{K})+y_{K}^{-1}\Lgf_{K-1}(x_{K})+\Bgf(y_{K})+\Fgf_{K}.\label{eq:HK_B_series}
\end{align}
If $M=1$ the series $\Agf(x_{0})$ and $\Bgf(y_{0})$ satisfy the equation
\begin{align}
K(x,y)\Qgf_{0}(x,y)&=-x^{p}y^{q}+\Agf(x_{0})+\Bgf(y_{0})+\Fgf_{0},\label{eq:AB_series_M1}
\end{align}
Moreover, in either case, $\Agf(x_{-L};t)\in\mathbb{R}[x_{-L}][[t]]$ and $\Bgf(y_{K};t)\in y_{K}\mathbb{R}[y_{K}][[t]]$.
\end{Lemma}
\begin{proof}
First, \eqref{eq:Vj_Hj_series}, \eqref{eq:V-L_A_series} and \eqref{eq:HK_B_series} result from expanding \eqref{eq:eq_for_q_spiral} for different values of $j$ using \eqref{eq:Tsplitting}, and using \eqref{eq:bc_Rgf_def} and \eqref{eq:bc_Lgf_def} to write $\Sgf$ in terms of $\Rgf$ and $\Lgf$. The equation \eqref{eq:AB_series_M1} follows from \eqref{eq:eq_for_q_spiral} at $j=0$, as in this case there is only one quadrant so the series $\Sgf_{j,k}=0$ (since $j,k$ are not both $0$) and we have $\Tgf_{0}(x,y;t)=\Agf(x_{0};t)+\Bgf(y_{0};t)+\Fgf_{0}(t)$. As stated just below \eqref{eq:Tsplitting}, we have $\Ugf_{0}(x_{-L};t)\in x_{-L}\mathbb{R}[x_{-L}][[t]]$, $\Ugf_{1}(t)\in \mathbb{R}[[t]]$ and $\Ugf_{2}(y_{K};t)\in y_{K}\mathbb{R}[y_{K}][[t]]$, hence from the definitions \eqref{eq:Agfdefinition} and \eqref{eq:BgfPVdefinition}, we have $\Agf(x_{-L};t)\in\mathbb{R}[x_{-L}][[t]]$ and $\Bgf(y_{K};t)\in y_{K}\mathbb{R}[y_{K}][[t]]$.
\end{proof}

\subsection{Analytic reformulation of functional equations}
\label{subsec:bc_analytic_functional_equations}
As for the three-quadrant cone, we will fix $t<1/|S|$, substitute $x=X(z)$, $y=Y(z)$ and rewrite the equations in terms of $z$, using results from Section \ref{sec:param} regarding $X(z)$ and $Y(z)$. We start by defining $X_{j}(z)$ and $Y_{j}(z)$ analogously to our definitions of $x_{j}$ and $y_{j}$, to be the functions such that $(x,y)=(X(z),Y(z))$ if and only if $(x_{j},y_{j})=(X_{j}(z),Y_{j}(z))$. That is
\begin{align*}
X_{4k}(z)&=Y_{4k-1}(z)=X(z)\\
X_{4k+1}(z)&=Y_{4k}(z)\quad=Y(z)\\
X_{4k+2}(z)&=Y_{4k+1}(z)=\frac{1}{X(z)}\\
X_{4k+3}(z)&=Y_{4k+2}(z)=\frac{1}{Y(z)}.
\end{align*}
Furthermore, we define similarly shifted versions of $\gamma$ and the roots $\alpha$ and $\beta$ of $X(z)$ and $Y(z)$ from Proposition \ref{prop:th_param}.
\begin{Definition}\label{def:jstuff}
For $k\in\mathbb{Z}$, define $\beta_{4k}=\beta+k\pi\tau$, $\beta_{4k+1}=\delta+k\pi\tau$, $\beta_{4k-1}=\alpha+k\pi\tau$ and $\beta_{4k-2}=\epsilon+k\pi\tau$. Moreover, define $\gamma_{2k}=\gamma+k\pi\tau$ and $\gamma_{2k-1}=-\gamma+k\pi\tau$, and defined $\alpha_{k}=\beta_{k-1}$.
\end{Definition}
The reason for these definitions is so that the following simple Proposition holds for all $j$.
\begin{Proposition}\label{prop:jstuff}
The functions $X_j(z),Y_{j}(z)$, the spaces $\Omega_{j}$ and the constants $\alpha_{j},\beta_{j},\gamma_{j}$ satisfy the following properties
\begin{enumerate}[label={\rm(\roman*)},ref={\rm(\roman*)}]
\item $|Y_{j}(z)|\leq 1$ for $z\in\Omega_{j}\cup\Omega_{j+1}$, \label{cond:Yjleq1}
\item $|X_{j}(z)|\leq 1$ for $z\in\Omega_{j}\cup\Omega_{j-1}$, \label{cond:Xjleq1}
\item $Y_{j}(\gamma_{j}-z)=Y_{j}(z)$, \label{cond:Yjflippy}
\item $\gamma_{j}-\Omega_{j}\cup\Omega_{j+1}=\Omega_{j}\cup\Omega_{j+1}$, \label{cond:Omegaj_flippy}
\item $\beta_{j}$ and $\gamma_{j}-\beta_{j}$ are the roots of $Y_{j}$ in $\Omega_{j}\cup\Omega_{j+1}$, \label{cond:Yj_roots}
\item $\alpha_{j}$ and $\gamma_{j-1}-\alpha_{j}$ are the roots of $X_{j}$ in $\Omega_{j}\cup\Omega_{j-1}$. \label{cond:Xj_roots}
\end{enumerate}
\end{Proposition}
\begin{proof} Since $X_{j}=Y_{j-1}$ and $\alpha_{j}=\beta_{j-1}$, Conditions \ref{cond:Yjleq1} and \ref{cond:Xjleq1} are equivalent, while Conditions \ref{cond:Yj_roots} and \ref{cond:Xj_roots} are also equivalent, so it suffices to prove \ref{cond:Yjleq1}, \ref{cond:Yjflippy}, \ref{cond:Omegaj_flippy} and \ref{cond:Yj_roots}. First, \ref{cond:Yjleq1} and \ref{cond:Omegaj_flippy} follow from Lemma \ref{lem:Omega}. For $j$ even, \ref{cond:Yjflippy} follows from
\[Y(z)=Y(\gamma-z)=Y(\gamma_{j}-z),\]
while for $j$ odd it follows from
\[X(z)=X(-\gamma-z)=X(\gamma_{j}-z).\]
Finally from Proposition \ref{prop:th_param}, we have
\begin{itemize}
\item $\alpha,-\gamma-\alpha\in\Omega_{0}\cup\Omega_{-1}$ are roots of $X(z)$, so $\beta_{4k-1},\gamma_{4k-1}-\beta_{4k-1}\in\Omega_{4k-1}\cup\Omega_{4k}$ are roots of $Y_{4k-1}(z)=X(z)$,
\item $\beta,\gamma-\beta\in\Omega_{0}\cup\Omega_{1}$ are roots of $Y(z)$, so $\beta_{4k},\gamma_{4k}-\beta_{4k}\in\Omega_{4k}\cup\Omega_{4k+1}$ are roots of $Y_{4k}(z)=Y(z)$,
\item $\delta,\pi\tau-\gamma-\delta\in\Omega_{1}\cup\Omega_{2}$ are poles of $X(z)$, so $\beta_{4k+1},\gamma_{4k+1}-\beta_{4k+1}\in\Omega_{4k+1}\cup\Omega_{4k+2}$ are roots of $Y_{4k+1}(z)=\frac{1}{X(z)}$,
\item $\epsilon,\gamma-\pi\tau-\epsilon\in\Omega_{-2}\cup\Omega_{-1}$ are poles of $Y(z)$, so $\beta_{4k-2},\gamma_{4k-2}-\beta_{4k-2}\in\Omega_{4k-2}\cup\Omega_{4k-1}$ are roots of $Y_{4k-2}(z)=\frac{1}{Y(z)}$.
\end{itemize}
\end{proof}

\begin{Proposition} Define the functions $V_{j}(z)$, $A(z)$, $B(z)$ and constants $F_{j}$ by
\begin{align}
\label{eq:Vjdefinition}V_{j}(z)&:=X_{j}(z)\Lgf_{j}(Y_{j}(z);t)-Y_{j}(z)\Rgf_{j+1}(Y_{j}(z);t),\qquad&&\text{for $z\in\Omega_{j}\cup\Omega_{j+1}$},\\
\label{eq:bc_PHdefinition}A(z)&:=\Agf(X_{-L}(z);t),\qquad&&\text{for $z\in\Omega_{-L-1}\cup\Omega_{-L}$},\\
\label{eq:bc_PVdefinition}B(z)&:=\Bgf(Y_{K}(z);t),\qquad&&\text{for $z\in\Omega_{K}\cup\Omega_{K+1}$},\\
\label{eq:Fdefinition}F_{j}&:=\Fgf_{j}(t).
\end{align}
These functions are well defined. Moreover, they satisfy the equations
\begin{align}
0&=\delta_{j,0}X(z)^{p}Y(z)^{q}+V_{j}(z)-V_{j-1}(z)-F_{j}~~&&\text{for $z\in\Omega_{j}$},~~~~\text{for $-L<j< K$},\label{eq:Vj_Hj_analytic}\\
0&=\delta_{-L,0}X(z)^{p}Y(z)^{q}+V_{-L}(z)-A(z)-F_{-L}~~&&\text{for $z\in\Omega_{-L}$ if $M\geq2$}\label{eq:V-L_A_analytic}\\
0&=\delta_{K,0}X(z)^{p}Y(z)^{q}-V_{K-1}(z)-B(z)-F_{K}~~&&\text{for $z\in\Omega_{K}$ if $M\geq2$}\label{eq:HK_B_analytic}\\
0&=X(z)^{p}Y(z)^{q}-A(z)-B(z)-F_{0}~~&&\text{for $z\in\Omega_{0}$ if $M=1$}\label{eq:B_A_analyticM1}\\
A(z)&=A(\gamma_{-L-1}-z),\label{eq:A_flippy}&&\text{for $z\in\Omega_{-L-1}\cup\Omega_{-L}$}\\
B(z)&=B(\gamma_{K}-z),\label{eq:B_flippy}&&\text{for $z\in\Omega_{K}\cup\Omega_{K+1}$}
\end{align}
\end{Proposition}
\begin{proof}
We start by showing that the functions are well defined and meromorphic on the domains on which they are defined. It suffices to show that the series defining the functions converge absolutely. The series $x_{j}\tilde{\Lgf}_{j}(y_{j};t)$, $x_{j}\tilde{\Rgf}_{j}(x_{j};t)$, $\Agf(x_{-L};t)$, $\Bgf(y_{K};t)$ and $\Fgf_{j}(t)$ all count some subset of (weighted) walks in the plane starting a certain point $(p,q)$. Since the total weighted number of walks of length $t$ is $P(1,1)^n$, all series must converge absolutely when $|x|=|y|=1$ and $t\in\left(0,\frac{1}{\Pgf(1,1)}\right)$. Hence the series defining \eqref{eq:Vjdefinition} and  \eqref{eq:bc_PVdefinition} converge absolutely, as $|Y_{j}(z)|\leq 1$ for $z\in\Omega_{j}\cup\Omega_{j+1}$, while the series defining \eqref{eq:bc_PHdefinition} also converges absolutely because $|X_{j}(z)|\leq 1$ for $z\in\Omega_{j}\cup\Omega_{j-1}$. Now, \eqref{eq:Vj_Hj_analytic}, \eqref{eq:V-L_A_analytic} and \eqref{eq:HK_B_analytic} follow from from substituting $(x,y)\to(X(z),Y(z))$ into \eqref{eq:Vj_Hj_series}, \eqref{eq:V-L_A_series} and \eqref{eq:HK_B_series} respectively. \eqref{eq:B_A_analyticM1} follows from from substituting $(x,y)\to(X(z),Y(z))$ into \eqref{eq:AB_series_M1}. Finally, \eqref{eq:A_flippy} holds because, by Proposition \ref{prop:jstuff}, $\gamma_{-L-1}-\left(\Omega_{-L-1}\cup\Omega_{-L}\right)=\Omega_{-L-1}\cup\Omega_{-L}$ and for $z\in \Omega_{-L-1}\cup\Omega_{-L}$, we have $X_{-L}(z)=X_{L}(\gamma_{-L-1}-z)$. Similarly, \eqref{eq:B_flippy} holds because $\gamma_{K}-(\Omega_{K}\cup\Omega_{K+1})=\Omega_{K}\cup\Omega_{K+1}$ and for $z\in \Omega_{K}\cup\Omega_{K+1}$, we have $Y_{K}(z)=Y_{K}(\gamma_{K}-z)$.
\end{proof}

\begin{Theorem}\label{thm:bc_PH_PV_characterisation}
The functions $A(z)$, $B(z)$ extend to meromorphic functions on $\mathbb{C}$ which, along with the constant
\begin{align*}
F&:=\sum_{j=-L}^{K}\Fgf_{j}(t),
\end{align*} are characterised by the equations
\begin{align}X(z)^{p}Y(z)^{q}&=F+A(z)+B(z),\label{eq:bc_PH_PV_nice}\\
A(z)&=A(\gamma_{-L-1}-z),\label{eq:bc_PH_flippy_full}\\
B(z)&=B(\gamma_{K}-z),\label{eq:bc_PV_flippy_full}\\
A(z)&=A(z+\pi),\label{eq:bc_PH_pi}\\
B(z)&=B(z+\pi),\label{eq:bc_PV_pi}
\end{align}
along with the conditions
\begin{enumerate}[label={\rm(\roman*)},ref={\rm(\roman*)}]
\item $A(z)$ has no poles in $\Omega_{j}$ for $0\geq j\geq -L-1$,
\item $A(z)$ has roots at the roots $\alpha_{-L}$ and $\gamma_{-L-1}-\alpha_{-L}$ of $X_{-L}(z)$ in $\Omega_{-L-1}\cup\Omega_{-L}$,
\item $B(z)$ has no poles in $\Omega_{j}$ for $0\leq j\leq K+1$,
\item $B(z)$ has roots at the roots $\beta_{K}$ and $\gamma_{K}-\beta_{K}$ of $Y_{K}(z)$ in $\Omega_{K}\cup\Omega_{K+1}$.
\end{enumerate}
\end{Theorem}
\begin{proof}
We start by defining $A(z)$ at all points $z\in\Omega_{j}$ for $-L<j\leq K+1$. 
\begin{align}A(z)&:=V_{j}(z)+\sum_{k=-L}^{j}\left(\delta_{k,0}X(z)^{p}Y(z)^{q}-F_{k}\right),&&\text{for $z\in\Omega_{j}\cup\Omega_{j+1}$,}&&\text{for $-L\leq j< K$},\label{eq:AdefinVj} \\
A(z)&:=-B(z)+X(z)^{p}Y(z)^{q}-F,&&\text{for $z\in\Omega_{K}\cup\Omega_{K+1}$.}&&\label{eq:AdefinB}
\end{align}
We now have two definitions of $A(z)$ in each $\Omega_{j}$ for $-L\leq j\leq K$, noting that $A(z)$ was previously define only on $\Omega_{-L-1}\cup\Omega_{-L}$, so we need to show that these definitions are equivalent in each case. For $z\in\Omega_{-L}$ and $M\geq2$, the definition \eqref{eq:AdefinVj} (for $j=-L$) coincides with $A(z)$ (as previously defined) due to \eqref{eq:V-L_A_analytic}. For $z\in\Omega_{s}$, where $-L<s<K$, and $M\geq2$, the function $A(z)$ is defined by \eqref{eq:AdefinVj} for $j=s$ and $j=s-1$, and these definitions are equivalent due to \eqref{eq:Vj_Hj_analytic} at $j=s$. For $z\in\Omega_{K}$ and $M\geq2$ the definitions \eqref{eq:AdefinVj} (for $j=K-1$) and \eqref{eq:AdefinB} are equivalent due to \eqref{eq:HK_B_analytic}. Finally, for $M=1$, the only case to check is $z\in\Omega_{0}$, where \eqref{eq:AdefinB} holds due to \eqref{eq:B_A_analyticM1}. Hence $A(z)$ is well defined for $z\in\Omega_{-L-1}\cup\Omega_{-L}\cup\cdots\cup\Omega_{K}\cup\Omega_{K+1}$. Moreover, each expression defining $A(z)$ on a set $\Omega_{j}\cup\Omega_{j+1}$ is meromorphic, so $A(z)$ is meromorphic on $\Omega_{-L-1}\cup\Omega_{-L}\cup\cdots\cup\Omega_{K}\cup\Omega_{K+1}$. We can then define $A(z)$ on $\gamma_{-L-1}-\Omega_{-L-1}\cup\Omega_{-L}\cup\cdots\cup\Omega_{K}\cup\Omega_{K+1}$ using $A(z):=A(\gamma_{-L-1}-z)$, as this is consistent on the overlapping region $\Omega_{-L-1}\cup\Omega_{-L}=\gamma_{-L-1}-\Omega_{-L-1}\cup\Omega_{-L}$ by \eqref{eq:A_flippy}. Now, for $z\in \Omega_{K}\cup\Omega_{K+1}$, using \eqref{eq:B_flippy} yields
\[-A(z)+X(z)^{p}Y(z)^{q}=-A(\gamma_{K}-z)+X(\gamma_{K}-z)^{p}Y(\gamma_{K}-z)^{q},\]
so by our extended definition,
\[A(\gamma_{-L-1}-\gamma_{K}+z)=A(z)-X(z)^{p}Y(z)^{q}+X(\gamma_{K}-z)^{p}Y(\gamma_{K}-z)^{q},\]
and we can use this to extend our definition of $A$ to all of $\mathbb{C}$.

We then define $B(z)$ on $\mathbb{C}$ by $B(z)=-A(z)+X(z)^{p}Y(z)^{q}-F$, as, by \eqref{eq:AdefinB}, this is consistent with the value of $B(z)$ in the region $\Omega_{K}\cup\Omega_{K+1}$ where it was already defined. From this definition of $B(z)$, we have immediately \eqref{eq:bc_PH_PV_nice}. Moreover, we know that \eqref{eq:bc_PH_flippy_full}, \eqref{eq:bc_PV_flippy_full}, \eqref{eq:bc_PH_pi} and \eqref{eq:bc_PV_pi} hold on the subsets where $A(z)$ and $B(z)$ were originally defined, so they must hold on all of $\mathbb{C}$ due to the meromorphic extension.

Now we will show that conditions (i)-(iv) hold. First, consider the definition \eqref{eq:Vjdefinition} for $j\in[-L,K-1]$. We know that for $z\in\Omega_{j}\cup\Omega_{j+1}$, the series $\Lgf_{j}$ and $\Rgf_{j+1}$ defining $V_{j}$ converge, so $V_{j}$ can only have a pole in this region at poles of $X_{j}(z)$ and $Y_{j}(z)$. In particular, $V_{j}(z)$ has no poles in $\Omega_{j}$ because $|X_{j}(z)|,|Y_{j}(z)|\leq1$ in this region. Now, to prove (i), assume $j\in [-L+1,0]$ is an integer and recall that for $z\in\Omega_{j}$, the function $A(z)$ satisfies \eqref{eq:AdefinVj}. Since $j\leq 0$, we can rewrite this as
\[A(z)=V_{j}(z)+\delta_{j,0}X(z)^{p}Y(z)^{q}-\sum_{k=-L}^{j}F_{k}.\]
Now, as discussed, $V_{j}(z)$ does not have a pole here. Moreover, $\delta_{j,0}X(z)^{p}Y(z)^{q}$ cannot have a pole either, as either $j\neq 0$ and so this term vanishes or $j=0$ and neither $X(z)$ nor $Y(z)$ has a pole in $\Omega_{j}=\Omega_{0}$. Hence $A(z)$ does not a any poles in $\Omega_{j}$ for $j\in [-L+1,0]$. If $j\in \{-L-1,-L\},$ then $A(z)$ is defined by \eqref{eq:bc_PHdefinition}, which converges for $z\in\Omega_{-L}\cup\Omega_{-L-1}$, so $A(z)$ has no poles in this region either. In fact, by Lemma \ref{lem:VH_tilde_ser_stuff}, the series $\Agf$ has no constant term, so any root of $X_{-L}(z)$ in $\Omega_{-L}\cup\Omega_{-L-1}$ must be a root of $A(z)$, proving (ii).

Conditions (iii) and (iv) can be proven similarly. For (iii), we consider $j\in[0,K-1]$ and compare  \eqref{eq:AdefinVj} to \eqref{eq:bc_PH_PV_nice}, which yields
\[F+B(z)=-V_{j}(z)+\sum_{k=-L}^{j}F_{k},\]
so the fact that $V_{j}(z)$ has no poles in $\Omega_{j}$ implies that $B(z)$ also has no poles in $\Omega_{j}$. For $z\in \Omega_{K}\cup\Omega_{K+1},$ the function $B(z)$ is defined by \eqref{eq:bc_PVdefinition}, which converges for $z\in\Omega_{K}\cup\Omega_{K+1}$, so $B(z)$ has no poles in this region either, proving (iii). In fact, by Lemma \ref{lem:VH_tilde_ser_stuff}, the series $\Bgf$ has no constant term, so any root of $Y_{K}(z)$ in $\Omega_{K}\cup\Omega_{K+1}$ must be a root of $B(z)$, proving (iv).

Finally it remains to show that these conditions uniquely define the functions $A(z)$, $B(z)$ and the constant $F$. Suppose that $\hat{A}(z)$, $\hat{B}(z)$ and $\hat{F}$ is an arbitrary triple satisfying the same conditions. Then it suffices to show that $A(z)=\hat{A}(z)$ and $B(z)=\hat{B}(z)$. Then using \eqref{eq:bc_PH_PV_nice}, \eqref{eq:PH_flippy}, \eqref{eq:PV_flippy} and \eqref{eq:PH_pi}, we see that that the difference
\[\Delta(z):=B(z)-\hat{B}(z)=\hat{A}(z)-A(z)+\hat{F}-F,\]
satisfies $\Delta(z)=\Delta(\gamma_K-z)=\Delta(\gamma_{-L-1}-z)=\Delta(z+\pi)$. Moreover, the four conditions on $A(z)$ and $B(z)$ imply, respectively, that
\begin{enumerate}[label={\rm(\roman*)},ref={\rm(\roman*)}]
\item $\Delta(z)$ has no poles in $\Omega_{j}$ for $0\geq j\geq -L-1$,
\item The values $\alpha_{-L},\gamma_{-L-1}-\alpha_{-L}\in\Omega_{-L-1}\cup\Omega_{-L}$ are roots of $\Delta(z)+F-\hat{F}$,
\item $\Delta(z)$ has no poles in $\Omega_{j}$ for $0\leq j\leq K+1$,
\item The values $\beta_{K},\gamma_{K}-\beta_{K}\Omega_{K}\cup\Omega_{K+1}$ are roots of $\Delta(z)$.
\end{enumerate}
Together with $\Delta(z)=\Delta(\pi\tau-\gamma-z)=\Delta(-\pi\tau+\gamma-z)$, the conditions (i) and (iii) imply that $\Delta(z)$ is an elliptic function with no poles, so it is constant. Moreover, condition (iv) implies that $\Delta(z)$ does have roots, so $\Delta(z)$ is the $0$ function. Condition (ii) then implies that $F=\hat{F}$. Together with the definition of $\Delta$ we have $B(z)=\hat{B}(z)$ and $A(z)=\hat{A}(z)$, as required. 
\end{proof}

For convenience we will define $\hat{\tau}=\frac{1}{\pi}(\gamma_{K}-\gamma_{-L-1})$. Note that combining \eqref{eq:bc_PH_PV_nice}, \eqref{eq:bc_PH_flippy_full}, \eqref{eq:bc_PV_flippy_full} yields
\begin{equation}B(\pi\hat{\tau}+z)-B(z)=\tilde{J}(z),\label{eq:bc_PH_qdiff}\end{equation}
where $\tilde{J}(z)$ is an elliptic function with periods $\pi$ and $\pi\tau$ given by
\begin{equation}\tilde{J}(z):=X(\gamma_{-L-1}-z)^{p}Y(\gamma_{-L-1}-z)^{q}-X(z)^{p}Y(z)^{q}.\label{def:tildej}\end{equation}

To complete this section, we show that the series $\Lgf_{j}(y_{j};t)$ and $\Rgf_{j}(x_{j};t)$ are uniquely determined by $A(z)$ and $B(z)$. Then the generating functions $\Qgf_{j}(x,y;t)$ are determined by \eqref{eq:Vj_Hj_series}, \eqref{eq:V-L_A_series} and \eqref{eq:HK_B_series}. This justifies our claim that Theorem \ref{thm:bc_PH_PV_characterisation} is a restatement of the problem, and it will allow us to determine the nature of the series $\Qgf_{j}(x,y;t)$ in $x$ and $y$ in each case.

\begin{Proposition}Define $s_{j}=1$ for $j\geq0$ and $s_{j}=0$ for $j<0$. The series $\Lgf_{j}(y_{j})$ and $\Rgf_{j+1}(y_{j})\in\mathbb{R}[[y_{j}]]$ are determined from $A(z)$ using the following equations
\begin{align}
V_{j}(z)&=A(z)-s_{j}X(z)^{p}Y(z)^{q}+c_{j}\label{eq:Vj_in_Aandc}\\
\Lgf_{j}(Y_{j}(z))&=\frac{V_{j}(z)-V_{j}(\gamma_{j}-z)}{X_{j}(z)-X_{j}(\gamma_{j}-z)},\qquad&&\text{for $z\in\Omega_{j}\cup\Omega_{j+1}$},\label{eq:MinVj}\\
\Rgf_{j+1}(Y_{j}(z))&=\frac{X_{j}(\gamma_{j}-z)V_{j}(z)-X_{j}(z)V_{j}(\gamma_{j}-z)}{Y_{j}(z)(X_{j}(\gamma_{j}-z)-X_{j}(z))},\qquad&&\text{for $z\in\Omega_{j}\cup\Omega_{j+1}$},\label{eq:NinVj}
\end{align}
where
\begin{equation}\label{eq:cj_as_sum}c_{j}=\sum_{k=-L}^{j}F_{k}\end{equation}
is an explicit constant.
\end{Proposition}
\begin{proof}
Combining \eqref{eq:Vj_Hj_analytic} and \eqref{eq:V-L_A_analytic} yields \eqref{eq:Vj_in_Aandc}. Furthermore,
we know that
\[V_{j}(z)=X_{j}(z)\Lgf_{j}(Y_{j}(z))+Y_{j}(z)\Rgf_{j+1}(Y_{j}(z))\qquad\text{for $z\in\Omega_{j}\cup\Omega_{j+1}$}.\]
Now, from Proposition \ref{prop:jstuff}, we have $\Omega_{j}\cup\Omega_{j+1}=\gamma_{j}-\Omega_{j}\cup\Omega_{j+1}$ and $Y_{j}(\gamma_{j}-z)=Y_{j}(z)$, so substituting $z\to\gamma_{j}-z$ into the equation above yields
\[V_{j}(\gamma_{j}-z)=X_{j}(\gamma_{j}-z)\Lgf_{j}(Y_{j}(z))+Y_{j}(z)\Rgf_{j+1}(Y_{j}(z))\qquad\text{for $z\in\Omega_{j}\cup\Omega_{j+1}$}.\]
Combining these two equations to solve for $\Lgf_{j}(Y_{j}(z))$ and $\Rgf_{j+1}(Y_{j}(z))$ yields \eqref{eq:MinVj} and \eqref{eq:NinVj}.

Finally, it remains to determine the constant $c_{j}$. We note that $\Rgf_{j+1}(Y_{j}(z))$ must converge for $z\in\Omega_{j}\cup\Omega_{j+1}$, in particular at $z=\beta_{j}$. But $Y_{j}(\beta_{j})=0$, so the numerator of \eqref{eq:NinVj} must also be $0$, that is,
\[X_{j}(\gamma_{j}-\beta_{j})V_{j}(\beta_{j})-X_{j}(\beta_{j})V_{j}(\gamma_{j}-\beta_{j})=0,\]
so
\[c_{j}=\frac{X_{j}(\gamma_{j}-\beta_{j})\left(A(\beta_{j})-s_{j}X(\beta_{j})^{p}Y(\beta_{j})^{q}\right)-X_{j}(\beta_{j})\left(A(\gamma_{j}-\beta_{j})-s_{j}X(\gamma_{j}-\beta_{j})^{p}Y(\gamma_{j}-\beta_{j})^{q}\right)}{X_{j}(\beta_{j})-X_{j}(\gamma_{j}-\beta_{j})}.\]
This formula determines the constant for $\beta_{j}\neq \frac{\gamma_{j}}{2}$. In the case $\beta_{j}= \frac{\gamma_{j}}{2}$, the denominator of \eqref{eq:NinVj} has a double root at $z=\beta_{j}=\frac{\gamma_{j}}{2}$, so the numerator of \eqref{eq:NinVj} must also have a double root at this point. Hence, we have
\[0=2X_{j}(\beta_{j})V_{j}'(\beta_{j})-2X_{j}'(\beta_{j})V_{j}(\beta_{j}),\]
so
\[c_{j}=\frac{X_{j}(\beta_{j})}{X_{j}'(\beta_{j})}(A'(\beta_{j})-\left.(s_{j}X(z)^pY(z)^q)'\right|_{z=\beta_{j}})+s_{j}X(\beta_{j})^pY(\beta_{j})^q-A(\beta_{j}).\]
\end{proof}
\section{Nature of series in an $M$-quadrant cone}
\label{sec:bc_solutions}
In this section we use the same methods as we used in Section \ref{sec:nature} to analyse the generating function of walks in the 3-quadrant cone. In particular we characterise the nature of the series $\Qgf_{j}(x,y;t)$. As we will show, for $M$ odd this nature is exactly the same as the nature of $\Cgf(x,y;t)$ which counts walks in the three-quadrant cone. For $M$ even the situation is different: the generating function is always D-finite, and it is algebraic if and only if $q=0$ and $L$ is even or, if we allow starting points outside the first quadrant, $p=0$ and $L$ is odd. These results are summarised in Figure \ref{fig:Characterisation_flowchart}.

In the previous section we reduced the problem to finding the unique meromorphic functions $A,B:\mathbb{C}\to\mathbb{C}\cup\{\infty\}$ and constant $\Fgf(t)$ characterised by Theorem \ref{thm:bc_PH_PV_characterisation} (for each $t$). This is a generalisation of Theorem \ref{thm:PH_PV_characterisation} as well as an analogous Theorem found by Raschel for walks in the quarter plane \cite{raschel2012counting}, which corresponds to $M=1$. 


\begin{figure}[ht]
\centering
   \includegraphics[scale=0.9]{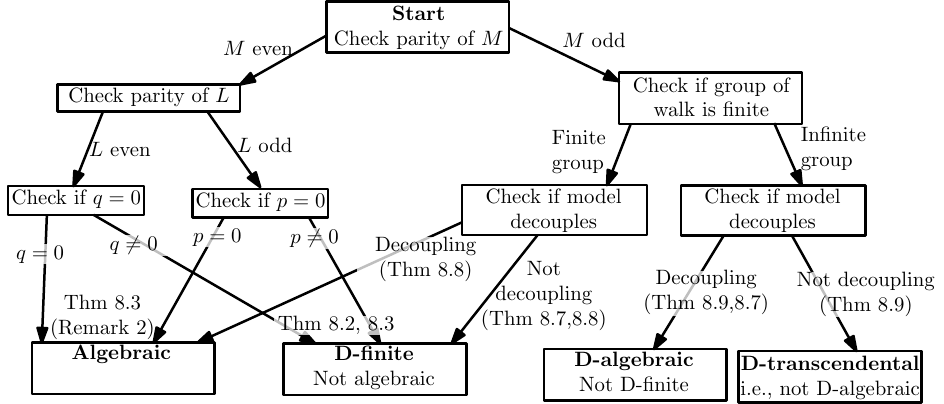}
   \caption{A chart showing how to determine the complexity of the series $\Qgf_{j}(x,y;t)$ counting walks starting at $(p,q)$ confined to an $M$-quadrant cone as a series in $x$ (or $y$) as proven by Theorems \ref{thm:bc_even_cone_D-finite}-\ref{thm:bc_D-algebraic_general_t}. Recall that $L$ determines which $M$ quadrants are used to define the $M$-quadrant cone.}
   \label{fig:Characterisation_flowchart}
\end{figure}

A first general result relates the complexity of $\Qgf_{j}(x,y;t)$ to that of the function $A(z)$ and $B(z)$ using the properties $X$-algebraic, $X$-D-finite and $X$-D-algebraic defined in appendix \ref{ap:nature_analytic}:
\begin{Lemma}\label{lem:bc_QjAB_complexities_equivalence}
Assume $\mathcal{P}$ is one of the properties {\em algebraic}, {\em D-finite} or {\em D-algebraic} and $t\in\left(0,\frac{1}{\Pgf(1,1)}\right)$ is fixed. The following are equivalent:
\begin{enumerate}[label={\rm(\roman*)},ref={\rm(\roman*)}]
\item The function $\Qgf_{j}(x,y;t)$  has property $\mathcal{P}$ as a function of $x$, \label{bc_propertyx}
\item The function $\Qgf_{j}(x,y;t)$  has property $\mathcal{P}$ as a function of $y$, \label{bc_propertyy}
\item $A(z)$ has property $X$-$\mathcal{P}$ (see Appendix \ref{ap:nature_analytic}), \label{bc_propertyHz}
\item $B(z)$ has property $X$-$\mathcal{P}$.\label{bc_propertyVz}
\end{enumerate}
\end{Lemma}
\begin{proof}
The equivalence $\ref{bc_propertyHz}\iff\ref{bc_propertyVz}$ follows from \eqref{eq:bc_PH_PV_nice}, since by Proposition \ref{prop:XP_closure_properties}, the property $X$-$\mathcal{P}$ is closed under addition. To show that these are equivalent to $\ref{bc_propertyx}$ and $\ref{bc_propertyy}$, we will start by considering two further equivalent properties, thinking of $\Qgf_{j}(x,y;t)$ as a function of $x_{j}$ and $y_{j}$ rather than $x$ and $y$:
\begin{enumerate}[label={\rm(\roman*')},ref={\rm(\roman*')}]
\item The function $\Qgf_{j}(x,y;t)$  has property $\mathcal{P}$ as a function of $x_{j}$, \label{bc_propertyxj}
\item The function $\Qgf_{j}(x,y;t)$  has property $\mathcal{P}$ as a function of $y_{j}$. \label{bc_propertyyj}
\end{enumerate}
Indeed if $j$ is even then $x_{j}\in\{x,\frac{1}{x}\}$ and $y_{j}\in\{y,\frac{1}{y}\}$ so \ref{bc_propertyxj}$\iff$\ref{bc_propertyx} and \ref{bc_propertyyj}$\iff$\ref{bc_propertyy}. If $j$ is odd then $x_{j}\in\{y,\frac{1}{y}\}$ and $y_{j}\in\{x,\frac{1}{x}\}$ so \ref{bc_propertyxj}$\iff$\ref{bc_propertyy} and \ref{bc_propertyyj}$\iff$\ref{bc_propertyx}. In either case it suffices to prove that the equivalent properties $\ref{bc_propertyHz}$ and $\ref{bc_propertyVz}$ are also equivalent to \ref{bc_propertyxj} and \ref{bc_propertyyj}. In fact we will just prove that \ref{bc_propertyyj} is equivalent to \ref{bc_propertyVz} as the proof that \ref{bc_propertyxj} is equivalent to \ref{bc_propertyHz} is essentially the same.

For $j<K$, it follows from \eqref{eq:Vj_Hj_series} and \eqref{eq:V-L_A_series} that \ref{bc_propertyyj} holds if and only if $x_{j}L_{j}(y_{j})-y_{j}R_{j+1}(y_{j})$ has property $\mathcal{P}$ as a function of $y_{j}$ since all other terms are linear in $y_{j}$. Moreover, this occurs if and only if $L_{j}(y_{j})$ and $R_{j+1}(y_{j})$ both have property $\mathcal{P}$ in $y_{j}$. If $y_{j}\in\{x,\frac{1}{x}\}$, then they have this property with respect to $x$, while if $y_{j}\in\{y,\frac{1}{y}\}$, then they have this property with respect to $y$.
Either way, by Proposition \ref{prop:PandXP}, this occurs if and only if the expressions \eqref{eq:MinVj} and \eqref{eq:NinVj} determining $L_{j}(Y_{j}(z))$ and $R_{j+1}(Y_{j}(z))$ each have property $X$-$\mathcal{P}$.
From these equations, combined with Proposition \ref{prop:XP_closure_properties}, this occurs if and only if $V_{j}(z)$ has this property.
Finally, by \eqref{eq:Vj_in_Aandc}, $V_{j}(z)$ has property $X$-$\mathcal{P}$ if and only if $A(z)$ has this property, which is equivalent to \ref{bc_propertyHz}.
So we have proven that for $j<K$, $\ref{bc_propertyyj}\iff\ref{bc_propertyHz}$. For $j=K$, \eqref{eq:HK_B_series} shows that \ref{bc_propertyyj} holds if and only if $\Bgf(y_{K};t)$ has property $\mathcal{P}$ in $y_{K}$. By Proposition  
\ref{prop:PandXP}, this is equivalent to \ref{bc_propertyVz}, so $\ref{bc_propertyyj}\iff\ref{bc_propertyVz}$. Hence, for all $j$, we have $\ref{bc_propertyyj}\iff\ref{bc_propertyHz}\iff\ref{bc_propertyVz}$. Finally, proving that \ref{bc_propertyxj} is equivalent to $\ref{bc_propertyHz}$ and $\ref{bc_propertyVz}$ is similar (noting that $X$-$\mathcal{P}$ and $Y$-$\mathcal{P}$ are equivalent).
\end{proof}

For $M$ even the situation is relatively simple, as described by the following theorems. In this case the $M$-quadrant cone on which the walks occur could alternatively be understood as a cone constructed from $M/2$ half planes, perhaps partially explaining why the complexity should match the complexity of half-plane walks.
\begin{Theorem}\label{thm:bc_even_cone_D-finite}
Assume $M\geq 2$ is even. For any integer $j\in[-L,K]$, the series $\Qgf_{j}(x,y;t)$ is D-finite in $x$ and $y$. 
\end{Theorem}
For fixed $t\in\left(0,\frac{1}{\Pgf(1,1)}\right)$, this theorem is proven in Theorem \ref{thm:bc_finite_group_case_thm} and Lemma \ref{lem:bc_even_D-finite}. Lemma \ref{lem:function_and_series_complexity} then implies that the series $\Qgf_{j}(x,y;t)\in\mathbb{R}[x,\frac{1}{x},y,\frac{1}{y}][[t]]$ is D-finite in $x$ and $y$. 

\begin{Theorem}\label{thm:bc_even_cone_algebraic}
Assume $M\geq 2$ is even. For any integer $j\in[-L,K]$ and fixed $t_{c}\in\left(0,\frac{1}{\Pgf(1,1)}\right)$, the following are equivalent:
\begin{enumerate}[label={\rm(\roman*)},ref={\rm(\roman*)}]
\item The series $\Qgf_{j}(x,y;t)$ is algebraic in $x$, \label{bc_even_algx_ser}
\item The series $\Qgf_{j}(x,y;t)$ is algebraic in $y$, \label{bc_even_algy_ser}
\item The function $\Qgf_{j}(x,y;t_{c})$ is algebraic in $x$, \label{bc_even_algx_func}
\item The function $\Qgf_{j}(x,y;t_{c})$ is algebraic in $y$, \label{bc_even_algy_func}
\item $L$ is even (so $K$ is odd) and $q=0$, that is, the walks start on the $x$-axis. \label{bc_Leven_q0}
\end{enumerate}
\end{Theorem}
\begin{proof}The equivalences $\ref{bc_even_algx_func}\iff\ref{bc_Leven_q0}$ and $\ref{bc_even_algy_func}\iff\ref{bc_Leven_q0}$ follow from Theorem \ref{thm:bc_finite_group_case_thm} and Lemma \ref{lem:bc_even_algebraic}. So  $\ref{bc_even_algx_func}$ and $\ref{bc_even_algy_func}$ each hold for some $t_{c}$ is and only if they hold for all $t_{c}$. Finally by Lemma \ref{lem:function_and_series_complexity}, this is equivalent to $\ref{bc_even_algx_ser}$ and $\ref{bc_even_algy_ser}$.\end{proof}

\textbf{Remark 1:} One could alternatively consider the nature of the complete generating function $\sum_{j=-L}^{K}\Qgf_{j}(x,y;t)$ rather than the individual generating functions $\Qgf_{j}$, although for $M>4$ this introduces some ambiguity as to how many paths end at each point in the cone. While it is clear that if each $\Qgf_{j}(x,y;t)$ is algebraic then the complete generating function is algebraic, the converse does not hold, as our result implies that even for $M=2$ the generating functions $\Qgf_{j}$ are not algebraic, even though the complete generating function in this case is a generating function for half-plane walks, which is therefore algebraic.

\textbf{Remark 2:} This theorem does not include the possibility that $p=0$ because we assumed that the walks start in the quadrant $\tilde{\cQ}_{0}$, where $p>0$. More generally, if the walk is allowed to start at any point $(s,(p,q))\in\Pi_{L,K}$, we can deduce that the generating function is algebraic if and only if $(p,q)$ lies on the same axis as the two boundaries of the cone. The $s=0$ case is precisely Theorem \ref{thm:bc_even_cone_algebraic}. Moreover, this statement is equivalent for any $s$ as the cone and starting point can simply be rotated sending each point $(k,(a,b))\mapsto (k+1,(-b,a))$ (or the inverse of this transformation)  until the starting point lies in $\Gamma_{0}$. Hence, allowing any starting point $(s,(p,q))$, the generating function $\Qgf_{j}(x,y;t)$ is algebraic if and only if $(p,q)$ lies on the same axis as the two boundaries of the cone, that is, if and only if either $q$ is $0$ and $L$ is even or $p$ is $0$ and $L$ is odd.

\begin{Theorem}\label{thm:bc_D-finite_fixed_t}
Assume $M\geq 3$ is odd. For fixed $t\in\left(0,\frac{1}{\Pgf(1,1)}\right)$ and any integer $j\in[-L,K]$, the following are equivalent
\begin{enumerate}[label={\rm(\roman*)},ref={\rm(\roman*)}]
\item The function $\Qgf_{j}(x,y;t)$  is D-finite in $x$, \label{bc_D-finitex}
\item The function $\Qgf_{j}(x,y;t)$  is D-finite in $y$, \label{bc_D-finitey}
\item $A(z)$ satisfies a linear differential equation whose coefficients are elliptic functions with periods $\pi$ and $\pi\tau$, \label{bc_D-finiteHz}
\item $B(z)$ satisfies a linear differential equation whose coefficients are elliptic functions with periods $\pi$ and $\pi\tau$, \label{bc_D-finiteVz}
\item the ratio $\frac{\gamma}{\pi\tau}\in\mathbb{Q}$, \label{bc_D-finiterational}
\item the orbit of each point $(x,y)\in\overline{E_{t}}$ under the group of the walk is finite. \label{bc_D-finiteorbits}
\end{enumerate}
\end{Theorem}
The equivalence $\ref{bc_D-finitex}\iff\ref{bc_D-finitey}\iff\ref{bc_D-finiteHz}\iff\ref{bc_D-finiteVz}$ is precisely the statement of Lemma \ref{lem:bc_QjAB_complexities_equivalence} where $\mathcal{P}$ is the property {\em D-finite}.
Later in this section we complete the proof of the theorem by showing that the first $4$ conditions are equivalent to \ref{bc_D-finiterational} and \ref{bc_D-finiteorbits}. We define the group of the walk in Appendix \ref{ap:group}, the equivalence of \ref{bc_D-finiterational} and \ref{bc_D-finiteorbits} is shown there in Proposition \ref{prop:finite_group_fixed_t}. We show that these equivalent conditions imply the conditions \ref{bc_D-finitex}-\ref{bc_D-finiteVz} in a combination of Theorem \ref{thm:bc_finite_group_case_thm} and Lemma \ref{lem:bc_odd_D-finite}, while we show the converse in Corollary \ref{cor:bc_inf_group_non-D-finite_Mgeq3}.

We have not been able to show the theorem above for $M=1$, meaning our classification of the complexity of $\Qgf_{j}(x,y;t)$ for $t$ fixed is incomplete. Nonetheless, in Theorem \ref{thm:bc_D-finite_general_t}, we are able to fully classify the complexity of the series $\Qgf_{j}(x,y;t)\in\mathbb{R}[x_{j},y_{j}][[t]]$.

\begin{Theorem}\label{thm:bc_algebraic_fixed_t}
Assume that $M\geq 1$ is odd and $\frac{\gamma}{\pi\tau}\in\mathbb{Q}$. For fixed $t\in\left(0,\frac{1}{\Pgf(1,1)}\right)$ and any integer $j\in[-L,K]$, the following are equivalent
\begin{enumerate}[label={\rm(\roman*)},ref={\rm(\roman*)}]
\item The function $\Qgf_{j}(x,y;t)$ is algebraic in $x$, \label{bc_algx}
\item The function $\Qgf_{j}(x,y;t)$ is algebraic in $y$, \label{bc_algy}
\item $A(z)$ has $m\pi\tau$ as a period for some positive integer $m$, \label{bc_algHz}
\item $B(z)$ has $m\pi\tau$ as a period for some positive integer $m$, \label{bc_algVz}
\item There are rational functions $R_{1}$ and $R_{2}$ satisfying $x^{p}y^{q}=R_{1}(x)+R_{2}(y)$ for all $(x,y)\in\overline{E_{t}}$, \label{bc_algdec}
\item There are rational functions $R_{1}$ and $R_{2}$ satisfying $X(z)^{p}Y(z)^{q}=R_{1}(X(z))+R_{2}(Y(z))$ for all $z\in\mathbb{C}$, \label{bc_algdecz}
\item The orbit sum of the model is $0$. \label{bc_algorbitsum}
\end{enumerate}
\end{Theorem}
The equivalence $\ref{bc_algx}\iff\ref{bc_algy}\iff\ref{bc_algHz}\iff\ref{bc_algVz}$ is precisely the statement of Lemma \ref{lem:bc_QjAB_complexities_equivalence} where $\mathcal{P}$ is the property {\em algebraic}.
The equivalence $\ref{bc_algdec}\iff\ref{bc_algdecz}$ is due to the parameterisation \eqref{eq:Eparam} of $\overline{E_{t}}$, while the equivalence $\ref{bc_algdecz}\iff\ref{bc_algorbitsum}$ follows from Proposition \ref{prop:orbit_sum0isdecoupling}.
 To complete the proof of this Theorem, we show that equivalent conditions $\ref{bc_algx}-\ref{bc_algVz}$ are equivalent to $\ref{bc_algdec}-\ref{bc_algorbitsum}$ in Theorem \ref{thm:bc_finite_group_case_thm}.

\begin{Theorem}\label{thm:bc_D-algebraic_fixed_t}
Assume that $M\geq 1$ is odd. Fix $t\in\left(0,\frac{1}{\Pgf(1,1)}\right)$ and assume that $\frac{\gamma}{\pi\tau}\notin\mathbb{Q}$. For any integer $j\in[-L,K]$, the following are equivalent
\begin{enumerate}[label={\rm(\roman*)},ref={\rm(\roman*)}]
\item The function $\Qgf_{j}(x,y;t)$ is D-algebraic in $x$, \label{bc_D-algx}
\item The function $\Qgf_{j}(x,y;t)$  is D-algebraic in $y$, \label{bc_D-algy}
\item $A(z)$ is D-algebraic in $z$, \label{bc_D-algHz}
\item $B(z)$ is D-algebraic in $z$, \label{bc_D-algVz}
\item There are rational functions $R_{1}$ and $R_{2}$ satisfying $x^{p}y^{q}=R_{1}(x)+R_{2}(y)$ for all $(x,y)\in\overline{E_{t}}$, \label{bc_D-algdec}
\item There are rational functions $R_{1}$ and $R_{2}$ satisfying $X(z)^{p}Y(z)^{q}=R_{1}(X(z))+R_{2}(Y(z))$ for all $z\in\mathbb{C}$. \label{bc_D-algdecz}
\end{enumerate}
\end{Theorem}
The equivalence $\ref{bc_algx}\iff\ref{bc_algy}\iff\ref{bc_algHz}\iff\ref{bc_algVz}$ is precisely the statement of Lemma \ref{lem:bc_QjAB_complexities_equivalence} where $\mathcal{P}$ is the property {\em D-algebraic}.
The equivalence $\ref{bc_D-algdec}\iff\ref{bc_D-algdecz}$ is due to the parameterisation \eqref{eq:Eparam} of $\overline{E_{t}}$.
We complete the proof of this theorem later in this section, starting with Theorem \ref{thm:bc_D-alg_proof} which shows that the equivalent conditions $\ref{bc_D-algdec}\iff\ref{bc_D-algdecz}$ imply the equivalent conditions $\ref{bc_D-algx}$-$\ref{bc_D-algVz}$, then we show the reverse implication in Theorem \ref{thm:bc_non-D-alg_proof}.

Finally, the following three Theorems characterise the nature of the {\em series} $\Qgf(x,y;t)$ for $M$ odd. 

\begin{Theorem}\label{thm:bc_D-finite_general_t}
Assume that $M\geq 1$ is odd. For any integer $j\in[-L,K]$, the following are equivalent
\begin{enumerate}[label={\rm(\roman*)},ref={\rm(\roman*)}]
\item The series $\Qgf_{j}(x,y;t)\in\mathbb{R}[x,y][[t]]$  is D-finite in $x$, \label{bc_series_D-finitex}
\item The series $\Qgf_{j}(x,y;t)\in\mathbb{R}[x,y][[t]]$  is D-finite in $y$, \label{bc_series_D-finitey}
\item Each of the conditions of Theorem \ref{thm:bc_D-finite_fixed_t} hold for all $t\in\left(0,\frac{1}{\Pgf(1,1)}\right)$,\label{bc_all_D-finite}
\item The group of the walk is finite.\label{bc_finite_group}
\end{enumerate}
\end{Theorem}
\begin{proof}
We start with condition \ref{bc_all_D-finite}, that any given condition of Theorem \ref{thm:bc_D-finite_fixed_t} holds for all $t$ if and only if any other given condition of Theorem \ref{thm:bc_D-finite_fixed_t} holds for all $t$. For $M\geq3$ this is true because Theorem \ref{thm:bc_D-finite_fixed_t} itself holds. For $M=1$, conditions \ref{bc_D-finitex}-\ref{bc_D-finiteVz} of Theorem \ref{thm:bc_D-finite_fixed_t} are equivalent and conditions \ref{bc_D-finiterational} and \ref{bc_D-finiteorbits} are equivalent and imply conditions \ref{bc_D-finitex}-\ref{bc_D-finiteVz} for the same reasons as in the proof of Theorem \ref{thm:bc_D-finite_fixed_t}. Moreover, if the equivalent conditions 
\ref{bc_D-finitex}-\ref{bc_D-finiteVz} of Theorem \ref{thm:bc_D-finite_fixed_t} hold for all $t$, then by Corollary \ref{cor:bc_inf_group_non-D-finite_M1}, the group of the walk is finite, so the equivalent conditions \ref{bc_D-finiterational} and \ref{bc_D-finiteorbits} hold for all $t$. Hence condition \ref{bc_all_D-finite} is well defined.

By Lemma \ref{lem:function_and_series_complexity}, Condition \ref{bc_D-finitex} of Theorem \ref{thm:bc_D-finite_fixed_t} holds for all $t$ if and only if \ref{bc_series_D-finitex} holds, that is $\ref{bc_series_D-finitex}\iff\ref{bc_all_D-finite}$. Similarly, $\ref{bc_series_D-finitey}\iff\ref{bc_all_D-finite}$. Finally, by Proposition \ref{prop:finite_group_general_t}, Condition \ref{bc_D-finiteorbits} of Theorem \ref{thm:bc_D-finite_fixed_t} holds for all $t$ if and only if \ref{bc_finite_group} holds, that is $\ref{bc_all_D-finite}\iff\ref{bc_finite_group}$.
\end{proof}

\begin{Theorem}\label{thm:bc_algebraic_general_t}
Assume that $M\geq 1$ is odd. For any integer $j\in[-L,K]$, the following are equivalent
\begin{enumerate}[label={\rm(\roman*)},ref={\rm(\roman*)}]
\item The series $\Qgf_{j}(x,y;t)\in\mathbb{R}[x,y][[t]]$  is algebraic in $x$, \label{bc_series_algx}
\item The series $\Qgf_{j}(x,y;t)\in\mathbb{R}[x,y][[t]]$  is algebraic in $y$, \label{bc_series_algy}
\item The equivalent conditions of Theorem \ref{thm:bc_algebraic_fixed_t} hold for all $t\in\left(0,\frac{1}{\Pgf(1,1)}\right)$.
\end{enumerate}
\end{Theorem}
 
\begin{Theorem}\label{thm:bc_D-algebraic_general_t}
Assume that $M\geq 1$ is and the group of the walks is infinite. For any integer $j\in[-L,K]$, the following are equivalent
\begin{enumerate}[label={\rm(\roman*)},ref={\rm(\roman*)}]
\item The series $\Qgf_{j}(x,y;t)\in\mathbb{R}[x,y][[t]]$  is D-algebraic in $x$, \label{bc_series_D-algx}
\item The series $\Qgf_{j}(x,y;t)\in\mathbb{R}[x,y][[t]]$  is D-algebraic in $y$, \label{bc_series_D-algy}
\item The equivalent conditions of Theorem \ref{thm:bc_D-algebraic_fixed_t} hold for all $t\in\left(0,\frac{1}{\Pgf(1,1)}\right)$.
\end{enumerate}
\end{Theorem}

For each of these theorems, the equivalences $\text{(i)}\iff\text{(iii)}$ and $\text{(ii)}\iff\text{(iii)}$ are both due to Lemma \ref{lem:function_and_series_complexity}.

\subsection{Finite group cases}\label{sec:bc_finite_group}
Define $\hat{\tau}:=\frac{\gamma_{K}-\gamma_{-L-1}}{\pi}$. In this section, we consider the cases in which $\frac{\hat{\tau}}{\tau}\in\mathbb{Q}$. We will show that this occurs for fixed $t$ if and only if $\Qgf_{j}(x,y;t)$ is D-finite in $x$. We will also show that the nature of this restriction depends on the parity of $M$: If $M$ is even then we always have $\frac{\hat{\tau}}{\tau}\in\mathbb{Q}$, while if $M$ is odd then $\frac{\hat{\tau}}{\tau}\in\mathbb{Q}$ if and only if $\frac{\gamma}{\pi\tau}\in\mathbb{Q}$. We also describe precisely in which cases $\Qgf_{j}(x,y;t)$ is algebraic in $x$.

\begin{Theorem}\label{thm:bc_finite_group_case_thm} If $\frac{\hat{\tau}}{\tau}=\frac{N_{1}}{N_{2}}\in\mathbb{Q}$, then $\Qgf_{j}(x,y)$ is D-finite in $x$ and $y$. Moreover, under this assumption $\Qgf_{j}(x,y)$ is algebraic in $x$ if and only if the function
\[\tilde{E}(z):=\sum_{j=0}^{N_{2}-1}\tilde{J}(j\pi\hat{\tau}+z)\]
is equal to $0$. Similarly $\Qgf_{j}(x,y)$ is algebraic in $y$ if and only if $\tilde{E}(z)=0$.\end{Theorem}
\begin{proof}Assume that $\frac{\hat{\tau}}{\tau}=\frac{N_{1}}{N_{2}}$ for some positive $N_{1},N_{2}\in\mathbb{Z}$. Now consider \eqref{eq:bc_PH_qdiff}:
\[B\left(\frac{N_{1}}{N_{2}}\pi\tau+z\right)-B(z)=\tilde{J}(z).\]
 Taking a telescoping sum of $N_{2}$ copies of this equation yields
\begin{equation}\label{eq:bc_PHandE}B(\pi\tau N_{1}+z)-B(z)=\sum_{j=0}^{N_{2}-1}\tilde{J}\left(j\frac{N_{1}}{N_{2}}\pi\tau+z\right)=\tilde{E}(z).\end{equation}

We will now consider the cases $\tilde{E}(z)=0$ and $\tilde{E}(z)\neq0$ separately. In the case that $\tilde{E}(z)=0$, we have
\[B(\pi\tau N_{1}+z)=B(z),\]
so $B(z)$ is $X$-algebraic (see Definition \ref{def:X-alg}). Hence by Lemma \ref{lem:bc_QjAB_complexities_equivalence}, $\Qgf_{j}(x,y)$ is algebraic in $x$ and $y$.

Finally we consider the case $\tilde{E}(z)\neq0$. Then from \eqref{eq:bc_PHandE}, we have
\begin{equation}\label{eq:bc_PH_holo}\frac{B(\pi\tau N_{1}+z)}{\tilde{E}(\pi\tau N_{1}+z)}-\frac{B(z)}{\tilde{E}(z)}=\frac{B(2\pi\tau N_{1}+z)}{\tilde{E}(z)}-\frac{B(z)}{\tilde{E}(z)}=1,\end{equation}
so the function
\begin{equation}\label{eq:bc_F_def}F(z):=\frac{\partial}{\partial z}\frac{B(z)}{\tilde{E}(z)}=\frac{1}{\tilde{E}(z)^{2}}(B'(z)\tilde{E}(z)-B(z)\tilde{E}'(z))\end{equation}
satisfies
\[F(2\pi\tau N_{1}+z)-F(z)=0.\]
Hence, $B(z)$ is weakly $X$-D-finite (see Definition \ref{def:X-D-finite-weak}), so by proposition \ref{prop:D-finiteofXorY}, it is $X$-D-finite. Therefore, by Lemma \ref{lem:bc_QjAB_complexities_equivalence}, the generating function $\Qgf_{j}(x,y)$ is D-finite in $x$ and $y$.

Finally we show that in this $\tilde{E}(z)\neq 0$ case, $\Qgf_{j}(x,y)$ is not algebraic in $x$ or $y$. Suppose the contrary, then by Lemma \ref{lem:bc_QjAB_complexities_equivalence}, $B(z)$ is $X$-algebraic, that is $m\pi\tau$ is a period of $B(z)$ for some positive integer $m$. But this is impossible as by \eqref{eq:bc_PH_holo},
\[\frac{B(m\pi\tau N_{1}+z)-B(z)}{\tilde{E}(z)}=\frac{B(m\pi\tau N_{1}+z)}{\tilde{E}(m\pi\tau N_{1}+z)}-\frac{B(z)}{\tilde{E}(z)}=m\neq0.\]
\end{proof}

To finish the proof of Theorem \ref{thm:bc_D-finite_fixed_t}, we need the following Lemma, which specifically relates to the case when $M$ is odd. We note that the Lemma below implies that for $M$ odd, the value of $N_{2}$ in this section is the same as the value $N$ used in Section \ref{sec:finite_group}.

\begin{Lemma}\label{lem:bc_odd_D-finite}
If $M\geq 1$ is odd, then $\frac{\hat{\tau}}{\tau}\in\mathbb{Q}$ if and only if $\frac{\gamma}{\pi\tau}\in\mathbb{Q}$. Moreover, for an integer $N_{2}$, we have $N_{2}\frac{\hat{\tau}}{\tau}\in\mathbb{Z}$ if and only if $N_{2}\frac{2\gamma}{\pi\tau}\in\mathbb{Z}$.
\end{Lemma}
\begin{proof}
As we will show, this follows easily from the definition of $\gamma_{j}$ (see Definition \ref{def:jstuff}). Assuming $M=K+L+1$ is odd, $K$ and $L$ must have the same parity. If they are both even, then $\gamma_{-L-1}=-\gamma-\pi\tau\frac{L}{2}$ and $\gamma_{K}=\gamma+\pi\tau\frac{K}{2}$, so
\[\frac{\hat{\tau}}{\tau}=\frac{\gamma_{K}-\gamma_{-L-1}}{\pi\tau}=\frac{2\gamma}{\pi\tau}+\frac{K+L}{2},\]
so $\frac{\hat{\tau}}{\tau}\in\mathbb{Q}$ if and only if $\frac{\gamma}{\pi\tau}\in\mathbb{Q}$ and $N_{2}\frac{\hat{\tau}}{\tau}\in\mathbb{Z}$ if and only if $N_{2}\frac{2\gamma}{\pi\tau}\in\mathbb{Z}$.

If $L$ and $K$ are both odd then $\gamma_{-L-1}=\gamma-\pi\tau\frac{L+1}{2}$ and $\gamma_{K}=-\gamma+\pi\tau\frac{K+1}{2}$, so
\[\frac{\hat{\tau}}{\tau}=\frac{\gamma_{K}-\gamma_{-L-1}}{\pi\tau}=-\frac{2\gamma}{\pi\tau}+\frac{K+L+2}{2},\]
so $\frac{\hat{\tau}}{\tau}\in\mathbb{Q}$ if and only if $\frac{\gamma}{\pi\tau}\in\mathbb{Q}$ and $N_{2}\frac{\hat{\tau}}{\tau}\in\mathbb{Z}$ if and only if $N_{2}\frac{2\gamma}{\pi\tau}\in\mathbb{Z}$.
\end{proof}
To finish the proof of Theorem \ref{thm:bc_even_cone_D-finite} we use the following Lemma
\begin{Lemma}\label{lem:bc_even_D-finite}
If $M\geq 2$ is even then $\frac{\hat{\tau}}{\tau}=\frac{M}{2}\in\mathbb{Z}$.
\end{Lemma}
\begin{proof}
From the definition of $\gamma_{j}$ (Definition \ref{def:jstuff}), we always have $\gamma_{j+2}=\gamma_{j}+\frac{\pi\tau}{2}$. Hence, if $M=K+L+1$ is even, then $\gamma_{K}=\gamma_{-L-1}+\frac{M}{2}\pi\tau$. Hence
$\frac{\hat{\tau}}{\tau}=\frac{\gamma_{K}-\gamma_{-L-1}}{\pi\tau}=\frac{M}{2},$
as required.
\end{proof}

To finish the proof of Theorem \ref{thm:bc_algebraic_fixed_t} we use the following lemma
\begin{Lemma}\label{lem:bc_odd_algebraic}
Define $F(z):=X(z)^{p}Y(z)^{q}$. If $M\geq 1$ is odd, then $\tilde{E}(z)$ is equal to the orbit sum
\[E(z):=\sum_{j=0}^{N_{2}-1}F((2j+1)\gamma-z)-F(2j\gamma+z)\]
from Definition \ref{defn:orbit_sum} and used throughout Subsection \ref{sec:finite_group}.
\end{Lemma}
\begin{proof}Recall from \eqref{def:tildej} that $\tilde{J}(z)=F(\gamma_{-L-1}-z)-F(z)$. Hence $\tilde{E}(z)$ can be written in terms of $F(z)$ as
\begin{align*}
\tilde{E}(z)&=\sum_{j=0}^{N_{2}-1}F((j+1)\gamma_{-L-1}-j\gamma_{K}-z)-F(j(\gamma_{K}-\gamma_{-L-1})+z).
\end{align*}
We will now consider two cases.\newline
\textbf{Case 1:} $L$ is even.\newline
In this case $K=M-L-1$ is also even, so by Definition \ref{def:jstuff}, $\gamma_{K}=\gamma+\frac{K}{2}\pi\tau$, while $\gamma_{-L-1}=-\gamma-\frac{L}{2}\pi\tau$. Since $\pi\tau$ is a period of $F(z)$, this implies that
\[\tilde{E}(z)=\sum_{j=0}^{N_{2}-1}F(-(2j+1)\gamma-z)-F(2j\gamma+z)=E(z).\]
\textbf{Case 2:} $L$ is odd.\newline
In this case $K=M-L-1$ is also odd, so by Definition \ref{def:jstuff}, $\gamma_{K}=-\gamma+\frac{K+1}{2}\pi\tau$, while $\gamma_{-L-1}=\gamma-\frac{L+1}{2}\pi\tau$. Since $\pi\tau$ is a period of $F(z)$, this implies that
\[\tilde{E}(z)=\sum_{j=0}^{N_{2}-1}F((2j+1)\gamma-z)-F(-2j\gamma+z).\]
Now, since $2\gamma N_{2}\in\pi\tau\mathbb{Z}$, we can rewrite this as
\[\tilde{E}(z)=\left(\sum_{j=0}^{N_{2}-1}F((2(j-N_{2})+1)\gamma-z)\right)-\left(\sum_{j=0}^{N_{2}-1}F(2(N_{2}-j)\gamma+z)\right).\]
Replacing $j$ by $N_{2}-1-j$ in the first sum and $N_{2}-j$ in the second sum yields
\[\tilde{E}(z)=\left(\sum_{j=0}^{N_{2}-1}F((-2j-1)\gamma-z)\right)-\left(\sum_{j=1}^{N_{2}}F(2j\gamma+z)\right)=E(z).\]
\end{proof}

\begin{Lemma}\label{lem:bc_even_algebraic}
If $M\geq 2$ is even, then $\tilde{E}(z)$ is the zero function if and only if $q=0$ and $L$ is even.
\end{Lemma}
\begin{proof}
Assume $M$ is even. From Lemma \ref{lem:bc_even_D-finite}, we know that $\frac{\hat{\tau}}{\tau}=\frac{M}{2}\in\mathbb{Z}$, so
\[\tilde{E}(z):=\sum_{j=0}^{N_{2}-1}\tilde{J}(j\pi\hat{\tau}+z)=\sum_{j=0}^{N_{2}-1}\tilde{J}(z)=N_{2}\tilde{J}(z),\]
so $\tilde{E}(z)=0$ if and only if $\tilde{J}(z)=0$. Now recall the definition \eqref{def:tildej} of $\tilde{J}(z)$: 
\[\tilde{J}(z)=X(\gamma_{-L-1}-z)^{p}Y(\gamma_{-L-1}-z)^{q}-X(z)^{p}Y(z)^{q}.\]
In the case that $L$ is even, we have $\gamma_{-L-1}=-\gamma-\frac{L}{2}\pi\tau$, so
\[\tilde{J}(z)=X(-\gamma-z)^{p}Y(-\gamma-z)^{q}-X(z)^{p}Y(z)^{q}=X(z)^{p}\left(Y(-\gamma-z)^{q}-Y(z)^{q}\right).\]
So we see immediately that if $q=0$, then $\tilde{J}(z)=0$. It remains to prove the converse. So, assume for the sake of contradiction that $q\neq 0$ but $\tilde{J}(z)=0$ for all $z$. then we have
\[Y(-\gamma-z)^{q}=Y(z)^{q}\]
for all $z$. Hence $\frac{Y(z)}{Y(-\gamma-z)}$ is a $q$th root of unity for each $z$, so in fact it must be a constant $c$. So $Y(z)=cY(-\gamma-z)$ for all $z$ and substituting $z\to -\gamma-z$ yields $Y(-\gamma-z)=cY(z)$, so $c^2=1$, that is either $Y(z)=Y(-\gamma-z)=Y(2\gamma+z)$ or $Y(z)=-Y(-\gamma-z)=-Y(2\gamma+z)$. In the first case $2\gamma$ is a period of $Y$ which is impossible because $0<\Im(2\gamma)<\Im(\pi\tau)$ by Lemma \ref{lem:param}. So we are left with the case $Y(z)=-Y(-\gamma-z)$. Now, to compare $Y(z)$ and $Y(-\gamma-z)$ in a different way, we consider the equation $y\Pgf(X(z),y)-\frac{y}{t}=0$, which is a quadratic equation in $y$ with solutions $y=Y(z)$ and $y=Y(-\gamma-z)$. Now define
\[A_{-1}(x)+yA_{0}(x)+y^{2}A_{1}(x):=\Pgf(x,y),\]
that is $A_{j}(x)=w_{(-1,j)}x^{-1}+w_{(0,j)}+w_{(1,j)}x$. Then
we have
\[y^{-1}A_{-1}(X(z))+\left(A_{0}(X(z))-\frac{1}{t}\right)+yA_{1}(X(z))=y\Pgf(X(z),y)-\frac{y}{t}=A_{1}(X(z))(y-Y(z))(y-Y(-\gamma-z)).\]
Hence
\begin{align*}
A_{0}(X(z))-\frac{1}{t}&=-A_{1}(X(z))(Y(z)+Y(-\gamma-z)),\\
A_{-1}(X(z))&=A_{1}(X(z))Y(z)Y(-\gamma-z).
\end{align*}
But in this case $Y(z)=-Y(-\gamma-z)$, so $A_{0}(X(z))-\frac{1}{t}=0$ for all $z$, or equivalently $A_{0}(x)-\frac{1}{t}=0$, which is impossible because $[x^{0}]A_{0}(x)=w_{(0,0)}=0$ as $(0,0)\notin S$.

Finally we consider the remaining case where $L$ is odd, so $\gamma_{-L-1}=\gamma-\frac{L+1}{2}\pi\tau$. It remains to prove that $\tilde{J}(z)$ is not the zero function in this case. We have
\[\tilde{J}(z)=X(\gamma_{-L-1}-z)^{p}Y(\gamma_{-L-1}-z)^{q}-X(z)^{p}Y(z)^{q}=Y(z)^{q}(X(\gamma-z)^{p}-X(z)^{p}),\]
so $\tilde{J}(z)=0$ if and only if $X(\gamma-z)^{p}=X(z)^{p}$. Now, recall that in our model the starting point $(p,q)$ must satisfy $p>0$, so similarly to the previous case for $Y(z)$, it follows that $X(\gamma-z)=X(z)$ or $X(\gamma-z)=-X(z)$. In the former case $2\gamma$ is a period of $X$, which is a contradiction since $0<2\gamma<\pi\tau$. and in the latter case we get a similar contradiction to before:
\[x\Pgf(x,Y(z))-\frac{x}{t}=B_{1}(z)(x-X(z))(x-X(\gamma-z))=B_{1}(z)(x-X(z))(x+X(z))=B_{1}(z)(x^2-X(z)^2),\]
where $B_{1}(z)=w_{(1,1)}Y(z)+w_{(0,1)}+w_{(-1,1)}Y(z)^{-1}$. This is a contradiction because $x\Pgf(x,Y(z))-\frac{x}{t}$ has a linear term equal to $w_{(0,1)}Y(z)+w_{(0,-1)}Y(z)^{-1}-\frac{1}{t}$, which cannot be identically 0, whereas the right hand side has no linear term.
\end{proof}

\subsection{Infinite group cases}

In this entire section we assume that $M$ is odd, as the cases for $M$ even were all handled in the previous section.

In this section we consider the case $\frac{\gamma}{\pi\tau}\notin\mathbb{Q}$. the first part of this section is dedicated to showing that $\Qgf_{j}(x,y;t)$ is not D-finite in $x$ in these cases, and we will subsequently analyse the D-algebraicity of $\Qgf_{j}(x,y;t)$. We start with a lemma which essentially proves that $A_{H}(\frac{1}{x};t)$ is not rational in $x$, as this turns out to be a case which needs to be treated separately. Surprising this seems to be the most difficult result of this section, in the sense that it is the only result for which our proof does not apply systematically to walks in an $M$-quadrant cone for any $M$. In particular, for walks in the quadrant we have only been able to prove the result when $t$ is sufficiently small (See Lemma \ref{lem:not-rationalM1}).

\begin{Lemma}\label{lem:not-rationalMbig}
Assume that $\frac{\gamma}{\pi\tau}\notin\mathbb{Q}$ and $M\geq3$ is odd. Then $\pi\tau$ is not a period of $A(z)$ or $B(z)$. 
\end{Lemma}
\begin{proof}
By \eqref{eq:bc_PH_PV_nice}, if either $A(z)$ or $B(z)$ has $\pi\tau$ as a period, then the other does as well. If $L\geq 2$, then by Theorem \ref{thm:bc_PH_PV_characterisation}, $A(z)$ has no poles in $\Omega_{-3}\cup\Omega_{-2}\cup\Omega_{-1}\cup\Omega_{0}$. Together with the fact that $\pi\tau$ is a period of $A(z)$, this implies that $A$ has no poles at all, so it is a constant function, which will quickly lead to a contradiction. Indeed, using Theorem \ref{thm:bc_PH_PV_characterisation}, this implies that $X(z)^{p}Y(z)^{q}=X(\gamma_{K}-z)^{p}Y(\gamma_{K}-z)^{q}$, and this expression has no poles in $\Omega_{j}$ for $0\leq j\leq K+1$. If $K$ is even then this implies $X(z)^{p}Y(z)^{q}=X(\gamma-z)^{p}Y(\gamma-z)^{q}$, so, $X(z)^{p}=X(\gamma-z)^{p}$. This is impossible as when $z$ is on the border between $\Omega_{1}$ and $\Omega_{2}$, $\gamma-z$ is on the border between $\Omega_{-1}$ and $\Omega_{0}$, so $|X(z)|>1>|X(\gamma-z)|$. If $K$ is odd then $X(z)^{p}Y(z)^{q}=X(-\gamma-z)^{p}Y(-\gamma-z)^{q}$, leading to a similar contradiction if $q\neq 0$. For the case where $K$ is odd and $q=0$, Theorem \ref{thm:bc_PH_PV_characterisation} implies that $X(z)^{p}=F+A(z)+B(z)$ has no poles in $\Omega_{j}$ for $0\leq j\leq K+1$, which is impossible as we know that $X(z)$ has poles in $\Omega_{1}\cup\Omega_{2}$. This completes the case $L\geq 2$. Similarly, if $K\geq 2$ then $B(z)$ has no poles in $\Omega_{0}\cup\Omega_{1}\cup\Omega_{2}\cup\Omega_{3}$, so it is constant. Theorem \ref{thm:bc_PH_PV_characterisation} then implies that $X(z)^{p}Y(z)^{q}=X(\gamma_{-L-1}-z)^{p}Y(\gamma_{-L-1}-z)^{q}$. When $L$ is odd this implies $X(z)^{p}=X(\gamma-z)^{p}$, which leads to a contradiction as in the previous case. When $q\neq 0$ and $L$ is even we get a similar contradiction. In the case $q=0$, and $L$ even, we must have $L>0$, for $(p,0)$ to be in the cone, but then $L\geq 2$, which we already showed to be impossible. 

Finally we are left with the cases where $K\leq 1$ and $L\leq 1$. Since $K+L+1=M\geq 3$, the only remaining case is $K=L=1$ and $M=3$. This is exactly the three-quadrant cone case covered by Lemma \ref{lem:not-rational}, so this completes the proof.\end{proof}

In the following lemma we prove a slightly weaker result for $M=1$. We suspect that the more general result above applies for $M=1$, but we have not been able to prove it. 

\begin{Lemma}\label{lem:not-rationalM1}
Assume that the group is not finite and that $M=1$ (the quadrant case). For all sufficiently small $t$ satisfying $\frac{\gamma}{\pi\tau}\notin\mathbb{Q}$, the function $A(z)$ does not have $\pi\tau$ as a period. 
\end{Lemma}
\begin{proof}
Assume for the sake of contradiction that $A(z)$ has $\pi\tau$ as a period. Then by \eqref{eq:bc_PH_PV_nice}, $B(z)$ also has $\pi\tau$ as a period. Moreover, since we are in the case $M=1$, there is only one quadrant which does not include the axes, so the starting point $(p,q)$ of the walk cannot be on either axis. Hence $p,q\geq1$. Moreover, $K=L=0$, so from Theorem \ref{thm:bc_PH_PV_characterisation}, $A(z)=A(-\gamma-z)$ and $B(z)=B(\gamma-z)$.

\textbf{Claim 1:} Let $\mathcal{D}_{t}=\left\{x\in\mathbb{C}:|x|<\sqrt{\frac{1}{t\Pgf(1,1)}}\right\}$. If $X(z_{0})\in\mathcal{D}_{t}$ then $A(z_{0})=\Agf(X(z_{0});t)$. In particular $A(z)$ does not have a pole at $z=z_{0}$.\newline
\textbf{Proof of Claim 1:} For $z_{0}\in\Omega_{-1}\cup\Omega_{0}$, we know that this result holds as it is the definition of $A(z_{0})$ in this region. Moreover, this extends analytically to the connected component of $X^{-1}(\mathcal{D}_{t})$ containing $\Omega_{0}\cup\Omega_{-1}$ as we know from Lemma \ref{lem:Q_series_convergence} that $\Agf(X(z_{0});t)$ converges for $X(z_{0})\in\mathcal{D}_{t}$. If $z_{0}$ is in the connected component of $X^{-1}(\mathcal{D}_{t})$ containing $\Omega_{4m}\cup\Omega_{4m-1}$, then we have
\[A(z_{0})=A(z_{0}-m\pi\tau)=\Agf(X(z_{0}-m\pi\tau);t)=\Agf(X(z_{0});t),\] since $z_{0}-m\pi\tau$ is in the connected component of $X^{-1}(\mathcal{D}_{t})$ containing $\Omega_{0}\cup\Omega_{-1}$. Hence to complete the proof of claim 1 it suffices to prove that every connected component $\Gamma$ of $X^{-1}(\mathcal{D}_{t})$ contains one of the sets $\Omega_{4m}\cup\Omega_{4m-1}$. Assume $\Gamma$ is such a connected component. Then, since $\mathcal{D}_{t}$ is an open set, $\Gamma$ must also be an open set, so by the open mapping theorem $X(\Gamma)$ is an open set, so the boundary of $X(\Gamma)$ is $X(\partial \Gamma)$, where $\partial \Gamma$ is the boundary of $\Gamma$. Moreover, since $\Gamma$ is a connected component of $X^{-1}(\mathcal{D}_{t})$, we have $X(\partial \Gamma)\subset \partial \mathcal{D}_{t}$. So $X(\Gamma)$ is a non-empty, open subset of the connected set $\mathcal{D}_{t}$ and its boundary is a subset of the boundary of $\mathcal{D}_{t}$, which is only possible if $X(\Gamma)=\mathcal{D}_{t}$. In particular, there is some $z_{0}\in\Gamma$ satisfying $X(z_{0})=0$, so $z_{0}\in\Omega_{4m}\cup\Omega_{4m-1}$ for some $m$, which implies that $\Gamma$ contains $\Omega_{4m}\cup\Omega_{4m-1}$. This completes the proof of Claim 1.

We will reach a contradiction using this claim along with the equation 
\begin{equation}X(z)^{p}Y(z)^{q}=F+A(z)+B(z),\label{eq:simpleABXY}\end{equation}
and the related equation
\begin{equation}X(z)^{p}(Y(z)^{q}-Y(-\gamma-z)^{q})=B(z)-B(-\gamma-z),\label{eq:simple_diff_BXY}\end{equation}
by comparing the possible positions of poles of $X$, $Y$, $A$ and $B$ for varying $t$. To analyse the relationship between $X(z)$, $Y(z)$ and $t$, we need to use the definition of $\Pgf(x,y)$ and define
\[A_{\beta}(x):=\sum_{\alpha=-1}^{1}w_{(\alpha,\beta)}x^{\alpha},\]
for $-1\leq\beta\leq 1$, so that
\begin{equation}\label{eq:P_splitting}\Pgf(x,y)=y^{-1}A_{-1}(x)+A_{0}(x)+yA_{1}(x).\end{equation}
Now, for general $z\in\mathbb{C}$, we have $\Pgf(X(z),Y(z))-\frac{1}{t}=0$ so the polynomial
$y\Pgf(X(z),y)-\frac{y}{t}$ has roots at $y=Y(z)$ and $y=Y(-\gamma-z)$. Hence it's coefficients are related by
\begin{align}
A_{0}(X(z))-\frac{1}{t}&=-A_{1}(X(z))(Y(z)+Y(-\gamma-z))\label{eq:Ycoef_sum},\\
A_{-1}(X(z))&=A_{1}(X(z))Y(z)Y(-\gamma-z)\label{eq:Ycoef_prod}.
\end{align}
The second consequence of \eqref{eq:P_splitting} is the following claim:\newline
\textbf{Claim 2:} There are values $x_{0},x_{1},y_{0},y_{1}\in\mathbb{C}\cup\{\infty\}$ which do not depend on $t$ such that, for any $t$, if $\epsilon$ is a pole of $Y(z)$, then $\{X(\epsilon),X(-\gamma-\epsilon)\}=\{x_{0},x_{1}\}$ and if $\delta$ is a pole of $X(z)$, then $\{Y(\delta),Y(\gamma-\delta)\}=\{y_{0},y_{1}\}$.\newline
\textbf{Proof of Claim 2:} We will prove the claim for $x_{0}$ and $x_{1}$ as the other part of this claim is equivalent. In general, $X(z)$ and $X(-\gamma-z)$ are the two roots for $x$ of the quadratic polynomial  $\frac{x}{Y(z)}\Pgf(x,Y(z))-\frac{x}{tY(z)}$. As $z\to \epsilon$, this converges to $xA_{1}(x)$, so we can simply define $\{x_{0},x_{1}\}$ to be the roots of $xA_{1}(x)$, which do not depend on $t$. In the case the $w_{(1,1)}=0$ but $w_{(0,1)}\neq 0$, the polynomial $xA_{1}(x)$ is linear, so it only has one root, however one of the roots of $\frac{x}{Y(z)}\Pgf(x,Y(z))-\frac{x}{tY(z)}$ converges to $\infty$ as $z\to\delta$, so we can let $x_{0}=\infty$ and $x_{1}$ be the unique root of $xA_{1}(x)$. Finally if $w_{(1,1)}=w_{(0,1)}=0$, we can set $x_{0}=x_{1}=\infty$. This completes the proof of Claim 2.

Now, assume that $t$ is sufficiently small that the values $x_{0},x_{1},y_{0},y_{1}$ that are not $\infty$ lie in $\mathcal{D}_{t}$.\newline
\textbf{Claim 3:} If $\delta$ is a pole of $X(z)$ but not of $Y(z)$, then it is not a pole of $B(z)$. Similarly, if $\epsilon$ is a pole of $Y(z)$ but not of $X(z)$, then it is not a pole of $A(z)$.\newline
\textbf{Proof of Claim 3:} Assume $\epsilon$ is a pole of $Y(z)$ but not of $X(z)$. Then by our assumption, $X(\epsilon)=x_{0}\in\mathcal{D}_{t}$ or $X(\epsilon)=x_{1}\in\mathcal{D}_{t}$. Hence by Claim 1, $A(z)$ does not have a pole at $\epsilon$. The other part of the claim is equivalent, so this completes the proof of the claim.

We will now consider a number of cases separately regarding the nature of $X$ and $Y$ around a pole $\delta$ of $X(z)$.\newline
\textbf{Case 1:} $w_{(1,1)}=w_{(1,0)}=0$, so $\delta$ is a double pole of $X(z)$ and a simple pole of $Y(z)$.\newline
In this case the second pole $-\gamma-\delta$ of $X(z)$ must in fact be equal to (a possibly shifted version of) $\delta$, that is, $-\gamma-\delta\in\delta+\pi\mathbb{Z}+\pi\tau\mathbb{Z}$. Hence, $X(z)=X(-\gamma-z)=X(2\delta-z)$ and $Y(z)=Y(\gamma-z)=Y(-2\delta-z)$. Now, since $Y(z)$ has a pole at $\delta$, it must have a pole at $\gamma-\delta$, but $X(z)$ cannot have a pole at this point, so by Claim 3, $A(z)$ does not have a pole at $z=\gamma-\delta$. Moreover, since $A(z)=A(-\gamma-z)$, this means that $A(z)$ does not have a pole at $z=-2\gamma+\delta$ either. Now taking successive differences of \eqref{eq:simpleABXY} yields
\[X(-2\delta-z)^{p}Y(-2\delta-z)^{q}-X(z)^{p}Y(z)^{q}+X(2\delta-z)^{p}Y(2\delta-z)^{q}-X(-4\delta+z)^{p}Y(-4\delta+z)^{q}=A(\gamma-z)-A(2\gamma+z),\]
and as we have shown the right hand side of this equation does not a have a pole at $z=\delta$. Now we will reach a contradiction by analysing the left hand side near $z=\delta$. Since $Y(z)$ has a simple pole at this point, while $X(z)$ has a double pole and $X(\delta-z)=X(\delta+z)$, we may write $X(z)=c_{1}(z-\delta)^{-2}+O(1)$ and $Y(z)=c_{2}(z-\delta)^{-1}+c_{3}+O(z-\delta)$, where $c_{1},c_{2}\neq0$. Then the left hand side of the equation above behaves as
\[-c_{1}^{p}c_{2}^{q}(1-(-1)^q)(z-\delta)^{-2p-q}-qc_{1}^{p}c_{2}^{q-1}c_{3}(1+(-1)^q)(x-\delta)^{-2p-q+1}+O\left((x-\delta)^{-2p-q+2}\right).\]
Since we know that the expression must not have a pole at $\delta$, the two leading terms must be $0$, which is only possible if $q$ is even and $c_{3}=0$. But then $Y(z)+Y(-\gamma-z)\to 2c_{3}= 0$ as $z\to\delta$. Comparing to \eqref{eq:Ycoef_sum}, this implies that
\[\frac{\frac{1}{t}-A_{0}(x)}{A_{1}(x)}\to 0\qquad\text{as }x\to\infty.\]
But since $w_{(1,1)}=w_{(1,0)}=0$, we must have
\[\frac{\frac{1}{t}-A_{0}(x)}{A_{1}(x)}\to \frac{1}{t w_{(0,1)}}\neq 0\qquad\text{as }x\to\infty,\]
a contradiction.\newline
\textbf{Case 2:} $w_{(1,1)}=w_{(0,1)}=0$.\newline
This is symmetric to the previous case so the same proof applies.

For the remaining cases we assume that $p\geq q$, as the $q\geq p$ case is equivalent.\newline
\textbf{Case 3:} $w_{(1,1)}=0\neq w_{(0,1)},w_{(1,0)}$, so $X(z)$ and $Y(z)$ share a simple pole $\delta\in\Omega_{1}\cup\Omega_{2}$.\newline
In this case $\gamma-\delta$ is a pole of $Y(z)$ but not $X(z)$ while $-\gamma-\delta$ is a pole of $X(z)$ but not $Y(z)$. Using \eqref{eq:simpleABXY}, we have
\[X(z)^{p}Y(z)^{q}-X(\gamma-z)^{p}Y(\gamma-z)^{q}-X(-\gamma-z)^{p}Y(-\gamma-z)^{q}=-F-A(\gamma-z)-B(-\gamma-z).\]
Analysing this around $z=\delta$ yields a contradiction: The right hand side has no pole at $z=\delta$, while the terms $X(z)^{p}Y(z)^{q}$, $X(\gamma-z)^{p}Y(\gamma-z)^{q}$ and $X(-\gamma-z)^{p}Y(-\gamma-z)^{q}$ have poles of orders $p+q$, $q$ and $p$, respectively, so the left hand side has a pole of order $p+q$.\newline
\textbf{Case 4:} $w_{(1,1)}\neq 0$, so $X(z)$ and $Y(z)$ do not share any poles.\newline
Both $\delta$ and $-\gamma-\delta$ are poles of $X(z)$, so they are not poles of $Y(z)$. Hence by  Claim 3, these must also not be poles of $B(z)$. In particular, this implies that the right hand side of \eqref{eq:simple_diff_BXY} does not have a pole at $z=\delta$, so the left hand side 
\[X(z)^{p}(Y(z)^{q}-Y(-\gamma-z)^{q})\]
is also analytic at $z=\delta$. Hence, $(Y(z)^{q}-Y(-\gamma-z)^{q})$ must have a root of order at least $p$ at $z=\delta$. Moreover, in the case that $X(z)$ has a double pole at $z=\delta$, the function $(Y(z)^{q}-Y(-\gamma-z)^{q})$ must have a root of order at least $2p$ at $z=\delta$. 
We will now consider different values of $q$ in separate cases.\newline
\textbf{Case 4a:} $q=1$\newline
In this case $Y(z)-Y(-\gamma-z)$ has a root of order at least $p\geq1$ at $z=\delta$, so, in particular $Y(\delta)=Y(-\gamma-\delta)$, which implies that $-\gamma-\delta\in \{\delta,\gamma-\delta\}+\pi\mathbb{Z}+\pi\tau\mathbb{Z}$. But since $\frac{\gamma}{\pi\tau}\in\mathbb{R}\setminus\mathbb{Q}$, we do not have $-2\gamma\in \pi\mathbb{Z}+\pi\tau\mathbb{Z}$, so we are left with the case $-\gamma-\delta\in \delta+\pi\mathbb{Z}+\pi\tau\mathbb{Z}$. this implies that $X(\delta+z)=X(-\gamma-\delta-z)=X(\delta-z)$, so $X(z)$ must have a double pole at $z=\delta$. The function $Y(z)-Y(-\gamma-z)$ has at most 4 poles in each fundamental domain (counting with multiplicity), so it must have at most 4 roots (counting with multiplicity). Since $Y(z)$ has both $\pi$ and $\pi\tau$ as periods, it is clear that $\frac{-\gamma}{2}$, $\frac{-\gamma+\pi}{2}$, $\frac{-\gamma+\pi\tau}{2}$, $\frac{-\gamma+\pi+\pi\tau}{2}$ are all distinct roots of $Y(z)-Y(-\gamma-z)$, so they must be the only roots and they must all be simple roots. In particular, this implies that $\delta$ cannot be a double root of $Y(z)-Y(-\gamma-z)$, so it must be a pole of $X(z)(Y(z)-Y(-\gamma-z))$, and hence $X(z)^{p}(Y(z)^{q}-Y(-\gamma-z)^{q})$, a contradiction.\newline
\textbf{Case 4b:} $q\geq 2$ and $Y(\delta)=0$.\newline
By our assumption that $p\geq q$, we also have $p\geq2$. We will use the product
\[Y(z)^{q}-Y(-\gamma-z)^{q}=\prod_{m=0}^{q-1}(Y(z)-e^{\frac{2m i\pi}{q}}Y(-\gamma-z)).\]
Since this has a root at $z=\delta$ we must have $Y(-\gamma-\delta)=0$. As in the previous case $-\gamma-\delta\in \delta+\pi\mathbb{Z}+\pi\tau\mathbb{Z}$ and so $X$ has a double pole at $\delta$. If $Y(z)$ has a double root at $\delta$ then it must be the only root of $Y(z)$, so $\gamma-\delta\in\delta+\pi\tau\mathbb{Z}+\pi\mathbb{Z}$, which is impossible since $-\gamma-\delta\in \delta+\pi\mathbb{Z}+\pi\tau\mathbb{Z}$ but $\frac{\gamma}{\pi\tau}\notin\mathbb{Q}$. So $Y(z)$ does not have a double pole at $\delta$. 
Now, each function $Y(z)-e^{\frac{2m i\pi}{q}}Y(-\gamma-z)$ has a pole at $z=\delta$, since $Y(\delta)=Y(-\gamma-\delta)=0$. However, the derivative of this function at $z=\delta$ is $Y'(\delta)+e^{\frac{2m i\pi}{q}}Y'(-\gamma-\delta)=Y'(\delta)\left(1+e^{\frac{2m i\pi}{q}}\right)$, so the function has a double pole if and only if $e^{\frac{2m i\pi}{q}}=-1$ i.e., $m=\frac{q}{2}$. Moreover, from \eqref{eq:Ycoef_sum}, the function
\[Y(z)+Y(-\gamma-z)=\frac{\frac{1}{t}-A_{0}(X(z))}{A_{1}(X(z))}\sim \frac{1}{t w_{(1,1)}X(z)},\]
which has a double root but not a triple root at $z=\delta$. Note that the expression above uses the fact that $w_{(1,0)}=0$ - if this was not the case then $Y(z)+Y(-\gamma-z)$ would not have a root at $z=\delta$. 
So finally, the function $Y(z)-e^{\frac{2m i\pi}{q}}Y(-\gamma-z)$ has a simple root at $z=\delta$ if $m\neq\frac{q}{2}$, while it has a double root if $m=\frac{q}{2}$. This implies that in the case that $q$ is odd, $Y(z)^{q}-Y(-\gamma-z)^{q}$ has a root of order exactly $q$ at $z=\delta$, while in the case that $q$ is even $Y(z)^{q}-Y(-\gamma-z)^{q}$ has a root of order exactly $q+1$ at this point. In either case, $X(z)^{p}(Y(z)^{q}-Y(-\gamma-z)^{q})$ has a pole of order at least $2p-q-1\geq p-1>0$, a contradiction.\newline
\textbf{Case 4c:} $q\geq 2$ and $Y(\delta)\neq 0$.\newline
In this case at most one function $Y(z)-e^{\frac{2m i\pi}{q}}Y(-\gamma-z)$ in the product
\[\prod_{m=0}^{q-1}(Y(z)-e^{\frac{2m i\pi}{q}}Y(-\gamma-z))\]
can have a root at $z=\delta$. Since we know that the product has a root at $\delta$, there must be some such value of $m$, so we define $u:=e^{\frac{2m i\pi}{q}}$ for this $m$. Then $X(z)^{p}(Y(z)-u Y(-\gamma-z))$ has no pole at $z=\delta$. For $m=0$ (i.e., $u=1$), this is equivalent to the $q=0$ case, which we showed was impossible, so we are left with the case $m\neq 0$. 
In the case that $u=-1$, using \eqref{eq:Ycoef_sum}, we have
\[Y(z)-u Y(-\gamma-z)=-\frac{A_{0}(X(z))-\frac{1}{t}}{A_{1}(X(z))}.\]
For sufficiently small $t$, the numerator lies in $\Omega(1)$ while the denominator is in $O((z-\delta)^{-1})$, so $Y(z)-u Y(-\gamma-z)=\Omega(z-\delta)$ around $z=\delta$. This implies that $X(z)^{p}(Y(z)-u Y(-\gamma-z))$ has a pole of order at least $p-1>0$, a contradiction.

If $u\neq -1$, then we must have $q\geq 3$, as for $q=2$ the only possibilities are $u=1$ and $u=-1$. The function $H(z)$ defined by
\[H(z)=(Y(z)-u Y(-\gamma-z))(u Y(z)-Y(-\gamma-z))\]
can be written as
\[H(z)=\frac{u(A_{0}(X(z))-\frac{1}{t})^{2}-(u+1)^{2}A_{-1}(X(z))A_{1}(X(z))}{A_{1}(X(z))^{2}}.\]
The numerator of this equation lies in $\Omega(1)$ (for sufficiently small $t$), while the denominator is in $O((z-\delta)^{-2})$, So $H(z)=\Omega((z-\delta)^{2})$ and so $Y(z)-u Y(-\gamma-z)=\Omega((z-\delta)^{2})$. This implies that $X(z)^{p}(Y(z)-u Y(-\gamma-z))$ has a pole of order at least $p-2\geq q-2>0$, a contradiction.
\end{proof}


\begin{Theorem}\label{thm:bc_inf_group_non-D-finite}
Assume that $\frac{\hat{\tau}}{\tau}\notin\mathbb{Q}$ and $B(z)$ does not have $\pi\tau$ as a period. Then $\Qgf_{j}(x,y;t)$ is not D-finite in $x$ or $y$.
\end{Theorem}  
\begin{proof}
Suppose the contrary, that $\Qgf_{j}(x,y;t)$ is D-finite in $x$ or $y$. Then by Lemma \ref{lem:bc_QjAB_complexities_equivalence}, $B(z)$ is $X$-D-finite in $z$ (see Definition \ref{def:X-D-finite}). Then by Lemma \ref{lem:D-fini_z}, the poles $z_{c}$ of $B(z)$ fall into only finitely many classes $z_{c}+\pi\mathbb{Z}+\pi\tau\mathbb{Z}$. Hence $B(\pi\tau+z)-B(z)$ has the same property.

Now, from \eqref{eq:bc_PH_qdiff}, we have
\[B(\pi\hat{\tau}+z)-B(z)=\tilde{J}(z)=\tilde{J}(z+\pi\tau)=B(\pi\tau+\pi\hat{\tau}+z)-B(\pi\tau+z),\]
and rearranging yields
\[B(\pi\tau+\pi\hat{\tau}+z)-B(\pi\hat{\tau}+z)=B(\pi\tau+z)-B(z).\]
This implies that $B(\pi\tau+z)-B(z)$ is an elliptic function with periods $\pi$ and $\pi\hat{\tau}$. If this function has a pole $z_0$, then for every $k\in\mathbb{Z}$, the value $\tilde{z}_{k}=z_{0}+k\pi\hat{\tau}$ is a pole. This is a contradiction as these points all define different classes $z_{k}+\pi\tau\mathbb{Z}+\pi\mathbb{Z}$, since $\frac{\hat{\tau}}{\tau}\in\mathbb{R}\setminus\mathbb{Q}$. The only remaining case to consider is when $B(\pi\tau+z)-B(z)$ has no poles, in which case it must be constant:
\[B(\pi\tau+z)-B(z)=c.\]
In fact combining this with \eqref{eq:bc_PH_flippy_full}, we see that $c=0$, as
\[c=B(\pi\tau+z)-B(z)=B(-\gamma-z)-B(\pi\tau-\gamma-z)=-c.\]
 But this contradicts the assumption that $\pi\tau$ is not a period of $B(z)$.
\end{proof}

\begin{Corollary}\label{cor:bc_inf_group_non-D-finite_Mgeq3} If $M\geq 3$ is odd and $\frac{\gamma}{\pi\tau}\notin\mathbb{Q}$, then $\Qgf_{j}(x,y;t)$ is not D-finite in $x$.
\end{Corollary}
\begin{proof}
By Lemma \ref{lem:not-rationalMbig}, $\pi\tau$ cannot be a period of $B(z)$ for $M\geq3$. Moreover, if $\frac{\gamma}{\pi\tau}\notin\mathbb{Q}$, then by Lemma \ref{lem:bc_odd_D-finite}, $\frac{\hat{\tau}}{\tau}\notin{\mathbb{Q}}$ so by Theorem \ref{thm:bc_inf_group_non-D-finite}, $\Qgf_{j}(x,y;t)$ is not D-finite in $x$.
\end{proof}

\begin{Corollary}\label{cor:bc_inf_group_non-D-finite_M1} If $M=1$, and the group of the walk is not finite, then there are values of $t$ for which $\Qgf_{j}(x,y;t)$ is not D-finite in $x$.
\end{Corollary}
\begin{proof}
Since the group of the walk is not finite, there are arbitrarily small values of $t$ satisfying $\frac{\gamma}{\pi\tau}\notin\mathbb{Q}$. Hence, for sufficiently small such $t$, Lemma \ref{lem:not-rationalM1} implies that $\pi\tau$ is not a period of $B(z)$. Therefore, by Theorem \ref{thm:bc_inf_group_non-D-finite}, these are values of $t$ for which $\Qgf_{j}(x,y;t)$ is not D-finite in $x$.
\end{proof}



\subsubsection{Decoupling cases}
Recall from Definition \ref{defn:decoupling} that we say that $X(z)^{p}Y(z)^{q}$ is decoupling if there is a pair of rational functions $R_{1}$ and $R_{2}$ satisfying
\[X(z)^{p}Y(z)^{q}=R_{1}(X(z))+R_{2}(Y(z)).\]
As we will show in the following theorem, this implies that $\Qgf_{j}(x,y;t)$ is D-algebraic in $x$ and $y$.
\begin{Theorem}\label{thm:bc_D-alg_proof}
Assume that
\[X(z)^{p}Y(z)^{q}=R_{1}(X(z))+R_{2}(Y(z))\]
holds for some rational functions $R_{1}$ and $R_{2}$. Then $\Qgf_{j}(x,y;t)$ is D-algebraic in $x$ and $y$. 
\end{Theorem}
\begin{proof}Since $M=L+K+1$ is odd, $L$ and $K$ have the same parity. If they are both even then $X_{-L}(z)\in\{X(z),\frac{1}{X(z)}\}$ and $Y_{K}(z)\in\{Y(z),\frac{1}{Y(z)}\}$, while if they are both odd then $X_{-L}(z)\in\{Y(z),\frac{1}{Y(z)}\}$ and $Y_{K}(z)\in\{Y(z),\frac{1}{Y(z)}\}$, (see Subsection \ref{subsec:bc_analytic_functional_equations}) so either way we can write
\[R_{1}(X(z))+R_{2}(Y(z))=\hat{R}_{1}(X_{-L}(z))+\hat{R}_{2}(Y_{K}(z)),\]
where $\hat{R}_{1}$ and $\hat{R}_{2}$ are rational functions.
Under the assumption, \eqref{eq:bc_PH_PV_nice} can be written as
\begin{equation}\label{eq:bc_Tequation_in_decoupling_case}T(z):=\hat{R}_{2}(Y_{K}(z))-B(z)=A(z)+F-\hat{R}_{1}(X_{-L}(z)),\end{equation}
which implies that $T(z)$ satisfies $T(z)=T(\gamma_{K}-z)=T(\gamma_{-L-1}-z)=T(z+\pi)$. Combining these shows that $T(z)$ is an elliptic function with periods $\pi$ and $\pi\hat{\tau}=\gamma_{K}-\gamma_{-L-1}$, so it must be D-algebraic (since $T'(z)$ and $T(z)$ must be related by a non-trivial polynomial equation). Now, since $X(z)$ is also D-algebraic, it follows from \eqref{eq:bc_Tequation_in_decoupling_case} that $B(z)$ is also D-algebraic in $z$, which is the same as being $X$-D-algebraic (see Definition \ref{def:X-D-alg}). It then follows from Lemma \ref{lem:bc_QjAB_complexities_equivalence} that $\Qgf_{j}(x,y;t)$ is D-algebraic in $x$ and $y$. 
\end{proof}


\subsubsection{Non-decoupling cases}
In this section we show that if there is no decoupling function, then the generating function is not D-algebraic in $x$, as in Subsection \ref{subsec:inf_group_non-dec} for the 3-quadrant cone. The proof works along the same lines as \cite{dreyfus2018nature,hardouin2020differentially} for the quarter plane case, which relies on Galois theory of q-difference equations. Rather than essentially rewriting these entire proofs, in Appendix \ref{ap:D-trans} we use results from \cite{dreyfus2018nature} to deduce Corollary \ref{cor:D-trans}, which avoids Galois theory language in its statement, and can be readily applied to show the main result of this section.
\begin{Theorem}\label{thm:bc_non-D-alg_proof}
Fix $t\in\left(0,\frac{1}{P(1,1)}\right)$. Assume that there are no rational functions $R_{1},R_{2}\in\mathbb{C}(x)$ satisfying 
\[X(z)^{p}Y(z)^{q}=R_{1}(X(z))+R_{2}(Y(z)).\]
Then the function $\Qgf_{j}(x,y;t)$ is not D-algebraic in $x$ or $y$.
\end{Theorem}
\begin{proof}
It suffices to prove that $B(z)$ is not D-algebraic, as we showed in Lemma \ref{lem:bc_QjAB_complexities_equivalence} that this is equivalent to $\Qgf_{j}(x,y;t)$ being D-algebraic in $x$ or $y$. Now assume for the sake of contradiction that $B(z)$ is D-algebraic in $z$. We will show that this implies that there are rational functions $R_{1}$ and $R_{2}$ satisfying the equation in the theorem. 

By \eqref{eq:bc_PH_PV_nice}, the functions $h(z):=X(z)^p Y(z)^q$, $f_{1}(z):=A(z)+F$ and $f_{2}(z):=B(z)$ satisfy the conditions of Corollary \ref{cor:D-trans}, with $\gamma_{1}=\gamma_{-L-1}$ and $\gamma_{2}=\gamma_{K}$. Hence, there are meromorphic functions $a_{1},a_{2}:\mathbb{C}\to\mathbb{C}\cup\{\infty\}$ satisfying 
\begin{align*}
X(z)^p Y(z)^q&=a_{1}(z)+a_{2}(z),\\
a_{1}(z)&=a_{1}(z+\pi)=a_{1}(z+\pi\tau)=a_{1}(\gamma_{-L-1}-z),\\
a_{2}(z)&=a_{2}(z+\pi)=a_{2}(z+\pi\tau)=a_{2}(\gamma_{K}-z).
\end{align*}
Recall that $M=K+L+1$ is odd, so $K$ and $L$ have the same parity. If $K$ and $L$ are both even then $\gamma_{K}=\gamma+\frac{K}{2}\pi\tau$ and $\gamma_{-L-1}=-\gamma-\frac{L}{2}\pi\tau$, so $a_{1}(z)=a_{1}(-\gamma-z)$ and $a_{2}(z)=a_{2}(\gamma-z)$. Hence by Proposition \ref{prop:rationalofXorY}, $a_{1}(z)$ is a rational function of $X(z)$, while $a_{2}(z)$ is a rational function of $Y(z)$. Hence we can write
\[X(z)^p Y(z)^q=R_{1}(X(z))+R_{2}(Y(z)),\]
as required. Similarly, in the case that $K$ and $L$ are both odd $a_{1}(z)=a_{1}(\gamma-z)$ and $a_{2}(z)=a_{2}(-\gamma-z)$, so by Proposition \ref{prop:rationalofXorY}, $a_{1}(z)$ is a rational function of $Y(z)$, while $a_{2}(z)$ is a rational function of $X(z)$. Hence in this case we can still write
\[X(z)^p Y(z)^q=R_{1}(X(z))+R_{2}(Y(z)).\]
\end{proof}

\section{Conclusion and further questions}\label{sec:conlusion_and_questions}

In this article we deduced a complex analytic functional equation characterising the generating function $\Cgf(x,y;t)$ counting three quarter plane walks, and subsequently we generalised this to count walks in the $M$-quadrant cone. We deduced several results using this functional equation, although there are also several questions that we have left open, which we pose below.

The most conspicuous unanswered question relates to the nature of the generating function with respect to the third variable $t$:
\begin{Question} Is the nature of the generating function $\Cgf(x,y;t)$ always the same with respect to $t$ as with respect to $x$ and $y$?\end{Question}
Indeed from a combinatorial perspective series in $t$ are the most natural as for example, the series $\Cgf(1,1;t)$ is the generating function for all walks in the three quarter plane, whereas, while fixing $t$ is convenient for the complex analytic methods used in this article, it has little interest from a combinatorial perspective. We believe that at least certain aspects of this question can be answered using our Theorem \ref{thm:PH_PV_characterisation}, for example proving that $\Cgf(x,y;t)$ is D-algebraic in $t$ when there is a decoupling function could likely be proven along the same lines as for $\Qgf(x,y;t)$ \cite{bernardi2017counting}. The converse may also be provable in the same way as it was proven in the quarter plane \cite{dreyfus2019length,hardouin2020differentially,dreyfus2021differential}. Moreover, it may be possible to prove that certain cases are algebraic or D-finite using modular properties of the solutions as functions of $\tau$, as in \cite{elveyprice_winding_FPSAC} and \cite{elvey2020six}.

Our next two questions relate to the nature of the generating functions for fixed $t$:
\begin{Question} Are there any models along with fixed values of $t$ for which $\frac{\gamma}{\pi\tau}\notin\mathbb{Q}$ but $\Qgf(x,y;t)$ is a rational function of $x$?\end{Question}
This is the only missing ingredient from a complete characterisation of the nature of the generating functions $\Qgf_{j}(x,y;t)$ with respect to $x$, as for all $M$-quadrant cones with $M\geq 3$ odd, we showed that no such model exists (see Theorem \ref{thm:bc_D-finite_fixed_t}).

Another unanswered question is the following:
\begin{Question} Are there any models for which the presence of a decoupling function depends on the value of $t$? 
\end{Question}
Equivalently we could ask whether it is possible that $\Cgf(x,y;t)$ is D-algebraic with respect to $x$ for some values of $t$ but not others. 
The analogous question of whether the group is finite is clearly {\em yes}, as it only depends on whether the parameter $\frac{\gamma}{\pi\tau}\in\mathbb{Q}$. It seems that in most cases $\frac{\gamma}{\pi\tau}$ varies as a function of $t$ (perhaps in all non-D-finite cases), and so this parameter will typically be rational for $t$ in a dense subset of its range $\left(0,\frac{1}{P(1,1)}\right)$. Despite this we do not know of an example where the existence of a decoupling function depends on $t$.

Another natural question is whether the results in this article also hold on subtly different spaces, which could also be called $M$-quadrant cones. We make this precise below:
\begin{Definition}Partition $\mathbb{Z}^{2}$ into four quadrants
\begin{align*}
\cQ_{0}&=\{(m,n)\in\mathbb{Z}^{2}:m,n\geq0\},\\
\cQ_{1}&=\{(m,n)\in\mathbb{Z}^{2}:n\geq0>m\},\\
\cQ_{2}&=\{(m,n)\in\mathbb{Z}^{2}:0>m,n\},\\
\cQ_{3}&=\{(m,n)\in\mathbb{Z}^{2}:m\geq0>n\},\\
\end{align*}
and define the {\em alternative $M$-quadrant cone} to be the gluing of $M$ copies of these quadrants in a spiral. Say a walk in an alternative $M$-quadrant cone can only step between adjacent quadrants. More precisely, we extend the definitions of $\cQ_{k}$ to any $k\in\mathbb{Z}$ by $\cQ_{k}=\cQ_{k-4}$, then the {\em alternative $M$-quadrant cone} is the set of points $(k,(m,n))$ where $(m,n)\in\cQ_{k}$, and steps are only allowed from $(k_{1},(m_{1},n_{1}))$ to $(k_{2},(m_{2},n_{2}))$ if $k_{1}-1\leq k_{2}\leq k_{1}+1$.
\end{Definition}
\begin{Question}Is the nature of the series the same for walks in the alternative $k$-quadrant cone as in the $k$-quadrant cone?\end{Question}
For the case $k=3$ this is equivalent to the question discussed in Subsection \ref{subsec:annoying_step_forbidding}:
\begin{Question} For walks in the three-quadrant cone, does the nature of the walk change if the steps between $(0,1)$ and $(1,0)$ are forbidden?\end{Question}

Finally, perhaps the most mysterious result of this article is Corollary \ref{cor:combi}, relating to walks starting on an axis. The statement of this Corollary is purely combinatorial, yet it's proof in this article is far from combinatorial. this leads to the natural question of finding a bijection:
\begin{Question}Find a bijective proof of Corollary \ref{cor:combi}. If this can be done, is there a directly combinatorial way to prove Theorem \ref{thm:p0}?\end{Question}
 
 \noindent {\bf Acknowledgements.} We are grateful to Kilian Raschel for suggesting that we consider the special case where walks start on an axis, as this led us to some interesting results. We also thank Mireille Bousquet-M\'elou for many helpful suggestions on an early version of this article.
 
 \appendix

 \section{Parameterisation of the Kernel curve}
 \label{ap:param_thm}
 In this appendix, we describe a parameterisation of the Kernel curve, which we transform to prove Lemmas \ref{lem:param} and \ref{lem:Omega}. the parameterisation was first given by Raschel in \cite{raschel2012counting} but we follow \cite[Section 2]{dreyfus2019differential} as they prove more results relevant to us. In this section they work under the same assumption on the step set as us - that the step-set is non-singular. They also assume that the sum of the weights $w_{i,j}$ is $1$ and that $t<1$, but by linearly rescaling $t$ and the weights $w_{i,j}$ (and hence $K(x,y)$), this becomes equivalent to our assumption
 \[t<P(1,1)=\sum_{(i,j)\in S}w_{i,j},\]
 without changing the set
 \[\overline{E_{t}}=\{(x,y)\in\mathbb{P}^{1}(\mathbb{C})^{2}:K(x,y)=0\}.\]
We start with the parameterisation of $\overline{E_{t}}$ given by Dreyfus and Raschel \cite[Proposition 2.1]{dreyfus2019differential} for $x$ and (2.16) of the same article for $y$. These parameterisations involve the Weierstrass $\wp$ function
\[\wp(\omega,\omega_{1},\omega_{2}):=\frac{1}{\omega^{2}}+\sum_{(m_{1},m_{2})\in\mathbb{Z}^{2}\setminus\{(0,0)\}}\frac{1}{(\omega+m_{1}\omega_{1}+m_{2}\omega_{2})^2}-\frac{1}{(m_{1}\omega_{1}+m_{2}\omega_{2})^2}.\]
To define the parameterisation, they start by defining rational functions $A_{-1}(x)$, $A_{0}(x)$ and $A_{1}(x)$ by
\[\Pgf(x,y)=:\frac{1}{y}A_{-1}(x)+A_{0}(x)+yA_{1}(x)\]
and the degree 4 (or 3) polynomial $D(x)=x^2\left(\left(A_{0}(x)-\frac{1}{t}\right)^2-4A_{-1}(x)A_{1}(x)\right)=\sum_{j=0}^{4}\alpha_{j}x^{j}$ (see (1.1),(1.8) and (1.10)). By their Theorem 1.11, the polynomial $D(x)$ has 3 real roots $a_{1}$, $a_{2}$, $a_{3}$ satisfying $-1<a_{1}<a_{2}<1<a_{3}<\infty$. If $D(x)$ has degree $4$, it has a fourth root $a_{4}\in(a_{3},\infty)\cup(-\infty,-1)$, while in the case that $D(x)$ only has degree $3$ we take $a_{4}=\infty$. The polynomial $E(y)=\sum_{j=0}^{4}\beta_{j}y^{j}$ and its roots $b_{1}$, $b_{2}$, $b_{3}$ and $b_{4}$ are defined similarly by swapping the roles of $x$ and $y$.

The following is proposition 2.1 in \cite{dreyfus2019differential}, using the expression (2.16) for $y(\omega)$.
\begin{Proposition}\label{prop:xyom}
The curve $\overline{E_{t}}$ admits a uniformisation of the form
\[\overline{E_{t}}=\{(x(\omega),y(\omega)):\omega\in\mathbb{C}/(\mathbb{Z}\omega_{1}+\mathbb{Z}\omega_{2})\},\]
where $x(\omega)$ is given by
\[x(\omega) = \left\{ 
\begin{array}{ c lr }
 & \displaystyle a_{4}+\frac{D'(a_{4})}{\wp(\omega,\omega_{1},\omega_{2})-\frac{1}{6}D''(a_{4})},\quad& \textrm{if } a_{4}\neq\infty, \\
 & \displaystyle \frac{3\wp(\omega,\omega_{1},\omega_{2})-\alpha_{2}}{3\alpha_{3}}\quad, &\textrm{if } a_{4}=\infty,
\end{array}
\right.
\]
and $y(\omega)$ is given by
\[y(\omega) = \left\{ 
\begin{array}{ c lr }
 & \displaystyle b_{4}+\frac{E'(b_{4})}{\wp(\omega-\omega_{3}/2,\omega_{1},\omega_{2})-\frac{1}{6}E''(b_{4})},\quad& \textrm{if } b_{4}\neq\infty, \\
 & \displaystyle \frac{3\wp(\omega-\omega_{3}/2,\omega_{1},\omega_{2})-\beta_{2}}{3\beta_{3}},\quad &\textrm{if } b_{4}=\infty,
\end{array}
\right.
\]
where $\omega_{1}\in i\mathbb{R}$ and $\omega_{2},\omega_{3}\in\mathbb{R}$ are given by
\begin{align*}
\omega_{1}=i\int_{a_{3}}^{a_{4}}\frac{1}{\sqrt{-D(x)}}dx,\quad \omega_{2}=\int_{a_{4}}^{a_{1}}\frac{1}{\sqrt{D(x)}}dx,\quad
\omega_{3}=\int_{a_{4}}^{x_{4}}\frac{1}{\sqrt{D(x)}}dx,
\end{align*}
where $K(x_{4},b_{4})=0$. when the integrals above take the form $\int_{a}^{b}$ with $a>b$, they are defined as $\int_{a}^{\infty}+\int_{-\infty}^{b}$.
\end{Proposition}

This is related to the parameterisation in Lemma \ref{lem:param} by
\begin{align*}
X(z)&=x\left(\frac{\omega_{2}}{2}+\frac{\omega_{3}}{4}-\frac{\omega_{1}}{\pi}z\right)\\
Y(z)&=y\left(\frac{\omega_{2}}{2}+\frac{\omega_{3}}{4}-\frac{\omega_{1}}{\pi}z\right)\\
\tau&=-\frac{\omega_{2}}{\omega_{1}}\\
\gamma&=-\frac{\omega_{3}\pi}{2\omega_{1}}
\end{align*}

Under these transformations, and writing $\omega=\frac{\omega_{2}}{2}+\frac{\omega_{3}}{4}-\frac{\omega_{1}}{\pi}z$, Lemma \ref{lem:param} is equivalent to the following lemma
\begin{Lemma}\label{lem:xwparam}
The numbers $\omega_{1}\in i\mathbb{R}$ and $\omega_{2},\omega_{3}\in \mathbb{R}$ satisfy $\Im(-\omega_{2}/\omega_{1})>\Im(-\omega_{3}/\omega_{1})>0$. Moreover, the meromorphic functions $x(\omega),y(\omega):\mathbb{C}\to\mathbb{C}\cup\{\infty\}$ satisfy 
\begin{enumerate}[label={\rm(\roman*)},ref={\rm(\roman*)}]
\item $K(x(\omega),y(\omega))=0$,
\item $x(\omega)=x(\omega-\omega_{1})=x(\omega+\omega_{2})=x(\omega_{2}-\omega)$,
\item $y(\omega)=y(\omega-\omega_{1})=y(\omega+\omega_{2})=y(\omega_{2}+\omega_{3}-\omega)$,
\item $|x(\frac{\omega_{2}}{2})|,|y(\frac{\omega_{2}}{2}+\frac{\omega_{3}}{2})|<1$,
\item counting with multiplicity, the functions $x(\omega)$ and $y(\omega)$ each contain two poles and two roots in each fundamental domain $\{z_{c}-r_{1}\omega_{1}+r_{2}\omega_{2}:r_{1},r_{2}\in[0,1)\}$.
\end{enumerate}
Moreover, $x(\omega)$ and $y(\omega)$ are differentially algebraic with respect to $\omega$ and $t$, while $\frac{\omega_{2}}{\omega_{1}}$ and $\frac{\omega_{3}}{\omega_{1}}$ are differentially algebraic as functions of $t$.
\end{Lemma}
\begin{proof}
The fact that $\Im(-\omega_{2}/\omega_{1})>\Im(-\omega_{3}/\omega_{1})>0$ follows immediately from $\frac{\omega_{1}}{i}>0$ and $0<\omega_{3}<\omega_{2}$, which is proven in Lemma 2.6 in \cite{dreyfus2019differential}.

The condition (i) is part of the statement of Proposition \ref{prop:xyom}. The claims (ii) and (iii) follow from the definitions of $x(\omega)$ and $y(\omega)$ and the fact that
\[\wp(\omega,\omega_{1},\omega_{2})=\wp(\omega-\omega_{1},\omega_{1},\omega_{2})=\wp(\omega+\omega_{2},\omega_{1},\omega_{2})=\wp(-\omega,\omega_{1},\omega_{2}).\]
From Lemma 2.3 in \cite{dreyfus2019differential}, we have $x(\frac{\omega_{2}}{2})=a_{1}$ and $y(\frac{\omega_{2}+\omega_{3}}{2})=b_{1}$, proving (iv), as $|a_{1}|,|b_{1}|<1$.
Finally, (v) follows from the definitions of $x(\omega)$ and $y(\omega)$ as it is a well known property of the Weierstrass function $\wp$ that it takes each value in $\mathbb{C}\cup\infty$ exactly twice (counting with multiplicity) in each fundamental domain.

finally we will show that all terms in the expression are D-algebraic both as functions of $\omega$ and $t$. This follows from results of \cite{bernardi2017counting}, in particular their Proposition 6.7 shows that $\wp$ is D-algebraic in all of it's variables, while their Lemma 6.10 shows that $\omega_{1}$, $\omega_{2}$ and $\omega_{3}$ are even D-finite in $t$.
\end{proof}
So, we have now proven Lemma \ref{lem:param}. Similarly, Lemma \ref{lem:Omega} is equivalent to the Lemma below, with $\Omega_{s}:=\frac{\pi}{\omega_{1}}\left(\frac{\omega_{3}}{4}-\Dnn_{s}\right)$
\begin{Lemma}\label{lem:appOmega}
The complex plane can be partitioned into simply connected regions $\{\Dnn_{s}\}_{s\in\mathbb{Z}}$ satisfying
\begin{align}
\bigcup_{s\in\mathbb{Z}}\Dnn_{4s}\cup\Dnn_{4s+1}&=\{\omega\in\mathbb{C}:|y(\omega)|<1\},\label{eq:D_ysmall}\\
\bigcup_{s\in\mathbb{Z}}\Dnn_{4s-2}\cup\Dnn_{4s-1}&=\{\omega\in\mathbb{C}:|y(\omega)|\geq1\},\label{eq:D_ybig}\\
\bigcup_{s\in\mathbb{Z}}\Dnn_{4s-1}\cup\Dnn_{4s}&=\{\omega\in\mathbb{C}:|x(\omega)|<1\},\label{eq:D_xsmall}\\
\bigcup_{s\in\mathbb{Z}}\Dnn_{4s+1}\cup\Dnn_{4s+2}&=\{\omega\in\mathbb{C}:|x(\omega)|\geq1\},\label{eq:D_xbig}
\end{align}
and for $s\in\mathbb{Z}$,
\begin{align}
\omega_{1}+\Dnn_{s}&=\Dnn_{s},\label{eq:D_w1periodic}\\
(s+1)\omega_{2}+\omega_{3}-\Dnn_{2s}\cup\Dnn_{2s+1}&=\Dnn_{2s}\cup\Dnn_{2s+1}\supset \frac{(s+1)\omega_{2}+\omega_{3}}{2}+i\mathbb{R},\label{eq:D_flip_yfix}\\
(s+1)\omega_{2}-\Dnn_{2s}\cup\Dnn_{2s-1}&=\Dnn_{2s}\cup\Dnn_{2s-1}\supset \frac{(s+1)\omega_{2}}{2}+i\mathbb{R}.\label{eq:D_flip_xfix}
\end{align}
\end{Lemma}
\begin{proof}
In \cite{dreyfus2019differential}, the Authors define $\Dnn_{x}=\{(x,y)\in\overline{E_{t}}:|x|<1\}$, $\Dnn_{y}=\{(x,y)\in \overline{E_{t}}:|y|<1\}$ and $\Dnn=\Dnn_{x}\cup\Dnn_{y}$. Moreover in Lemma 2.8 they show that these sets are connected.

Next, writing $\Lambda(\omega):=(x(\omega),y(\omega))$, as the homeomorphism from $\mathbb{C}/(\mathbb{Z}\omega_{1}+\mathbb{Z}\omega_{2})$ to $\overline{E_{t}}$ and $\tilde{\Lambda}$ being the corresponding function with domain $\mathbb{C}$, they define $\tilde{\Dnn}$ to be a connected component of the set $\tilde{\Lambda}^{-1}(\Dnn).$ Their only other condition on $\tilde{\Dnn}$ is that it intersects the fundamental parallelogram $\omega_{1}[0,1)+\omega_{2}[0,1)$. Since we know $\frac{\omega_{2}}{2}$ is in this parallelogram and $|x(\frac{\omega_{2}}{2})|<1$, we may assume that $\tilde{\Dnn}$ is the connected component containing $\frac{\omega_{2}}{2}$. They also define connected components $\tilde{\Dnn}_{x}$ and $\tilde{\Dnn}_{y}$ of $\tilde{\Lambda}^{-1}(\Dnn_{x})$ and $\tilde{\Lambda}^{-1}(\Dnn_{y})$, respectively, satisfying \[\tilde{\Dnn}=\tilde{\Dnn}_{x}\cup\tilde{\Dnn}_{y}.\]
Finally they define non-intersecting infinite paths $\tilde{\Gamma}^{+}_{x},\tilde{\Gamma}^{-}_{x},\tilde{\Gamma}^{+}_{y},\tilde{\Gamma}^{+}_{y}$, and they prove the following in and above Lemma  2.9:
\begin{itemize}
\item For $\omega\in\tilde{\Gamma}^{+}_{x}$, we have $|x(\omega)|=1$ and $|y(\omega)|>1$ while for $\omega\in\tilde{\Gamma}^{-}_{x}$, we have $|x(\omega)|=1$ and $|y(\omega)|<1$
\item For $\omega\in\tilde{\Gamma}^{+}_{y}$, we have $|y(\omega)|=1$ and $|x(\omega)|>1$ while for $\omega\in\tilde{\Gamma}^{-}_{y}$, we have $|y(\omega)|=1$ and $|x(\omega)|<1$
\item The paths $\tilde{\Gamma}^{+}_{x}$ and $\tilde{\Gamma}^{-}_{x}$ are $\omega_{1}$-periodic and do not cross the vertical line through $\frac{\omega_{2}}{2}$.
\item The paths $\tilde{\Gamma}^{+}_{y}$ and $\tilde{\Gamma}^{-}_{y}$ are $\omega_{1}$-periodic and do not cross the vertical line through $\frac{\omega_{2}+\omega_{3}}{2}$.
\item The domain $\tilde{\Dnn}$ is delimited by a left boundary $\tilde{\Gamma}^{+}_{x}$ and a right boundary $\tilde{\Gamma}^{+}_{y}$ and it contains $\tilde{\Gamma}^{-1}_{x}$ and $\tilde{\Gamma}^{-1}_{y}$.
\item The domain $\tilde{\Dnn}_{x}$ is delimited by $\tilde{\Gamma}^{+}_{x}$ and $\tilde{\Gamma}^{-}_{x}$, while the domain $\tilde{\Dnn}_{y}$ is delimited by $\tilde{\Gamma}^{-}_{y}$ and $\tilde{\Gamma}^{+}_{y}$.
\end{itemize}
We know that $\tilde{\Dnn}$ and $\tilde{\Dnn}+\omega_{2}$ do not intersect, as this would contradict the claim that $\tilde{\Dnn}$ is a connected component of $\tilde{\Lambda}^{-1}(\Dnn)$. This implies that $\tilde{\Gamma}^{+}_{x}+\omega_{2}$ is to the right of $\tilde{\Gamma}^{+}_{y}$. Moreover, the fact that $\tilde{\Dnn}$ is connected implies that $\tilde{\Gamma}^{-}_{x}$ is to the right of $\tilde{\Gamma}^{-}_{y}$. So the lines $\tilde{\Gamma}^{+}_{x}$, $\tilde{\Gamma}^{-}_{y}$, $\tilde{\Gamma}^{-}_{x}$, $\tilde{\Gamma}^{+}_{y}$ and $\tilde{\Gamma}^{+}_{x}+\omega_{2}$ are in that order from left to right. finally we can define the sets $\Dnn_{s}$. For $s\in\mathbb{Z}$, we define
\begin{itemize}
\item $\Dnn_{4s-1}=\tilde{\Dnn}_{x}\setminus\tilde{\Dnn}_{y}+s\omega_{2}$, which is delimited by a left boundary $\tilde{\Gamma}^{+}_{x}+s\omega_{2}$ (not included in $\Dnn_{4s-1}$) and a right boundary $\tilde{\Gamma}^{-}_{y}+s\omega_{2}$ (included in $\Dnn_{4s-1}$),
\item $\Dnn_{4s}=\tilde{\Dnn}_{x}\cap\tilde{\Dnn}_{y}+s\omega_{2}$, which is delimited by a left boundary $\tilde{\Gamma}^{-}_{y}+s\omega_{2}$ and a right boundary $\tilde{\Gamma}^{-}_{x}+s\omega_{2}$,
\item $\Dnn_{4s+1}=\tilde{\Dnn}_{y}\setminus\tilde{\Dnn}_{x}+s\omega_{2}$, which is delimited by a left boundary $\tilde{\Gamma}^{-}_{x}+s\omega_{2}$ (included in $\Dnn_{4s+1}$) and a right boundary $\tilde{\Gamma}^{+}_{y}+s\omega_{2}$,
\item $\Dnn_{4s+2}$ is the closed set delimited by a left boundary $\tilde{\Gamma}^{+}_{y}+s\omega_{2}$ and a right boundary $\tilde{\Gamma}^{+}_{x}+(s+1)\omega_{2}$.
\end{itemize}
It is clear from these definitions that these sets partition $\mathbb{C}$. Moreover, since the paths $\Gamma_{x}^{\pm}$ and $\Gamma_{y}^{\pm}$ are $\omega_{1}$-periodic, we know that $\tilde{\Dnn}_{y}$ is also $\omega_{1}$-periodic, that is, \eqref{eq:D_w1periodic} holds.

Now, to prove equations \eqref{eq:D_ysmall}-\eqref{eq:D_xbig}, 
note that since $\Dnn_{y}$ is connected, we must have $\tilde{\Lambda}(\tilde{\Dnn}_{y})=\Dnn_{y}$, so
\[\tilde{\Lambda}^{-1}(\Dnn_{y})=\tilde{\Dnn}_{y}+\mathbb{Z}\omega_{1}+\mathbb{Z}\omega_{2},\]
since $\Lambda$ is a bijection from $\mathbb{C}/(\mathbb{Z}\omega_{1}+\mathbb{Z}\omega_{2})$ to $\overline{E_{t}}$. Hence, using \eqref{eq:D_w1periodic}, we have
\[\tilde{\Lambda}^{-1}(\Dnn_{y})=\tilde{\Dnn}_{y}+\mathbb{Z}\omega_{2}.\]
This equation is equivalent to \eqref{eq:D_ysmall}. Moreover, \eqref{eq:D_ybig} follows by taking the complement of both sides in \eqref{eq:D_ysmall}. The equations \eqref{eq:D_xsmall} and \eqref{eq:D_xbig} follow similarly from considering $\tilde{\Lambda}^{-1}(\Dnn_{x})$.

The only remaining equations to prove are \eqref{eq:D_flip_yfix} and \eqref{eq:D_flip_xfix}. Recall that $\tilde{\Dnn}_{x}=\Dnn_{-1}\cup\Dnn_{0}$ is a connected component of $\tilde{\Lambda}^{-1}(\Dnn_{x})=\{\omega\in\mathbb{C}:|x(\omega)|<1\}$. Then since $x(\omega)=x(-\omega)=x(\omega_{2}-\omega)$, the set $\omega_{2}-\tilde{\Dnn}_{x}$ is also a connected component of $\tilde{\Lambda}^{-1}(\Dnn_{x})$. Moreover, $\tilde{\Dnn}_{x}$ and $\omega_{2}-\tilde{\Dnn}_{x}$ both contain $\frac{\omega_{2}}{2}$, so they must be the same set, that is $\tilde{\Dnn}_{x}=\omega_{2}-\tilde{\Dnn}_{x}$, so $\tilde{\Gamma}^{+}_{x}=\omega_{2}-\tilde{\Gamma}^{-}_{x}$. This implies that the transformation $\omega\to (2s+1)\omega_{2}-\omega$ swaps the boundaries $\tilde{\Gamma}_{x}^{+}+s\omega_{2}$ and $\tilde{\Gamma}_{x}^{-}+s\omega_{2}$ of $\Dnn_{4s}\cup\Dnn_{4s-1}$, while the transformation $\omega\to 2s\omega_{2}-\omega$ swaps the boundaries $\tilde{\Gamma}_{x}^{-}+(s-1)\omega_{2}$ and $\tilde{\Gamma}_{x}^{+}+s\omega_{2}$ of $\Dnn_{4s}\cup\Dnn_{4s-1}$. Together these imply the first part of \eqref{eq:D_flip_xfix}. Now we will show that $\Dnn_{2s}\cup\Dnn_{2s-1}$ contains the line $\frac{(s+1)\omega_{2}}{2}+i\mathbb{R}$. Consider the reflection in this line, which is given by $\omega\mapsto(s+1)\omega_{2}-\overline{\omega}$. We know that the transformation $\omega\mapsto(s+1)\omega_{2}-\omega$ swaps the boundaries of $\Dnn_{2s}\cup\Dnn_{2s-1}$, so by symmetry in the real line, the reflection $\omega\mapsto(s+1)\omega_{2}-\overline{\omega}$ also swaps the boundaries of $\Dnn_{2s}\cup\Dnn_{2s-1}$. Hence the line $\frac{(s+1)\omega_{2}}{2}+i\mathbb{R}$ at the centre of this reflection must be contained in $\Dnn_{2s}\cup\Dnn_{2s-1}$. 
This completes the proof of \eqref{eq:D_flip_xfix}. To prove \eqref{eq:D_flip_yfix} similarly, we just need to show that $\frac{\omega_{2}+\omega_{3}}{2}\in \tilde{\Dnn}_{y}$. Note that $\frac{\omega_{2}}{2}<\frac{\omega_{2}+\omega_{3}}{2}<\omega_{2}$. We know that $\frac{\omega_{2}}{2}\in\Dnn_{-1}\cup\Dnn_{0}$. Moreover, $2\omega_{2}-\Gamma^{-}_{x}=(\omega_{2}+\Gamma^{+}_{x})$, hence $\omega_{2}$ must lie between the lines $\Gamma^{-}_{x}=(\omega_{2}+\Gamma^{+}_{x})$. So, both $\frac{\omega_{2}}{2}$ and $\omega_{2}$ lie between $\Gamma^{+}_{x}$ and $\Gamma^{+}_{x}+\omega_{2}$, hence $\frac{\omega_{2}+\omega_{3}}{2}$ also lies between these lines, that is $\frac{\omega_{2}+\omega_{3}}{2}\in \Dnn_{-1}\cup\Dnn_{0}\cup\Dnn_{1}\cup\Dnn_{2}$. Now, we know from Lemma \ref{lem:xwparam} that $|y(\frac{\omega_{2}+\omega_{3}}{2})|<1$, so $\frac{\omega_{2}+\omega_{3}}{2}\in \Dnn_{0}\cup\Dnn_{1}=\Dnn_{y}$, as required.
\end{proof}

\section{Nature of analytic functions of $X(z)$}\label{ap:nature_analytic}
In this appendix we assume that $f:\mathbb{C}\to\mathbb{C}\cup\{\infty\}$ is a meromorphic function, with $\pi$ as a period, $\Lambda\subset\mathbb{C}$ is a set with non-empty interior and $A:X(\Lambda)\to\mathbb{C}\cup\{\infty\}$ is a meromorphic function satisfying $A(X(z))=f(z)$ for $z\in\Lambda$. We define properties $X$-rational, $X$-algebraic, $X$-D-finite and $X$-D-algebraic and in each of the four cases we show that the function $f(z)$ has the property $X$-$\mathcal{P}$ if and only if $A(x)$ has the property $\mathcal{P}$. Importantly, the properties $X$-$\mathcal{P}$ do not depend on the set $\Lambda$. Similarly, we assume $g:\mathbb{C}\to\mathbb{C}\cup\{\infty\}$ is a meromorphic function with period $\pi$ and $B:Y(\Lambda)\to\mathbb{C}\cup\{\infty\}$ is a meromorphic function satisfying $B(Y(z))=g(z)$ for $z\in\Lambda$. We then define properties $Y$-$\mathcal{P}$ such that $g(z)$ has the property $Y$-$\mathcal{P}$ if and only if $B(y)$ has the property $\mathcal{P}$. In fact, as we will see if $\mathcal{P}$ is one of the properties {\em algebraic}, {\em D-finite} or {\em D-algebraic} then the properties $X$-$\mathcal{P}$ and $Y$-$\mathcal{P}$ are the same. The results in this appendix are usually applied in the article in cases where $\Lambda=\Omega_{j}$ for some $j$, in which case $f(z)$ and $g(z)$ automatically have $\pi$ as a period, as this property is inherited from $X$ and $Y$.

\begin{Definition}\label{def:X-rat}
We say that the function $f:\mathbb{C}\to\mathbb{C}\cup\{\infty\}$ with period $\pi$ is {\em $X$-rational} if it satisfies
\[f(z)=f(z+\pi\tau)=f(-\gamma-z),\]
for all $z\in\mathbb{C}$. We say that the function $g:\mathbb{C}\to\mathbb{C}\cup\{\infty\}$ is {\em $Y$-rational} if it satisfies
\[f(z)=f(z+\pi\tau)=f(\gamma-z).\]
\end{Definition}

\begin{Proposition}\label{prop:rationalofXorY}The function $A(x)$ is rational if and only if $f(z)=A(X(z))$ is $X$-rational. Similarly, $B(y)$ is rational if and only if $g(z)=B(Y(z))$ is $Y$-rational.\end{Proposition}
\begin{proof}
If $A$ is a rational function then by meromorphic extension we must have $f(z)=A(X(z))$ for all $z\in\mathbb{C}$, so $f(z)$ is $X$-rational due to the fact that
\[X(z)=X(z+\pi)=X(z+\pi\tau)=X(-\gamma-z).\]

For the converse, consider the transformation of $\wp$ that defines $X(z)$ and $Y(z)$ given in Appendix \ref{ap:param_thm}. Under this transformation, this Proposition is equivalent to the classical result that any even elliptic function is a rational function of the Weierstrass function $\wp$ with the same periods (see e.g., \cite[page 44]{akhiezer1990elements}). Alternatively one could prove this directly by constructing a rational function of $X(z)$ with poles at the same points and of the same nature as those of $f(z)$, thereby proving that $f(z)$ is this rational function of $X(z)$.
\end{proof}
\textbf{Remark:} An equivalent argument shows that a function $h(z)$ is a rational function of $W(z)$ (see Definition \ref{def:W}) if and only if \[h(z)=h(z+\pi)=h(\pi\tau-\gamma-z)=h(2(\pi\tau-\gamma)+z).\]

\begin{Definition}\label{def:X-alg}
We say that the function $f:\mathbb{C}\to\mathbb{C}\cup\{\infty\}$ with period $\pi$ is {\em $X$-algebraic} or equivalently {\em $Y$-algebraic} if it has $m\pi\tau$ as a period for some positive integer $m$.
\end{Definition}

\begin{Proposition}\label{prop:algebraicofXorY}The function $A(x)$ is algebraic if and only if $f(z)=A(X(z))$ is $X$-algebraic.\end{Proposition}
\begin{proof}
If $f(z)$ is $X$-algebraic then there is a positive integer $m$, such that $f(z)$ and $X(z)$ share the periods $\pi$ and $m\pi\tau$. Hence $A(x)$ is algebraic due to the classical result that any two elliptic functions with the same periods are related by some non-trivial polynomial equation $P(f(z),X(z))=0$.

For the converse, assume that $A(x)$ is algebraic. Then there is some non-polynomial $P$ satisfying $P(A(x),x)=0$, and hence $P(f(z),X(z))=0$ for $z\in\Lambda$, and by meromorphic extension we must have $P(f(z),X(z))=0$ for all $z\in\mathbb{C}$. There can only be finitely many values $z_{c}\in[0,\pi)$ satisfying $P(a,X(z_c))=0$ for all $a$, and for all the rest, there can only be finitely many values of $a$ satisfying $P(a,X(z_c))=0$, so $f(z)$ takes only finitely many values for $z\in z_c+\pi\tau\mathbb{Z}$. Hence for each $z_{c}$ there are some $m_{1},m_{2}\in\mathbb{Z}$ with $m_{1}\neq m_{2}$ satisfying $f(z_{c}+m_{1}\pi\tau)=f(z_{c}+m_{2}\pi\tau)$. Moreover, since there are only countably many choices for $m_{1},m_{2}$, there must be infinitely many values $z_{c}$ corresponding to the same pair $m_{1},m_{2}$. Hence $f(z+m_{1}\pi\tau)-f(z+m_{2}\pi\tau)=0$ for infinitely many values $z\in[0,\pi)$, and hence infinitely many different values $X(z)$. Moreover, both  $f(z+m_{1}\pi\tau)$ and $f(z+m_{2}\pi\tau)$ are algebraic functions of $X(z)$ (as the both satisfy $P(f(z+m_{j}\pi\tau),X(z))=0$), so $f(z+m_{1}\pi\tau)-f(z+m_{2}\pi\tau)$ is an algebraic function of $X(z)$ with infinitely many roots, so it must be the $0$ function. Hence $f(z+m_{1}\pi\tau)=f(z+m_{2}\pi\tau)$, so setting $m=m_{1}-m_{2}$ yields the desired result.
\end{proof}
\textbf{Remark:} Similarly $B(y)$ is algebraic if and only if $g(z)=B(Y(z))$ satisfies the exact same condition.

\begin{Definition}\label{def:X-D-finite}
We say that the function $f:\mathbb{C}\to\mathbb{C}\cup\{\infty\}$ with period $\pi$ is {\em $X$-D-finite} or equivalently {\em $Y$-D-finite} if it satisfies an equation of the form
\[S(z)+\sum_{j=0}^{n} f^{(j)}(z)S_{j}(z)=0,\]
where $S(z)$ and each $S_{j}(z)$ is an elliptic function with periods $\pi$ and $\pi\tau$, and $S_{j}(z)$ is not the zero function.
\end{Definition}

We will show that this is also equivalent to the following, seemingly weaker property: 

\begin{Definition}\label{def:X-D-finite-weak}
We say that the function $f:\mathbb{C}\to\mathbb{C}\cup\{\infty\}$ is weakly {\em $X$-D-finite} if it satisfies an equation of the form
\[S(z)+\sum_{j=0}^{n} f^{(j)}(z)S_{j}(z)=0,\]
where $S(z)$ and each $S_{j}(z)$ is an elliptic function with periods $\pi$ and $m\pi\tau$, for some $m\in\mathbb{Z}\setminus\{0\}$ and $S_{j}(z)$ is not the zero function.
\end{Definition}

\begin{Proposition}\label{prop:D-finiteofXorY}The following are equivalent:
\begin{itemize}
\item The function $A(x)$ is D-finite
\item The function $f(z)$ satisfying $f(z)=A(X(z))$ is $X$-D-finite
\item The function $f(z)$ satisfying $f(z)=A(X(z))$ is weakly $X$-D-finite
\end{itemize}
\end{Proposition}
\begin{proof}
Clearly if $f(z)$ is $X$-D-finite then it is weakly $X$-D-finite. We will now assume that $f(z)$ is weakly $X$-D-finite and prove that $A(x)$ is D-finite. We have
\[S(z)+\sum_{j=0}^{n} f^{(j)}(z)S_{j}(z)=0.\]
Considering this equation for $z\in\Lambda$, where $f(z)=A(X(z))$, we can apply the chain rule to write
\[f^{(k)}(z)=\left(\frac{d}{dz}\right)^{k}A(X(z))=\sum_{j=0}^{k}A^{(j)}(X(z))T_{j}(z),\]
where each $T_{j}(z)$ is a rational function of $X(z)$ and its derivatives, with $T_{k}(z)=X'(z)^{k}\neq0$. In particular, each $T_{j}(z)$ is an elliptic function with periods $\pi$ and $\pi\tau$, so we have a new equation
\[S(z)+\sum_{j=0}^{n} A^{(j)}(X(z))U_{j}(z)=0,\]
where $S(z)$ and each $U_{j}(z)$ is an elliptic function with periods $\pi$ and $m\pi\tau$, and $U_{n}(z)=S_{n}(z)T_{n}(z)\neq0$. 

Now, consider a neighbourhood $\Delta\subset\Lambda$ on which $X:\Delta\to\mathbb{C}$ is injective, so $X^{-1}:X(\Delta)\to\Delta$ is well defined. Then for $x\in X(\Delta)$, define $\Ugf_{j}(x):=U_{j}(X^{-1}(x))$ and $\Sgf(x):=S(X^{-1}(x))$. Since $U_{j}$ and $S$ have $\pi$ and $m\pi\tau$ as periods, the functions $\Sgf(x)$ and $\Ugf_{j}(x)$ are algebraic by Proposition \ref{prop:algebraicofXorY}. Now for $x\in\Delta$, we have
\[\Sgf(x)+\sum_{j=0}^{n} A^{(j)}(x)\Ugf_{j}(x)=0,\]
so $A(x)$ satisfies a non-trivial linear differential equation with coefficients algebraic in $x$. It is a classical result that the existence of such an equation implies the existence of a linear differential equation (see, for example, \cite{Poincare1884,Schmidt75}), so $A(x)$ is D-finite.

For the final implication, we assume that $A(x)$ is D-finite and we will show that $f(z)$ is $X$-D-finite. We can write 
\[\sum_{j=0}^{n}\Ugf_{j}(x)\left(\frac{d}{dx}\right)^{j}A(x)=0,\]
where each $\Ugf_{j}(x)$ is a polynomial, with $\Ugf_{n}$ non-zero. Substituting $x\to X(z)$ for $z\in\Lambda$ yields
\[\sum_{j=0}^{m}\Ugf_{j}(X(z))\left(\frac{1}{X'(z)}\frac{d}{dz}\right)^{j}f(z)=0.\]
Moreover, since the expression on the left hand side is meromorphic on $\mathbb{C}$, the equation must hold on all of $\mathbb{C}$. Using the product rule yields an equation of the form described in Definition \ref{def:X-D-finite}, so $f(z)$ is $X$-D-finite.
\end{proof}
\textbf{Remark:} Similarly $B(y)$ is D-finite if and only if $g(z)=B(Y(z))$ is $Y$-D-finite.

Below we prove a simple consequence of the Lemma above regarding the possible positions of poles of such a function $f(z)$. This will be useful for proving that certain series are not D-finite.

\begin{Lemma}\label{lem:D-fini_z}
If $f(z)$ is $X$-D-finite, then its poles fall into only finitely many classes $z_{c}+\pi\mathbb{Z}+\pi\tau\mathbb{Z}$.
\end{Lemma}
\begin{proof}
From Proposition \ref{prop:D-finiteofXorY} we can write
\[S(z)+\sum_{j=0}^{n} f^{(j)}(z)S_{j}(z)=0,\]
where $S(z)$ and each $S_{j}(z)$ is an elliptic function with periods $\pi$ and $\pi\tau$, and $S_{j}(z)$ is not the zero function.

We will start by showing that any pole $z_{c}$ of $f(z)$ is either a root of $S_{n}(z)$ or a pole of one of the functions $S_{j}(z)$ or $S(z)$. Let $z_{c}$ be a pole of $f(z)$ and let $k$ be the order of the pole $z_{c}$ of $f(z)$. If $z_{c}$ is not a root or pole of $S_{n}(z)$ then $f^{(n)}(z)S_{n}(z)$ has a pole at $z=z_{c}$ of order $k+n$. If, in addition, $z_{c}$ is not a pole of any of the functions $S_{j}(z)$, then each term $f^{(j)}(z)S_{j}(z)$ for $j<n$ has a pole at $z=z_c$ of order at most $k+j<k+n$. Hence the entire sum has a pole of order $k+n$ at $z=z_c$, which is impossible as the sum is $0$.

Hence, any pole $z_{c}$ of $f(z)$ is either a root of $S_{n}(z)$ or a pole of one of the functions $S_{j}(z)$ or $S(z)$. But each of these functions is an elliptic function with periods $\pi$ and $\pi\tau$, so the poles and roots fall into finitely many classes $z_{c}+\pi\mathbb{Z}+\pi\tau\mathbb{Z}$. Hence, the poles of $f(z)$ fall into only finitely many such classes.
\end{proof}

\begin{Definition}\label{def:X-D-alg}
We say that the function $f:\mathbb{C}\to\mathbb{C}\cup\{\infty\}$ is {\em $X$-D-algebraic} or equivalently {\em $Y$-D-algebraic} if it is D-algebraic.
\end{Definition}

\begin{Proposition}\label{prop:D-algebraicofXorY}The function $A(x)$ is D-algebraic if and only if $f(z)=A(X(z))$ is a D-algebraic function of $z$, that is, if and only if $f(z)$ is $X$-D-algebraic.\end{Proposition}
\begin{proof}
Since D-algebraic functions are closed under function composition and taking inverses, it suffices to prove that $X(z)$ is a D-algebraic function. Indeed $X'(z)$ and $X(z)$ are elliptic functions with the same periods, so they are related by some non-trivial polynomial equation $P(X(z),X'(z))=0$. This is a non-trivial differential algebraic equation satisfied by $X(z)$.
\end{proof}
\textbf{Remark:} Similarly $B(y)$ is D-algebraic if and only if $g(z)=B(Y(z))$ is differentially algebraic.

Finally we have the main result of this appendix:
\begin{Proposition}\label{prop:PandXP}
If $\mathcal{P}$ is one of the properties {\em rational}, {\em algebraic}, {\em D-finite} or {\em D-algebraic}, then $A(x)$ satisfies the property $\mathcal{P}$ if and only if $f(z)=A(X(z))$ satisfies the property $X$-$\mathcal{P}$. 
\end{Proposition}
\begin{proof}The result is a combination of Propositions \ref{prop:rationalofXorY}, \ref{prop:algebraicofXorY}, \ref{prop:D-finiteofXorY} and \ref{prop:D-algebraicofXorY}.
\end{proof}

\begin{Proposition} \label{prop:XP_closure_properties}Let $\mathcal{P}$ be one of the properties {\em Algebraic}, {\em D-finite} or {\em D-Algebraic}. If $f,g:\mathbb{C}\to\mathbb{C}\cup\{\infty\}$ are functions, each with $\pi$ as a period, which have the property $X$-$\mathcal{P}$ then the functions $f(z)g(z)$, $f(z)+g(z)$ and $f(c-z)$ also have the property $X$-$\mathcal{P}$ for any $c\in\mathbb{C}$. 
\end{Proposition}
\begin{proof}
We first prove that $X$-$\mathcal{P}$ is closed under addition and multiplication, that is that $f(z)g(z)$ and $f(z)+g(z)$ satisfy the property $X$-$\mathcal{P}$. As we will show this is due to the fact that the property $\mathcal{P}$ is closed under addition and multiplication.

Let $\Delta\in\mathbb{C}$ be a non-empty open set on which $X$ is injective, then define $f_{1},g_{1}:X(\Delta)\to\mathbb{C}\cup\infty$ by $f_{1}(x)=f(X^{-1}(x))$ and $g_{1}(x)=g(X^{-1}(x))$, where we take the inverse $X^{-1}(x)\in\Delta$. Then from Proposition \ref{prop:PandXP}, our assumption that $f$ and $g$ satisfy $X$-$\mathcal{P}$ implies that $f_{1}$ and $g_{1}$ satisfy $\mathcal{P}$. Then since $\mathcal{P}$ is closed under addition and multiplication of functions, this implies that $f_{1}(x)+g_{1}(x)$ and $f_{1}(x)g_{1}(x)$ satisfy $\mathcal{P}$. Then since $f(z)+g(z)=f_{1}(X(z))+g_{1}(X(z))$ and $f(z)g(z)=f_{1}(X(z))g_{1}(X(z))$ for $z\in\Delta$, it follows from Proposition \ref{prop:PandXP} that these functions satisfy $X$-$\mathcal{P}$.

Finally we will show that $f(c-z)$ satisfies $X$-$\mathcal{P}$. This follows directly from the definition of $X$-$\mathcal{P}$, so we do it separately for each property $\mathcal{P}$. If $f(z)$ is $X$-algebraic it has $m\pi\tau$ as a period for some integer $m$, so $f(c-z)$ also has $m\pi\tau$ as a period, which means it is $X$-algebraic, as required. If $f(z)$ is $X$-D-algebraic then it is D-algebraic, so $f(c-z)$ is also D-algebraic, which means it is $X$-D-algebraic, as required. finally, if $f(z)$ is $X$-D-finite, then it satisfies an equation of the form
\[S(z)+\sum_{j=0}^{n} f^{(j)}(z)S_{j}(z)=0,\]
which means $\hat{f}(z)=f(c-z)$ satisfies
\[S(c-z)+\sum_{j=0}^{n} (-1)^j\hat{f}^{(j)}(z)S_{j}(c-z)=0,\]
which has the same form. So $f(c-z)$ is $X$-D-finite, as required.
\end{proof}

\section{Group of the walk}\label{ap:group}
In this appendix we will discuss the {\em group of the walk}, which was introduced in \cite{bousquet2010walks} and is a variant on a group defined on analytic functions, previously introduced in the study of random walks in the quadrant \cite{fayolle1999random,malyshev1972analytical}. In fact the original analytic version of this group is closer to the version we use in Section \ref{sec:finite_group}.

As in Appendix \ref{ap:param_thm}, we define Laurent polynomials $A_{-1}(x)$, $A_{0}(x)$, $A_{1}(x)$, $B_{-1}(x)$, $B_{0}(x)$, $B_{1}(x)$. by
\[\Pgf(x,y)=\frac{1}{y}A_{-1}(x)+A_{0}(x)+yA_{1}(x)=\frac{1}{x}B_{-1}(x)+B_{0}(x)+xB_{1}(x).\]
Then we define transformations $\psi$ and $\varphi$ by
\[\psi(x,y)=\left(x,\frac{A_{-1}(x)}{A_{1}(x)y}\right)~~~~~~~~~\text{and}~~~~~~~~~\varphi(x,y)=\left(\frac{B_{-1}(y)}{B_{1}(y)x},y\right),\]
as then
\[\Pgf(x,y)=\Pgf(\psi(x,y))=\Pgf(\varphi(x,y)),\]
so the Kernel $K(x,y)=t\Pgf(x,y)-1$ satisfies similar equations
\[K(x,y)=K(\psi(x,y))=K(\varphi(x,y)).\]
\begin{Definition}
The {\em group of the walk} is defined as the group generated by the transformations $\psi$ and $\varphi$.
\end{Definition}
It has been shown that for unweighted walks in the quarter plane, the group is finite if and only if the generating function $\Qgf(x,y;t)$ is D-finite \cite{bousquet2010walks,bostan2010complete,fayolle2010holonomy}. Under the parameterisation $(x,y)\to (X(z),Y(z))$, the transformations $\psi$ and $\varphi$ are equivalent to $z\to -\gamma-z$ and $z\to \gamma-z$ respectively, in the sense that
\begin{align}
\label{eq:psi_z_transform}\psi(X(z),Y(z))&=(X(-\gamma-z),Y(-\gamma-z)),\\
\label{eq:phi_z_transform}\varphi(X(z),Y(z))&=(X(\gamma-z),Y(\gamma-z)).
\end{align}
To prove, for example, the first of these equations, it suffices to observe that $Y(-\gamma-z)$ is the unique point other than $Y(z)$ satisfying $K(X(z),Y(-\gamma-z))=0$, noting that the only cases where $Y(-\gamma-z)=Y(z)$ are those in which $y=Y(z)$ is a double root of $K(X(z),y)$ and $\gamma-z=z$, in which case $\psi(X(z),Y(z))=(X(z),Y(z))=(X(\gamma-z),Y(\gamma-z))$.

This allows us to prove the following proposition:
\begin{Proposition}\label{prop:finite_group_fixed_t} For $n\in\mathbb{N}$, we have $\frac{2n\gamma}{\pi\tau}\in\mathbb{Z}$ if and only if the transformation $(\psi\circ\varphi)^{n}$ fixes every element $(x,y)\in \overline{E_{t}}$.
\end{Proposition}
\begin{proof} Applying equations \eqref{eq:psi_z_transform} and \eqref{eq:phi_z_transform}, we see that
\[(\psi\circ\varphi)^{n}(X(z),Y(z))=(X(z-2n\gamma),Y(z-2n\gamma)).\]
Hence, if $\frac{2n\gamma}{\pi\tau}\in\mathbb{Z}$, then $(\psi\circ\varphi)^{n}$ fixes every element $(X(z),Y(z))$, and hence every element of $\overline{E_{t}}=\{(X(z),Y(z)):z\in\mathbb{C}\}$.

Conversely, if $(\psi\circ\varphi)^{n}$ fixes every point $(X(z),Y(z))$ in $\overline{E_{t}}$, then $X(z)=X(z-2n\gamma)$ for all $z$, so $2n\gamma$ is a period of $X$. Hence $2n\gamma\in \pi\tau\mathbb{Z}+\pi\mathbb{Z}$. More precisely, since $\pi\tau$ and $\gamma$ are both purely imaginary, this implies  $2n\gamma\in \pi\tau\mathbb{Z}$.\end{proof}

\begin{Proposition}\label{prop:finite_group_general_t} The group of the walk is finite if and only if $\frac{\gamma}{\pi\tau}\in\mathbb{Q}$ for all $t\in\left(0,\frac{1}{P(1,1)}\right)$.
\end{Proposition}
\begin{proof}
If the group of the walk is finite, then let $n$ be the order of the element $(\psi\circ\varphi)$. By Proposition \ref{prop:finite_group_fixed_t}, this implies that $\frac{2n\gamma}{\pi\tau}\in\mathbb{Z}$ for all $t$, so $\frac{\gamma}{\pi\tau}\in\mathbb{Q}$ for all $t$.

Conversely, if $\frac{\gamma}{\pi\tau}\in\mathbb{Q}$ for all $t\in\left(0,\frac{1}{P(1,1)}\right)$, then there is some integer $n_{t}$ for each $t$ satisfying $\frac{2n_{t}\gamma}{\pi\tau}\in\mathbb{Z}$. So by Proposition \ref{prop:finite_group_fixed_t}, we have $(\psi\circ\varphi)^{n_{t}}(x,y)=(x,y)$ for $(x,y)\in\overline{E_{t}}$. Now, since There is an integer $n_{t}$ for each of uncountably many values $t$, there must be some integer $N$ such that $N=n_{t}$ for infinitely many different values $t$. This implies that for each $y$, we have $(\psi\circ\varphi)^{N}(x,y)=(x,y)$ for infinitely many values $x$. But, for fixed $y$, the equation $(\psi\circ\varphi)^{N}(x,y)=(x,y)$ is two polynomial equations of $x$, so if it holds for infinitely many values of $x$ it must hold for all $x$. Hence $(\psi\circ\varphi)^{N}(x,y)=(x,y)$ for all $x,y\in\mathbb{C}$, so the group is finite.
\end{proof}


\section{Differential transcendence criteria}\label{ap:D-trans}
The purpose of this appendix is to present the following proposition, which is a restatement of results in \cite{dreyfus2018nature}. We will use this to prove that Certain generating functions are not differentially algebraic in the cases where the group of the walk is infinite and there is not a decoupling function. 
\begin{Proposition}\label{prop:D-trans_poles}
Let $\omega_{1},\omega_{2},\omega_{3}\in\mathbb{C}$ be linearly independent (over $\mathbb{Z}$) and let $f,b:\mathbb{C}\to\mathbb{C}\cup\{\infty\}$ be meromorphic functions satisfying $b(\omega+\omega_{1})=b(\omega)$, $b(\omega+\omega_{2})=b(\omega)$ and $f(\omega+\omega_{3})-f(\omega)=b(\omega)$.
If $f$ is differentially algebraic then for any pole $q_{0}$ of $b$, the function
\[\sum_{i=1}^{k}b(\omega+n_{i}\omega_{3})\]
does not have a pole at $q_{0}$, where $q_{0}+n_{1}\omega_{3},q_{0}+n_{2}\omega_{3},\ldots, q_{0}+n_{k}\omega_{3}$ are the poles of $b$ in $q_{0}+\omega_{3}\mathbb{Z}$.
\end{Proposition}
\begin{proof}
We will describe how this follows from Proposition 3.6 and Proposition B.2 from \cite{dreyfus2018nature}. We start by describing the application of  \cite[Proposition 3.6]{dreyfus2018nature}. In this proposition, the authors consider the triple $(\mathbb{C}(\overline{E_{t}}),\delta,\tau)$, as a sub-field of $(\mathcal{M}(\mathbb{C}),\frac{d}{d\omega},\tau:\omega\to\omega+\omega_{3})$, where $\overline{E_{t}}$ is an arbitrary elliptic curve and $\mathcal{M}(\mathbb{C})$ denotes the meromorphic functions on $\mathbb{C}$. In particular, $\mathbb{C}(\overline{E_{t}})$ corresponds to the elements of $\mathcal{M}(\mathbb{C})$ that are fixed by two specified independent periods - we take these to be $\omega_{1}$ and $\omega_{2}$. With this specification, \cite[Proposition 3.6]{dreyfus2018nature} states that under precisely the conditions of this Proposition, 
there is a function $g\in \mathcal{M}(\mathbb{C})$ satisfying
$g(\omega)=g(\omega+\omega_{1})=g(\omega+\omega_{2})$, an integer $n\geq0$ and $c_{0},c_{1},\ldots,c_{n-1}\in\mathbb{C}$ satisfying
\[\left(\frac{d}{d\omega}\right)^{n}b(\omega)+c_{n-1}\left(\frac{d}{d\omega}\right)^{n-1}b(\omega)+\cdots+c_{1}\left(\frac{d}{d\omega}\right)b(\omega)+c_{0}b(\omega)=g(\omega+\omega_{3})-g(\omega).\]
In other words, there is a non-zero linear operator $L\in\mathbb{C}[\frac{d}{d\omega}]$ satisfying $L(b)=\tau(g)-g$. This means that $b$ and $g$ satisfy condition $(1)$ of \cite[Proposition B.2]{dreyfus2018nature}, and so they also satisfy condition $(2)$, which is equivalent to the following:
For any pole $q_{0}$ of $b$, the function
\[\sum_{i=1}^{k}b(\omega+n_{i}\omega_{3})\]
does not have a pole at $q_{0}$, where $q_{0}+n_{1}\omega_{3},q_{0}+n_{2}\omega_{3},\ldots, q_{0}+n_{k}\omega_{3}$ are the poles of $b$ in $q_{0}+\omega_{3}\mathbb{Z}$.
\end{proof}

We transform this into the following corollary, which allows us to quickly prove that when certain functions are D-algebraic in infinite group cases there must be a decoupling function.

\begin{Corollary}\label{cor:D-trans}
Let $\tau,\gamma_{1},\gamma_{2}\in\mathbb{C}$ such that $\pi,\pi\tau,\gamma_{2}-\gamma_{1}$ are linearly independent (over $\mathbb{Z}$) and let $f_{1},f_{2},h:\mathbb{C}\to\mathbb{C}\cup\{\infty\}$ be meromorphic functions satisfying
\begin{align*}
h(z)&=f_{1}(z)+f_{2}(z)\\
h(z)&=h(z+\pi)=h(z+\pi\tau)\\
f_{1}(z)&=f_{1}(z+\pi)=f_{1}(\gamma_{1}-z)\\
f_{2}(z)&=f_{2}(z+\pi)=f_{2}(\gamma_{2}-z).
\end{align*}
If $f_{2}$ is differentially algebraic there are meromorphic functions $a_{1},a_{2}:\mathbb{C}\to\mathbb{C}\cup\{\infty\}$ satisfying 
\begin{align*}
h(z)&=a_{1}(z)+a_{2}(z)\\
a_{1}(z)&=a_{1}(z+\pi)=a_{1}(z+\pi\tau)=a_{1}(\gamma_{1}-z)\\
a_{2}(z)&=a_{2}(z+\pi)=a_{2}(z+\pi\tau)=a_{2}(\gamma_{2}-z).
\end{align*}
\end{Corollary}
\begin{proof}
Define $b(z)=h(z)-h(\gamma_{1}-z)$. Then
\[b(z)=f_{1}(z)+f_{2}(z)-f_{1}(\gamma_{1}-z)-f_{2}(\gamma_{1}-z)=f_{2}(z)-f_{2}(\gamma_{2}-\gamma_{1}+z).\]
Hence setting $f:=f_{2}$, we have exactly the conditions of Proposition \ref{prop:D-trans_poles}, with $\omega_{1}=\pi$, $\omega_{2}=\pi\tau$ and $\omega_{3}=\gamma_{2}-\gamma_{1}$. Hence, we have the result that for any pole $q_{0}$ of $b(z)$, the function 
\[\sum_{i=1}^{k}b(z+n_{i}\omega_{3})\]
does not have a pole at $q_{0}$, where the sum includes all integers $n_{i}$ such that $q_{0}+n_{i}\omega_{3}$ is a pole of $b$ (including $n_{i}=0$). In fact, this implies that for $N$ sufficiently large the sum
\[\sum_{n=-N}^{N}b(z+n\omega_{3})\]
does not have a pole at $q_{0}$. In order to construct functions $a_{1}$ and $a_{2}$, we will determine an explicit, general form of $b$ using the Jacobi theta function $\th(z,\tau)$. More precisely, we will use the function
\[A_{1}(z)\equiv A(z):=\frac{\th'(z,\tau)}{\th(z,\tau)},\]
which satisfies $A(z)=A(z+\pi)=A(z+\pi\tau)+2i$ and the only poles of $A(z)$ are the points of the lattice $\pi\mathbb{Z}+\pi\tau\mathbb{Z}$, and at each such point $q$ we have
\[A(z)\sim\frac{1}{z-q}.\]
Now for $k>1$, we define
\[A_{k}(z):=\frac{(-1)^{k-1}}{(k-1)!}\left(\frac{d}{dz}\right)^{k-1}A(z),\]
which satisfies $A_{k}(z)=A_{k}(z+\pi)=A_{k}(z+\pi\tau)$, and it's only poles are the points $q\in \pi\mathbb{Z}+\pi\tau\mathbb{Z}$, where it satisfies
\[A_{k}(z)\sim(z-q)^{-k}+O(1).\]

Since $b$ is an elliptic function (with periods $\pi$ and $\pi\tau$) it must have only finitely many poles in any fundamental domain. Hence, there are only finitely many sets $q+\pi\mathbb{Z}+\pi\tau\mathbb{Z}+\omega_{3}\mathbb{Z}$, where $q$ is a pole of $b$. Let $\{q_{0},q_{1},\ldots,q_{K}\}$ be a maximal set of poles of $b$ such that each $q_{k}+\pi\mathbb{Z}+pi\tau\mathbb{Z}+\omega_{3}\mathbb{Z}$ is a different set, and let $N$ be sufficiently large that
\[\sum_{n=-N}^{N}b(z+n\omega_{3})\]
does not have a pole at any of the $K+1$ points $q_{k}$ (a priori there is a different $N$ for each $q_{k}$, but we take the maximum of these). Now for each $k\leq K$ and $n\in[-N,N]$, let 
\[b(z)=s_{n,k,J}(z-q_{k}-n\omega_{3})^{-J}+s_{n,k,J-1}(z-q_{k}-n\omega_{3})^{-J+1}+\cdots+s_{n,k,1}(z-q_{k}-n\omega_{3})^{-1}+O(1),\]
as $z\to q_{k}+n\omega_{3}$. As for $N$, we choose $J$ sufficiently large for all $k$, $n$. Then the result of Proposition \ref{prop:D-trans_poles} is that for each fixed $k$ and $j$, we have
\[\sum_{n=-N}^{N}s_{n,k,j}=0.\]
Now, by our construction, the function
\[\tilde{b}(z):=\sum_{k=0}^{K}\sum_{n=-N}^{N}\sum_{j=1}^{J}s_{n,k,j}A_{j}(z-q_{k}-n\omega_{3})\]
has poles at exactly the same points as $g$, and they are of the same nature, so $b(z)-\tilde{b}(z)$ has no poles. Moreover $\tilde{b}(z+\pi)=\tilde{b}(z)$ and
\[\tilde{b}(z+\pi\tau)-\tilde{b}(z)=\sum_{k=0}^{K}\sum_{n=-N}^{N}\sum_{j=1}^{J}s_{n,k,j}\left(A_{j}(z+\pi\tau-q_{k}-n\omega_{3})-A_{j}(z-q_{k}-n\omega_{3})\right)=\sum_{k=0}^{K}\sum_{n=-N}^{N}-2is_{n,k,1}=0.\]
Hence $b(z)-\tilde{b}(z)$ is an elliptic function with no poles, so it must be constant. Hence for some $c\in\mathbb{C}$ we have $b(z)=\tilde{b}(z)+c.$
Now that we have a explicit form for $b(z)$, we are almost ready to construct $a_{1}$ and $a_{2}$. First we define an auxiliary function
\[p(z):=\sum_{k=0}^{K}\sum_{n=-N}^{N}\sum_{j=1}^{J}\sum_{m=-N}^{n}s_{n,k,j}A_{j}(z-q_{k}-m\omega_{3}),\]
as then
\begin{align*}p(z+\omega_{3})-p(z)&=\sum_{k=0}^{K}\sum_{n=-N}^{N}\sum_{j=1}^{J}\left(\sum_{m=-N}^{n}s_{n,k,j}A_{j}(z-q_{k}-(m-1)\omega_{3})-\sum_{m=-N}^{n}s_{n,k,j}A_{j}(z-q_{k}-m\omega_{3})\right)\\
&=\sum_{k=0}^{K}\sum_{n=-N}^{N}\sum_{j=1}^{J}\left(s_{n,k,j}A_{j}(z-q_{k}+(N+1)\omega_{3})-s_{n,k,j}A_{j}(z-q_{k}-n\omega_{3})\right)\\
&=-\sum_{k=0}^{K}\sum_{n=-N}^{N}\sum_{j=1}^{J}s_{n,k,j}A_{j}(z-q_{k}-n\omega_{3})=-\tilde{b}(z).\end{align*}
Moreover, by the construction, $p(z+\pi)-p(z)=0$ and $p(z+\pi\tau)-p(z)=-2id$ for some constant $d\in\mathbb{C}$ .
Finally, we define $a_{2}(z)=\frac{p(z)+p(\gamma_{2}-z)}{2}$ and $a_{1}(z)=h(z)-a_{2}(z)$. Then $a_{2}(z)=a_{2}(\gamma_{2}-z)=a_{2}(z+\pi)=a_{2}(z+\pi\tau)$ and $h(z)=a_{1}(z)+a_{2}(z)$, so $a_{1}(z)=a_{1}(z+\pi)=a_{1}(z+\pi\tau)$. Hence, it suffices to show that $a_{1}(z)=a_{1}(\gamma_{1}-z)$. Indeed
\begin{align*}a_{1}(\gamma_{1}-z)-a_{1}(z)&=h(\gamma_{1}-z)-a_{2}(\gamma_{1}-z)-h(z)+a_{2}(z),\\
&=-b(z)+\frac{p(z)+p(\gamma_{2}-z)-p(\gamma_{1}-z)-p(\gamma_{2}-\gamma_{1}+z)}{2},\\
&=-b(z)+\frac{p(z)+p(\omega_{3}+\gamma_{1}-z)-p(\gamma_{1}-z)-p(\omega_{3}+z)}{2},\\
&=-b(z)+\frac{\tilde{b}(z)-\tilde{b}(\gamma_{1}-z)}{2},\\
&=\frac{-b(z)-b(\gamma_{1}-z)}{2}=\frac{-h(z)+h(\gamma_{1}-z)-h(\gamma_{1}-z)+h(z)}{2}=0.\end{align*}

\end{proof}

\bibliographystyle{plain}
\bibliography{biblio}

\end{document}